\documentclass[10pt,letterpaper]{article}
\usepackage{amsfonts}
\usepackage{amsmath}
\usepackage{amsrefs}
\usepackage{amssymb}
\usepackage{amsthm}
\usepackage{array}
\usepackage{bm}	
\usepackage{dcpic,pictexwd}
\usepackage{fancyhdr}
\usepackage{indentfirst}
\usepackage{ifthen}
\usepackage{accents}
\usepackage{graphicx}

\newcommand\hV{{\accentset{\LARGEhat}{V}}}
\newcommand\hW{{\accentset{\LARGEhat}{W}}}
\newcommand\hZ{{\accentset{\LARGEhat}{Z}}}

\newcommand\cV{{\accentset{\LARGEcheck}{V}}}
\newcommand\cW{{\accentset{\LARGEcheck}{W}}}
\newcommand\cZ{{\accentset{\LARGEcheck}{Z}}}

\newcommand\hsigma{\accentset{\Largehat}{\sigma}}
\newcommand\csigma{\accentset{\Largecheck}{\sigma}}
\newcommand\homega{\accentset{\Largehat}{\omega}}
\newcommand\comega{\accentset{\Largecheck}{\omega}}

\numberwithin{equation}{section}
\newtheorem{Theorem}{Theorem}[section]

\newtheorem*{ThmA}{Theorem A}

\newtheorem{Lemma}[Theorem]{Lemma}

\newtheorem{Corollary}[Theorem]{Corollary}
\newtheorem{Definition}[Theorem]{Definition}
\theoremstyle{definition}
\newtheorem{Example}[Theorem]{Example}
\newtheorem{Remark}[Theorem]{Remark}

\newfont{\deffont}{cmbxti10}
\newfont{\german}{eufm10}
\newfont{\mymath}{cmr12}
\newcommand{\CalG}{\mathcal{G}}

\newcommand\rme{{\rm e}}

\newcommand\hook{\mathbin{\raise2.5pt\hbox{\hbox{{\vbox{\hrule height.4pt width6pt depth0pt}}}\vrule height3pt width.4pt depth0pt}\,}}

\newcommand\cTM{T^*\kern-2ptM}

\newcommand\semibasic{\text{\bf  sb}}

\renewcommand\sl{\mathfrak{s}\mathfrak{l} }
\newcommand\vess{\text{\german v\german e\german s\german s}}

\newcommand\thetaX{\theta_{\kern -1 pt X}}

\newcommand\ann{\text{\rm ann}}
\newcommand\CalPf{\mathcal{P}\text{\it \kern -.3pt f}}
\newcommand\Real{\text{\bf  R}}

\newcommand\diag{\text{\rm diag}}
\renewcommand\:{\colon}

\newcommand\TM{T\kern -2pt M}
\newcommand\real{\text{\bf R}}
\newcommand\mycap{\hbox{\ $\rlap{\kern -.3pt $\cap$}\raise.8pt\hbox{$\scriptstyle+$}$\ } }

\newcommand\hxi{\hat \xi}
\newcommand\cxi{\check \xi}

\newcommand\Ao{{\kern-2.3pt}\stackrel{\scriptscriptstyle o}{A}{}{\kern-2.3pt}}
\newcommand\Bo{{\kern-2.3pt}\stackrel{\scriptscriptstyle o}{B}{}{\kern-2.3pt}}
\newcommand\Ko{{\kern-2.3pt}\stackrel{\scriptscriptstyle o}{K}{}{\kern-2.3pt}}
\newcommand\Co{{\kern-2.3pt}\stackrel{\scriptscriptstyle o}{C}{}{\kern-2.3pt}}
\newcommand\Qo{{\kern-2.3pt}\stackrel{\scriptscriptstyle o}{Q}{}{\kern-2.3pt}}
\newcommand\Mo{{\kern-2.3pt}\stackrel{\scriptscriptstyle o}{M}{}{\kern-2.3pt}}

\newcommand\Xo{{\stackrel{\scriptscriptstyle o}{X}}{\kern-1.3pt}}
\newcommand\Yo{{\stackrel{\scriptscriptstyle o}{Y}}{\kern-1.3pt}}

\newcommand\barW{
	\hbox{\kern 2.5 true pt
	\vbox{\hrule width 6  true pt height .3 true pt \kern 1.4 true pt
	\hbox{\kern -2 true pt $W$}}}}

\DeclareMathOperator{\rank}{rank}

\DeclareMathOperator{\spn}{span}
\newboolean{proofmode}
\newcommand{\StTag}[1]{ \label{st:#1}
\ifthenelse{\boolean{proofmode}}{\ \marginpar{\quad\scriptsize st:#1} }{}      }
\newcommand{\EqTag}[1]{
\ifthenelse{\boolean{proofmode}}
{ {\label{eq:#1}}
  \stepcounter{equation}
  \tag{\theequation \rlap{\kern 23 pt{\scriptsize eq:#1}}}
}
{\label{eq:#1}}
 }
\newcommand{\EqRef}[1]{\eqref{eq:#1}}
\newcommand{\StRef}[1]{\ref{st:#1}}
\newcolumntype{C}{>\scriptstyle>{$}c <{$} }
\newcolumntype{L}{>\scriptstyle >{$} l <{$} }
\newcommand\CalA{\mathcal{A}}
\newcommand\CalB{\mathcal{B}}

\newcommand\CalE{\mathcal{E}}
\newcommand\CalI{\mathcal{I}}
\newcommand\CalL{\mathcal{L}}

\newcommand\CalK{\mathcal{K}}

\newcommand\CalH{\mathcal{H}}

\newcommand\CalS{\mathcal{S}}

\newcommand\sd{\mathbin{ \raise0.0pt\hbox{ \vrule height5pt width.4pt depth0pt}\!\times}}
\newcommand\barC{
	\hbox{\kern 2.3 true pt
	\vbox{\hrule width 6.5  true pt height .3 true pt \kern .9 true pt
	\hbox{\kern -0.8 true pt $C$}}}}
\newcommand\barM{
	\hbox{\kern 2.3 true pt
	\vbox{\hrule width 8.5  true pt height .3 true pt \kern .7 true pt
	\hbox{\kern -2.3 true pt $M$}}}}

\newcommand\barK{
	\hbox{\kern 2.5 true pt
	\vbox{\hrule width 6  true pt height .3 true pt \kern 1.4 true pt
	\hbox{\kern -2 true pt $K$}}}}

\newcommand\barU{
	\hbox{\kern .8 true pt
	\vbox{\hrule width 6.5  true pt height .3 true pt \kern .9 true pt
	\hbox{\kern -.8 true pt $U$}}}}
\newcommand\barCalI{
	\hbox{\kern 4.3 true pt
	\vbox{\hrule width 6.5  true pt height .3 true pt \kern .9 true pt
	\hbox{ \kern -4.3 true pt  $\CalI$}}}}
\newcommand\barXi{
	\hbox{\kern 1 true pt
	\vbox{\hrule width 6.5  true pt height .3 true pt \kern .9 true pt
	\hbox{\kern 1 true pt $\Xi$}}}}

\newcommand\Largehat{\smash{\raise -7.5 pt \hbox{\rm\Large\^{}}}}
\newcommand\LARGEhat{\smash{\raise -9.5 pt \hbox{\rm\LARGE\^{}}}}
\newcommand\hugehat{\smash{\raise -7.5 pt \hbox{\rm\huge\^{}}}}
\newcommand\Hugehat{\smash{\raise -7.5 pt \hbox{\rm\Huge\^{}}}}
\newcommand\Largecheck{\smash{\raise -7.5 pt \hbox{\rm\Large\v{}}}}
\newcommand\LARGEcheck{\smash{\raise -9 pt \hbox{\rm\LARGE\v{}}}}
\newcommand\hugecheck{\smash{\raise -7.5 pt \hbox{\rm\huge\v{}}}}
\newcommand\Hugecheck{\smash{\raise -7.5 pt \hbox{\rm\Huge\v{}}}}

\newcommand\hF{{\accentset{\LARGEhat}{F}}}

\newcommand\hI{{\accentset{\LARGEhat}{I}}}
\newcommand\hJ{{\accentset{\LARGEhat}{J}}}

\newcommand{\hpi}{\accentset{\Largehat}{\pi}}

\newcommand\cF{{\accentset{\LARGEcheck}{F}}}

\newcommand\cI{{\accentset{\LARGEcheck}{I}}}
\newcommand\cJ{{\accentset{\LARGEcheck}{J}}}

\newcommand\cpi{\accentset{\Largecheck}{\pi}}

\newcommand\bfq{\mathbf{q}}
\newcommand\bfp{\mathbf{p}}
\newcommand\tbfp{\mathbf{p}}
\newcommand\tbfpi{{{\pi}}}

\newcommand\bfGamma{\boldsymbol{\Gamma}}

\newcommand\bfthetaX{{\boldsymbol{\theta_{\kern -1 pt X}}}}

\newcommand\hf{{{\hat f}}}

\newcommand\cf{{{\check f}}}

\newcommand\Rtheta{{ \raise 1pt \hbox{$\scriptstyle {\boldsymbol{\theta}}$}}}
\newcommand\Ltheta{{ \lower 1pt \hbox{$\scriptstyle {\boldsymbol{\theta}}$}}}

\newcommand\Rsigma{{ \raise 1pt \hbox{$\scriptstyle \sigma$}}}

\newcommand\Reta{{ \raise 1pt \hbox{$\scriptstyle \eta$}}}

\newcommand\vecthsigma{\partial_{{\displaystyle \hat {\raise 1.3pt \hbox{$\scriptstyle \sigma$}}}^a}}
\newcommand\vectcsigma{\partial_{{\displaystyle \check {\raise 1.3pt \hbox{$\scriptstyle \sigma$}}}^\alpha}}

\setboolean{proofmode}{false}

\def\Obj(#1, #2)[#3]#4{\obj(#1, #2)[#3]{#4}}
\def\beginDC#1[#2]{\begindc{#1}[#2]}

\begin{document}

\title{  B\"acklund Transformations for  
\\
	Darboux Integrable Differential Systems:  
\\	
	Examples and Applications
} 
\author{ Ian M. Anderson \\ Dept of Math. and Stat. \\ Utah State University  \and Mark E. Fels  \\Dept of Math. and Stat. \\ Utah State University }
\maketitle
\newpage
\setcounter{page}{1}
\section{Introduction}

	This article is a continuation of our paper {\it  B\"acklund Transformations for  
	Darboux Integrable Differential Systems} \cite{anderson-fels:2014a}. There 
	we established a general group-theoretical approach to the construction of B\"acklund transformations and  we also showed how this construction can be applied to construct B\"acklund transformation between equations which are Darboux integrable. In this paper we demonstrate this theory with a number of detailed examples and new applications.


We adopt the coordinate-free formulation of B\"acklund transformations provided by the differential-geometric setting of exterior differential systems (EDS)  \cite{Clelland-Ivey:2009a}.
	From this viewpoint,  two  differential systems $\CalI_1 \subset \Omega^*(N_1) $ and $\CalI_2  \subset \Omega^*(N_2) $, 
	defined on manifolds $N_1$ and $N_2$, are said to be related by a B\"acklund transformation
	if there exists a  differential system $\CalB  \subset \Omega^*(N) $ on a manifold $N$  and maps

\begin{equation}
\begin{gathered}
\begindc{\commdiag}[3]
\obj(0, 14)[B]{$(\CalB, N)$}
\obj(-12, 0)[I1]{$(\CalI_1, N_1)$}
\obj(12, 0)[I2]{$(\CalI_2, N_2)$\,}
\mor{B}{I1}{$\bfp_1$}[\atright, \solidarrow]
\mor{B}{I2}{$\bfp_2$}[\atleft, \solidarrow]
\enddc
\end{gathered}
\EqTag{BackDef}
\end{equation}
which  define  $\CalB$  as  integrable extensions for both  $\CalI_1$ and $\CalI_2$.   The main result from \cite{anderson-fels:2014a} on the  construction of B\"acklund transformations by  symmetry group reduction is the following theorem.

\begin{ThmA}
\StTag{ThIntro1}
	Let $\CalI$ be a differential system on a manifold $M$ with Lie symmetry groups $G_1$ and $G_2$. 
	Let $H$ be a  common subgroup of  $G_1$ and $G_2$ and assume that the actions of $G_1$, $G_2$ and $H$ are all
	regular on $M$.  Then the orbit  projection maps $\bfp_1: M/H \to M/G_1$ and $\bfp_2: M/H \to M/G_2$ are smooth
	surjective submersions  and
\begin{equation}
\begin{gathered}
\begindc{\commdiag}[3]
\obj(0, 32)[I]{$(\CalI,M)$}
\obj(0, 12)[H]{$(\CalI /H,M/H)$}
\obj(-30, 0)[I1]{$(\CalI /G_1,M/G_1)$}
\obj(30,  0)[I2]{$(\CalI /G_2,M/G_2)$}
\mor{I}{H}{$\bfq_H$}[\atleft, \solidarrow]
\mor{H}{I1}{$\bfp_1$}[\atleft, \solidarrow]
\mor{H}{I2}{$\bfp_2$}[\atright, \solidarrow]
\mor{I}{I1}{$\bfq_{G_1}$}[\atright, \solidarrow]
\mor{I}{I2}{$\bfq_{G_2}$}[\atleft, \solidarrow]
\enddc
\end{gathered}
\EqTag{Intro21}
\end{equation}
is a commutative diagram of EDS. Furthermore, if  the actions of $G_1 $ and $G_2$ are transverse to $\CalI$, then the  maps  in \EqRef{Intro21} are all  integrable extensions and  the diagram
\begin{equation}
\begin{gathered}
\begindc{\commdiag}[3]
\obj(0, 20)[H]{$(\CalB=\CalI /H,\, N=M/H)$}
\obj(-35, 0)[I1]{$(\CalI_1=\CalI /G_1,\, N_1=M/G_1)$}
\obj(35,  0)[I2]{$(\CalI_2=\CalI /G_2,\, N_2 = M/G_2)$}
\mor(-2,17)(-25,3){$\bfp_1$}[\atleft, \solidarrow]
\mor(2,17)(25,3){$\bfp_2$}[\atright, \solidarrow]
\enddc
\end{gathered}
\EqTag{Intro222}
\end{equation}
	defines  $\CalB=\CalI/H$ as  a B\"acklund transformation between $\CalI_1=\CalI/G_1$ and  $\CalI_2=\CalI/G_2$.
\end{ThmA}


This paper is organized as follows. Section \StRef{prelim} provides a basic review of the definitions and theory required to construct diagram \EqRef{Intro21}. In Section \StRef{Examples} we then use Theorem A and the methods from Section \StRef{prelim} to give explicit examples of B\"acklund transformations in local coordinates. In each example we shall begin with a differential system $\CalI$  given as a direct sum $\CalK_1 +\CalK_2$ on a product manifold $M_1\times M_2$.  
We define a  group action $H$ acting diagonally on $M_1\times M_2$ which is a symmetry group of $\CalI$ and 
	which acts transversely to $\CalI$. We then calculate the reduced differential system $\CalB = \CalI/H$ in local coordinates.
	We then pick two more  Lie symmetry  groups  $G_1$ and $G_2$  of  $\CalI$, with  $H\subset G_1 \cap G_2$, and 
	calculate  the reduced differential systems $\CalI_1 = \CalI/G_1$ and $\CalI_2 = \CalI/G_2$. 
	The  orbit projection maps $\bfp_i:(\CalB,N) \to (\CalI_i,N_i)$, which define the sought-after B\"acklund transformation, are also given in local coordinates.

The examples in Section \StRef{Examples} satisfy a property called Darboux integrability. In Section \StRef{SRDI} we show how Darboux integrable systems 
are constructed using symmetry reduction as well as how to compute the fundamental invariants for these systems using symmetry reduction (Theorem \StRef{Dreduce} and Theorem \StRef{Rdiag}). In Section \StRef{ExP2} we  revisit the examples from Section \StRef{Examples} from the perspective of Darboux integrability. The intermediate integrals are computed using Theorem \StRef{Dreduce}.   The fundamental invariant called the {\deffont Vessiot algebra}  (or Vessiot group) is also determined, using  Theorem \StRef{Rdiag},  for each example.

Section 6 provides some basic theory relating the fundamental invariants of a Darboux integrable system to those of an integrable extension.  Section 7 combines the theory from Section 6 and the geometry of the double fibration for a B\"acklund transformation to show that {\it every B\"acklund transformation  constructed  in  \cite{Clelland-Ivey:2009a} and \cite{zvyagin:1991a}	between hyperbolic Monge-Amp\`ere systems and the hyperbolic Monge-Amp\`ere system for the wave equation arises by symmetry reduction.}  For B\"acklund transformation of this type, Theorem \StRef{NoTom} shows the Vessiot algebra for $\CalB$ is isomorphic to a 2-dimensional subalgebra of the
Vessiot algebra of $\CalI_2$. We give an explicit example of a Darboux integrable Monge-Amp\'ere equation with $\mathfrak{s} \mathfrak{o}(3,\real)$ Vessiot algebra. Because $\mathfrak{s} \mathfrak{o}(3,\real)$ has no 2-dimensional subalgebras, Theorem \StRef{NoTom} implies that $\CalI_2$ can not admit such a B\"acklund transformation to the wave equation.

The proof of Theorem \StRef{UMA} uses a novel condition, called {\deffont maximal compatibility} (and defined in Section \StRef{IEDI}), which insures that an integrable extension between Darboux integrable systems arises by symmetry reduction. It is our belief that this condition can be generalized to other settings which will broaden the scope of Theorem A as a mechanism underlying other B\"acklund transformations.

The calculations for this paper were performed using the Maple DifferentialGeometry package. 
The current release of the DifferentialGeometry package, as well as worksheets which support our results, can 
be downloaded from ${\bf {\rm http\!:\!/\!/digitalcommons.usu.edu/dg/}}\, $.



\section{Preliminaries}	\StTag{prelim}

In this section we  gather together a number of definitions and basic results on integrable extensions and reductions of  exterior differential systems. These are needed to construct diagram \EqRef{Intro21} and hence B\"acklund transformations by Theorem A. 	Conventions and some additional results from \cite{anderson-fels:2005a} are used here,	see also \cite{anderson-fels:2013a}, \cite{anderson-fels:2014a}.

\subsection{Extensions and Reductions of Exterior Differential Systems}\StTag{EREDS}
		An exterior differential system (EDS) $\CalI$ on a manifold $M$  is an ideal $\CalI\subset \Omega^*(M)$ which
		is closed with respect to exterior differentiation (see \cite{bryant-chern-gardner-griffiths-goldschmidt:1991a}). We assume that an EDS $\CalI$ has constant rank in the sense that each one  
	of its homogeneous components $\CalI^p \subset \Omega^p(M)$  coincides with the sections 
	$\CalS(I^p)$ of a  constant rank subbundle $I^p \subset \Lambda^p(M)$. If $\CalA$ is a subset of $\Omega^*(M)$, 
	we let  $\langle \CalA \rangle_{\text{alg}} $ and $\langle \CalA \rangle_{\text{\rm diff}} $ denote the algebraic and differential ideals  generated by $\CalA$. As usual, a differential system $\CalI$ is  called a {\deffont Pfaffian system} if there is  a constant rank subbundle $I\subset   \Lambda^1(M)$  
	such that $\CalI = \langle \CalS(I) \rangle_\text{diff}$.  As is customary, the subbundle $I$ shall also  be referred to as a  Pfaffian system.

An EDS  $\CalE$ on a manifold $N$	is called an {\deffont integrable extension}  of  the EDS $\CalI$ on a manifold $M$ if there exists a submersion $\bfp:N \to M$ and a subbundle  $J \subset  \Lambda^1(N)$, where $\ann(J)$ is transverse to the fibres of $\bfp:N \to M$, such that 
\begin{equation}
	\CalE = \langle \CalS(J) \cup  \bfp^*(\CalI) \rangle_\text{alg}.
\EqTag{IntExtDef}
\end{equation}
Here $\ann(J)$ is the subbundle of vectors in $TN$ which annihilate the 1-forms in $J$. A subbundle $J$  satisfying these conditions is called an {\deffont admissible} subbundle for the integrable extension $\CalE$. In particular, $\CalE$ is an integrable extension of $\CalI$ if there exists $\dim N -\dim M$ 1-forms $ \{\, \xi^u \,\}$  on $N$
	which define a local basis for $\CalS(J)$ and satisfy
\begin{equation}
	d\, \xi^u  \equiv 0  \mod \{\bfp^*(\CalI) ,\, \xi^u \} ,
\EqTag{IntExt4}
\end{equation}
so that, locally,
\begin{equation}
\CalE = \langle\ \xi^u ,\,  \bfp^* \CalI \ \rangle_{\text{alg}}.
\EqTag{CheckIE}
\end{equation}
In the special case where $\CalE$ is a Pfaffian system, then $E=\spn\{\, \xi^u \, \} + \bfp^* I^1 $  and
\begin{equation}
\xi^u \in E',
\EqTag{IEP}
\end{equation}
where $E'$ is the derived system of $E$. For additional details see \cite{bryant-griffiths:1995a} and Section 2.1 of \cite{anderson-fels:2014a}.

	Given  a smooth submersion $\bfp :N \to M$  and  an EDS $\CalI$ on $N$,  
	then the {\deffont reduction of $\CalI$} with respect to $\bfp$, denoted by $\CalI/\bfp$, is the EDS on $M$ defined by
\begin{equation}
	\CalI/\bfp = \{ \ \theta \in \Omega^*(M) \ | \ \bfp^* \theta \in \CalI \ \}.
\EqTag{Imodp}
\end{equation}
For example, if $\bfp :(\CalE,N) \to (\CalI,M)$ is an integrable extension, then $\CalE/\bfp =\CalI$. If the fibres of $\bfp$ are connected, then the reduction $\CalI/\bfp$ can be computed using Corollary II.2.3 of \cite{bryant-chern-gardner-griffiths-goldschmidt:1991a}.

	We now specialize the general reduction \EqRef{Imodp} to the case of reduction by a Lie symmetry group of $\CalI$. Let $G$ be a finite dimensional Lie group acting on $M$ with left action $\mu:G \times M \to M$.
We denote by  $\Gamma_G$ the Lie algebra (of vector fields) of infinitesimal generators for the action of $G$, and we let $\bfGamma_G \subset TM$ be the integrable distribution generated by the point-wise span of $\Gamma_G$. 
	We assume that all actions are  {\deffont regular} in the sense that the orbit space $M/G$ has a smooth manifold structure for which the  canonical projection $\bfq_G \: M \to M/G$ is a smooth submersion.  With this assumption we have $\bfGamma_G = \ker (\bfq_{G*})$.  
		
	A group $G$ acting on $M$ is a {\deffont symmetry group of an EDS $\CalI$} if,
	for each $g\in G$ and $\theta \in \CalI$,   $\mu_g^*(\theta) \in \CalI$, where $\mu_g(x) = \mu(g,x)$ for $x \in M$. Under these circumstances and with the hypothesis that
	the action is regular, we define the  {\deffont symmetry reduction of $\CalI$  by G} to be
\begin{equation}
	\CalI/G = \CalI/\bfq_G=  \{\ \bar \theta \in \Omega^*(M/G) \ | \ \bfq_G^*(\bar \theta ) \in \CalI  \ \}.
\EqTag{IGdef}
\end{equation}
 	In other words, $\CalI/G$ is the reduction given by equation \EqRef{Imodp} with respect to the submersion $\bfq_G:M\to M/G$. 
	Utilizing the $G$-invariance of $\CalI$, the determination of local generators for $\CalI/G$ is now a purely linear algebraic computation. See Section \StRef{LGE} below or \cite{anderson-fels:2005a}.


	
	A symmetry group $G$ of an EDS $\CalI$ is said to be {\deffont transverse} to $\CalI$ if
\begin{equation}
	 \ann(I^1) \cap \bfGamma_G = 0.
\EqTag{Itrans}
\end{equation}
	This transversality condition holds for all the examples we consider.  
The importance of this condition is reflected in the fact that {\it if $G$ is transverse to $\CalI$, then $\bfq_G:(\CalI,M)\to ( \CalI/G,M/G)$ is an integrable extension} (see Theorem 2.1 in \cite{anderson-fels:2014a}). 
The hypothesis in Theorem A for constructing B\"acklund transformations requires that the actions of the groups $G_1$ and $G_2$ satisfy the transversality condition \EqRef{Itrans}. Theorem 3.4 in \cite{anderson-fels:2014a} then implies that all maps in diagram \EqRef{Intro21} are integrable extensions. 

As a cautionary remark, we observe that if $\CalI$ is a $G$-invariant Pfaffian differential system,then it is generally not true that $\CalI/G$ is a Pfaffian system. However, if $\CalI$ is a Pfaffian system and $G$ acts transversely to the derived system $I'$ then $\CalI/G$ is a constant rank Pfaffian system \cite{anderson-fels:2005a}.



\subsubsection{Constructing Local Generators for $\CalI/\Gamma_G$}\StTag{LGE}

A Lie algebra of vector fields $\Gamma_G$ on a  manifold $M$ is a {\deffont symmetry algebra} of an EDS $\CalI$ if $\CalL_X \CalI \subset \CalI$ for all vector fields $X\in \Gamma_G$  where $\CalL_X$ is the Lie derivative along $X$. The Lie algebra $\Gamma_G$ is said to act regularly if the canonical projection  $\bfq_{\Gamma_G} : M \to M/\bfGamma_G$ to the leaf space $M/\bfGamma_G$ (of maximal integral manifolds of $\bfGamma_G$) is a smooth submersion. We shall write $\CalI/\Gamma_G$ for the reduction of $\CalI$ in this case. If $G$ is a symmetry group of $\CalI$ and $\Gamma_G$ are the infinitesimal generators, then $\Gamma_G$ is a symmetry algebra of $\CalI$. When the orbits of $G$ are connected, the leaves of $\bfGamma_G$ are identical to the orbits of $G$.

A differential form $\theta \in \Omega^p(M)$  is {\deffont  $\Gamma_G$-basic} if there exists a form $\bar \theta \in \Omega^p(M/\bfGamma_G)$ such that  $\theta=\bfq_{\Gamma_G}^*(\bar \theta ) $. 

In order to compute $\CalI/\Gamma_G$ from the definition in equation \EqRef{Imodp} or \EqRef{IGdef}  we need to determine the $\Gamma_G$-basic forms in $\CalI$.
There are two conditions on $\theta\in \Omega^p(M)$ which determine whether it is $\Gamma_G$-basic. First, 
a differential form $\theta \in \Omega^p(M)$ is {\deffont $\Gamma_G$ semi-basic} if
\begin{equation*}
	X\hook \theta = 0  \quad\text{for all  $X \in \Gamma_G$}.
\end{equation*}
Second, a form $\theta \in \Omega^p(M)$ is {\deffont $\Gamma_G$-invariant} if
\begin{equation*}
	\CalL _X \, \theta =  0
\	\quad\text{for all  $X \in \Gamma_G$}.
\end{equation*}
A form $\theta \in \Omega^p(M)$ is $\Gamma_G$-basic if and only if it is  $\Gamma_G$ semi-basic and $\Gamma_G$-invariant, see Appendix A in \cite{anderson-fels:2005a}.

An $n$-dimensional Lie algebra of vector fields $\Gamma_G$ on an $m$-dimensional manifold $M$ is said to act freely if $\Gamma_G= \spn \{ X_i \}_{1\leq i \leq n}$ and $X_i$ are point-wise linearly independent. We also say that $\Gamma_G$ acts transversely to $\CalI$ if equation \EqRef{Itrans} is satisfied.  We now compute the reduction $\CalI/\Gamma_G$ when $\Gamma_G$ acts freely and transversally to $\CalI$. 

Let  $U$ be an open neighbourhood in $M$ and let $\{ X_i \}_{1\leq i \leq n}$ be a basis for the Lie algebra  $\Gamma_G$ which, by the hypothesis of a free action, are linearly independent at each point of $U$.  We may assume that $U$ has been chosen so that there exists $m-n$ independent functions $ f^1,f^2,\ldots, f^{m-n}$ which satisfy $X_i(f^a) =0$.
The $\Gamma_G$ invariance of the functions $f^a$ allow us, by a slight abuse of
notation, to write $f^a :U/\bfGamma_G \to \real$ and use $(f^a)$ as a set of coordinates on $U/\bfGamma_G$. We may also assume that $U$ has been chosen so that there exists a coframe 
$\{ \tilde \theta^A, \tilde \omega^u \}, 1\leq A\leq K, 1\leq u \leq m-k $ on $U$  such that
\begin{equation}
I^1|_U=\spn \{ \ \tilde \theta^A  \ \}, \quad  1\leq A\leq k.
\EqTag{LGI}
\end{equation}

The transversality condition \EqRef{Itrans} implies that there exists a subset of 1-forms $\{\theta^i\}_{1\leq i \leq n}  $ from those in equation \EqRef{LGI} satisfying
$$
\theta^i(X_j) = \delta^i_j.
$$
This immediately implies that $k \geq n$, that is, $\rank(I^1) \geq \dim \Gamma_G$.  The generators for $\CalI^1|_U$ in equation \EqRef{LGI} can then be replaced by
$$
I^1|_U= \spn\{ \ \theta^i\, , \ \hat \theta^a\ \},\quad 1\leq i \leq n, \ n+1\leq a \leq k.
$$
By the simple linear change of 1-forms,
$$
\theta_{\Gamma_G,\semibasic}^a = \hat \theta^a  - \theta^a(X_i) \theta^i
$$
we have
$$
I^1|_U= \spn\{ \ \theta^i\, , \ \theta_{\Gamma_G,\semibasic}^a \ \}.
$$
The 1-forms $\theta_{\Gamma_G,\semibasic}^a$ are now $\Gamma_G$ semi-basic for the projection $\bfq_{\Gamma_G}:U \to U/\bfGamma_G$. 

Next we construct the $\Gamma_G$ semi-basic 1-forms $\omega^u_{\Gamma_G,\semibasic}= \tilde \omega^u - \tilde \omega^u(X_i) \theta^i$ to obtain the {\deffont $\CalI-\Gamma_G$ adapted coframe} on $U$,
\begin{equation}
T^*U = \spn \{\ \theta^i\, ,\ \theta^a_{\Gamma_G,\semibasic}\, , \ \omega^u_{\Gamma_G,\semibasic}\ \}.
\EqTag{O1basis}
\end{equation}
Using the coframe in equation \EqRef{O1basis}, any $\tilde \tau\in \CalI^2|_U$ may be written
$$
\tilde \tau= T_{uv} \, \omega^u_{\Gamma_G,\semibasic}\wedge \omega^v_{\Gamma_G,\semibasic}\quad \mod \CalI^1|_U.
$$
Then $\tau_{\Gamma_G,\semibasic} = T_{uv}\, \omega^u_{\Gamma_G,\semibasic}\wedge \omega^v_{\Gamma_G,\semibasic} \in \CalI^2|_U$  is $\Gamma_G$ semi-basic. In the case where $\CalI$ is generated algebraically by 1-forms and 2-forms (as in all the examples), this algorithm produces local algebraic generators for $\CalI|_U$,
\begin{equation}
\CalI|_U = \langle \ \theta^i\,,\ \theta^a_{\Gamma_G,\semibasic}\,,\ \tau^\alpha_{\Gamma_G,\semibasic} \ \rangle_{\text {alg}}.
\EqTag{sbg}
\end{equation}
See \cite{anderson-fels:2005a} for the case where $\CalI$ contains generators of higher degree.

We describe two ways to use the generators for $\CalI$ in \EqRef{sbg} to
compute local generators for the reduction $(\CalI/\Gamma_G)_{U/\bfGamma_G}$. The first method (see \cite{anderson-fels:2005a}) is to choose a cross-section $\sigma :U/\bfGamma_G  \to U$, in which case
\begin{equation}
(\CalI/\Gamma_G)|_{U/\bfGamma_G} = \langle\ \sigma^*\theta^a_{\Gamma_G,\semibasic}\, ,\ \sigma^* \tau^\alpha_{\Gamma_G,\semibasic} \ \rangle_{\text {alg}}.
\EqTag{semib}
\end{equation}
This method is demonstrated in equation \EqRef{ExA3} of Example \StRef{ExA1}.


An alternative way to compute $(\CalI/\Gamma_G)|_{U/{\bfGamma_G}}$ is to directly solve the algebraic conditions $\bfq_{\Gamma_G}^* \bar \theta \in \CalI|_U$ 
for $\bar \theta \in (\CalI/\Gamma_G)|_{U/\bfGamma_G}$. This method is particularly simple if a set of generators for $\CalI|_U$ has been found where the semi-basic forms in \EqRef{semib} are in fact $\Gamma_G$-basic. See equations \EqRef{sb21} and \EqRef{sb21b} of Example \StRef{ExA1}.

The proof of Theorem 2.1 in \cite{anderson-fels:2014a} shows how to construct local generators for $\CalI/\Gamma_G$ without the hypothesis that the action of $\Gamma_G$ (or $G$) is free. See also Theorem 2.4 in \cite{anderson-fels:2014a} for 
more detailed information about the structure equations of $\CalI/G$ when the action of $G$ is free.


\section{Examples: B\"acklund Transformations by Quotients}	\StTag{Examples}

In Examples \StRef{ExA1} and \StRef{ExA2}, we show how Theorem A gives the B\"acklund transformations constructed  in \cite{Clelland-Ivey:2009a} and \cite{zvyagin:1991a} between various $f$-Gordon equations. Example \StRef{ExB}  shows that Theorem A can be used to construct B\"acklund transformations for a class of non-Monge-Amp\'ere scalar PDE in the plane. B\"acklund transformations are also constructed using Theorem A for the $A_2$ Toda system in Example  \StRef{ExD}, and  for an over-determined system of PDE  in Example \StRef{ExE}. The final Example \StRef{ExC} contains a positive integer parameter from which a chain of B\"acklund transformations can be constructed.

\medskip

In order to give explicit formulas in the examples we now give local coordinate descriptions of the orbit projection maps $\bfp_a: N \to N_a$  in diagram \EqRef{Intro222} when the actions are given by the Lie algebras $\Gamma_H$ and $\Gamma_{G_a}$. Let $x^i$ be local coordinates on an open set $U \subset M$, let $g_a^\alpha(x^i)$ be an independent set of $\Gamma_{G_a}$ invariants on $U$, and let $h^A(x^i)$ be an independent set of $\Gamma_H$ invariants on $U$. The projection maps $\bfq_H:U\to U/\Gamma_H\subset N$ and $\bfq_{\Gamma_{G_a}}:U\to U/\Gamma_{G_a}\subset N_a$ are then
\begin{equation}
\bfq_{\Gamma_H} = [\ V^A = h^A (x^i) \  ] \ \ {\rm and} \quad \bfq_{\Gamma_{G_a}}= [\ u_a^\alpha = g^\alpha_a(x^i)\ ],
\EqTag{bfqs}
\end{equation}
where $V^A$ and $u^a_\alpha$ are local coordinates on $U/\Gamma_H$ and $U/\Gamma_{G_a}$ respectively.
Since $\Gamma_H \subset \Gamma_G$, every $\Gamma_{G_a}$ invariant $g^\alpha_a(x^i)$ is $\Gamma_H$ invariant, and therefore may be written as
a function of $h^A(x^i)$, that is,
\begin{equation}
g^\alpha_a(x^i)= F_a^\alpha( h^A(x^i)).
\EqTag{FRDI}
\end{equation}
The maps $\bfp_a:U/\Gamma_H \to U/\Gamma_{G_a}$, in the coordinates $u_a^\alpha$ and $V^A$ introduced in equation \EqRef{bfqs}, are then
\begin{equation}
\bfp_a= [ \ u_a^\alpha =  F_a^\alpha ( V^A ) \ ].
\EqTag{pinc}
\end{equation}
Alternatively, if $\sigma:U/\Gamma_H \to U$ is a local cross-section, where $\sigma=[ x^i =\sigma^i (V^A)]$, then the maps $\bfp_a=\bfq_{\Gamma_{G_a}}\circ \sigma$ are obtained from \EqRef{bfqs} as
\begin{equation}
\bfp_a= [ \ u_a^\alpha= g^\alpha_a ( \sigma^i(V^A)) \ ].
\EqTag{csa}
\end{equation}

\subsection{Examples}

\begin{Example}  
\StTag{ExA1}

In this first example we show how Theorem A can be used to obtain the
well-known B\"acklund transformation between Liouville's equation and the wave equation.
All the details of this group theoretic construction are provided. The computations are summarized by diagram \EqRef{ExA51} below.

We start by letting  $\CalI$ be the sum \footnote{\label{fn1}The sum of two EDS is defined in Section \StRef{SRDI}.}  $\CalI=\CalK_1+\CalK_2$, where $\CalK_1$ and $\CalK_2$ are the standard contact systems on the jet spaces $J^2(\real,\real)$. In coordinates $(y,  w, w_y, w_{yy}, x, v, v_x, v_{xx} )$ on $J^2(\real, \real) \times J^2(\real, \real) $ we have 
\begin{equation*}
\CalI=\CalK_1+\CalK_2 = \langle\, \theta_w=dw -w_y\, dy,\,  \theta_{w_y}= dw_y -w_{yy}\,dy, \, \theta_v= dv - v_x\,dx,\,  \theta_{v_x} =dv_x - v_{xx}\, dx\,   \rangle_{\text{diff}}.
\end{equation*}
The total derivative vector fields on $J^2(\real, \real) \times J^2(\real, \real) $ (which lie in the annihilator of $I$) are
\begin{equation}
		D_y =  \partial_y  +w_y\,\partial_w  +w_{yy}\, \partial_{w_y} 
		\quad\text{and}\quad 
	D_x =  \partial_x  +v_x\,\partial_v+v_{xx} \,\partial_{v_x}.
\EqTag{ExA1}
\end{equation}
For the computations that follow we let  $M \subset J^2(\real, \real) \times J^2(\real, \real) $  be the open set where $v + w >0$, $v_x>0$ and $w_y>0$.

To apply Theorem A we start with  the Lie algebra of vector fields $\Gamma_H$ on  $M$ given by
\begin{equation}
\Gamma_H=\spn \{\	X_1 = \partial_w - \partial_v,   \quad   X_2 =   w\,\partial_w + w_y \partial_{w_y} + w_{yy} \partial_{w_{yy}}+v\,\partial_v+v_x \partial_{v_x}+v_{xx} \partial_{v_{xx}}  \ \}.
\EqTag{ExA2}
\end{equation}
The vector fields in $\Gamma_H$ are symmetries of $\CalI$. Indeed, $X_2$ is just the prolongation of a simultaneous scaling in $w$ and $v$. The algebra $\Gamma_H$ acts freely and transversely to the derived system $I'$ on  $M$. Thus $\CalB=\CalI/\Gamma_H$ will be a Pfaffian system \cite{anderson-fels:2005a}. 
	
To find the reduction $\CalI/\Gamma_{H}$, we begin by computing the projection map $\bfq_H:M \to N= M/\Gamma_H$ as in equation \EqRef{bfqs}. The lowest order invariants of $\Gamma_H$ acting on $M$ are 
\begin{equation}
h^1=x, \ h^2=y, \ 	 h^3=\log \dfrac{v_x}{ v + w},\ {\rm and} \quad  h^4= \log \dfrac{w_y}{v+w}.
\EqTag{Hinv}
\end{equation}	
The total vector fields $D_x$ and $D_y$ in \EqRef{ExA1} commute with the vector fields $\Gamma_H$ and therefore
the $D_x$ and $D_y$ derivatives of the invariants in \EqRef{Hinv} produce the higher order $\Gamma_H$ invariants
\begin{equation}
h^5= D_x(h^3) = \frac{v_{xx}}{v_x}-\frac{v_x}{v+w}\ , \quad 
h^6=D_y(h^4) = \frac{w_{yy}}{w_y}-\frac{w_y}{v+w}.
\EqTag{h56}
\end{equation}
By virtue of equation \EqRef{bfqs},
the projection map $\bfq_H \: M \to N$ can now be written by using the invariants in equations \EqRef{Hinv} and \EqRef{h56} and coordinates $(x, y$, $V$, $ V_x$, $W$, $W_y)$ on $N$, as
\begin{equation}
\begin{aligned}
	\bfq_H=
	[\,  x  = x,  \, y = y,\,  V  =\log \frac{v_x}{v+w},\, W = \log \frac{w_y}{v+w},  \, V_x = \frac{v_{xx}}{v_x}-\frac{v_x}{v+w},\,    W_y = \frac{w_{yy}}{w_y}-\frac{w_y}{v+w}\, ].
\end{aligned}
\EqTag{bfpH}
\end{equation}

We now compute the reduced differential system $\CalI/\Gamma_H$ utilizing equation \EqRef{sbg} in the algorithm from Section \StRef{LGE}.  Since $\CalI/\Gamma_H$ is a Pfaffian system we only need to compute the $\Gamma_H$ semi-basic 1-forms in $\CalI$. These are
\begin{equation}
I^1_{\Gamma_H,\semibasic} =\spn \{ \ \theta^1_{\Gamma_H,\semibasic} = \frac{1}{v_x} \theta_{v_x}-\frac{1}{w_y} \theta_{w_y}\ , \quad
 \theta^2_{\Gamma_H,\semibasic}= \frac{1}{v_x} \theta_{v_x}+ \frac{1}{w_y} \theta_{w_y} - \frac{2}{v+w} (\theta_v+ \theta_w)
\ \}.
\EqTag{SM1F}
\end{equation}
The pullback of the 1-forms in \EqRef{SM1F} with the cross-section of $\bfq_H$ in \EqRef{bfpH} given by
\begin{equation}
\begin{aligned}
& \sigma(x,y,V,W,V_x,W_y)=  \\
&\qquad\qquad (y = y,w = 0, w_y = \rme^W, w_{yy} =\rme^W W_y+\rme^{2W}, x = x, v = 1, v_x = \rme^V, v_{xx} =\rme^V V_x +\rme^{2V} ),
\end{aligned}
\EqTag{sigmaex1}
\end{equation}
produces the following generators (see equation \EqRef{semib})
\begin{equation}
\begin{aligned}
	B =  \spn \{\ \beta^1 & = \frac{1}{2}\left(\sigma^* \theta^1_{\Gamma_H,\semibasic} + \sigma^* \theta^2_{\Gamma_H,\semibasic}\right) =dV   -V_x  \,dx + e^W dy  ,\\  \beta^2 & = \frac{1}{2}\left( \sigma^* \theta^2_{\Gamma_H,\semibasic} - \sigma^* \theta^1_{\Gamma_H,\semibasic}\right) = dW  + e^V\,dx - W_ydy  \ \} 
\end{aligned}
\EqTag{ExA3}
\end{equation} 
for the rank 2 Pfaffian system  $B=I/\Gamma_H$. The standard  PDE system which gives rise to the Pfaffian system \EqRef{ExA3} is easily seen to be
\begin{equation}
	V_y = - e^W,\quad  W_x = -e^V.
\EqTag{PDEH1}
\end{equation}

As an alternative method for computing $\CalB$ (which by-passes the differential system formulation), we can compute $D_y(h^3)$ and $D_x(h^4)$ from \EqRef{Hinv} to obtain the following syzygies among the  $\Gamma_H$ invariants on $M$
\begin{equation}
D_y (h^3)= -\frac{w_y}{v+w}= - \rme ^{h^4} , \quad D_x (h^4) = - \frac{v_x}{v+w}  = - \rme^{h^3}.
\EqTag{QDES}
\end{equation}
If we use $V=h^3$ and $W=h^4$ from \EqRef{bfpH}, then equations \EqRef{QDES} agree with the differential equations in \EqRef{PDEH1} arising from the reduced EDS representing the B\"acklund transformation $\CalB$.

In order to construct a B\"acklund transformation from  $\CalB$ we first add  to the  Lie algebra 	$\Gamma_H$  in \EqRef{ExA2} the vector field
\begin{equation}
Z_1 =   \partial_w+ \partial_v,
\EqTag{defZ1}
\end{equation}
and let  $\Gamma_{G_1}= \spn \{ X_1, X_2, Z_1\}=\spn\{ \partial_w, \partial_v, X_2 \} $.

Since $\Gamma_H \subset \Gamma_{G_1}$ we can use the $\Gamma_H$ invariants in equation \EqRef{Hinv} to compute the $\Gamma_{G_1}$ invariants. Equations \EqRef{Hinv} and \EqRef{defZ1} give $Z_1(h^3)=(v+w)^{-2}$ and $Z_1(h^4)=(v+w)^{-2}$ from which we find that the function $g^3_1$ on $M$ defined by
\begin{equation}
g^3_1=h^3-h^4= \log \dfrac{v_x}{ v + w}-  \log \dfrac{w_y}{v+w}=\log {v_x}-\log {w_y}
\EqTag{Fp1}
\end{equation}
is $Z_1$ invariant, and hence $\Gamma_{G_1}$ invariant. Again, $D_x$ and $D_y$ are invariant vector fields with respect to $\Gamma_{G_1}$ so that the other $\Gamma_{G_1}$ invariants on $M$ are 
\begin{equation}
g^1_1=x,g^2_1=y, g^4_1=D_x(g^3_1)=v_{xx}v_x^{-1},\ {\rm  and} \quad g^5_1=D_y(g^3_1)=-w_{yy}w_y^{-1}.
\EqTag{MG1}
\end{equation}
In terms of these $\Gamma_{G_1}$ invariants and coordinates $(x,y,u_1,u_{1x},u_{1y})$ on $N_1$ in equation \EqRef{bfqs}, the map $\bfq_{\Gamma_{G_1}}:M\to N_1$ is 
\begin{equation}
\bfq_{\Gamma_{G_1}}= [  x=x,y=y,u_1=\log {v_x}-\log{w_y},
u_{1x} = \frac{v_{xx}}{v_x}, u_{1y}= -\frac{w_{yy}}{w_y} ].
\EqTag{QG1}
\end{equation}

To complete our coordinate description of the left-hand side of diagram \EqRef{Intro21} we 
need to calculate $\bfp_1$ and $\CalI/\Gamma_{G_1}$. We use equations \EqRef{FRDI} and \EqRef{pinc} to give a coordinate description of $\bfp_1$. Equation \EqRef{Fp1} expresses the $\Gamma_{G_1}$ invariant  $g^3_1$ in terms of the $\Gamma_H$ invariants in \EqRef{Hinv}. We also have from equations \EqRef{MG1}, \EqRef{Hinv} and \EqRef{h56}
\begin{equation}
g^4_1= \frac{v_{xx}}{v_x}= h^5+{\rm e}^{h^4}, \ {\rm and } \quad g^5_1=\frac{-w_{yy}}{w_y} = - h^6 -{\rm e}^{h^4}.
\EqTag{g1h1}
\end{equation}
Equations  \EqRef{Fp1}  and \EqRef{g1h1} along with $h^1=g^1_1=x$ and $h^2=g^2_1=y$, provide the relations in \EqRef{FRDI} so that equation \EqRef{pinc} gives  $\bfp_1:N \to N_1$ as
\begin{equation}
\bfp_1=[\,x=x,\,y=y,\, u_1=V-W,\, u_{1x}= V_x+\rme^V,\, u_{1y}=-W_y-\rme^W\,].
\EqTag{p1map}
\end{equation}
This result can also be obtained from equation \EqRef{csa} by computing
$\bfp_1= \bfq_{G_1} \circ \sigma$ using equations \EqRef{QG1} and \EqRef{sigmaex1}.

To compute the reduced system $\CalI/\Gamma_{G_1}$ we again proceed to find generators for $\CalI$ as in equation \EqRef{sbg}.  Since 
\begin{equation}
I'=  \spn \{ \ \theta^v, \ \theta^w\ \}
\EqTag{Ip}
\end{equation}
$ \Gamma_{G_1}$ is not the transverse to $I'$ and $\CalI_1=\CalI/\Gamma_{G_1}$ is not a Pfaffian system \cite{anderson-fels:2005a}. 

To determine algebraic generators for $\CalI_1$  (see equation \EqRef{sbg})
we therefore need to compute the $\Gamma_{G_1}$ semi-basic 1-forms and 2-forms in $\CalI$. The generators of $\CalI$ in the form \EqRef{sbg} are
\begin{equation}
\begin{aligned}
\CalI =\bigg< \ &  \theta_w,\theta_v,\ \frac{1}{w_y}\theta_{w_y}+\frac{1}{v_x}\theta_{v_x}, \  \theta_{\Gamma_{G_1},\semibasic}=  \frac{1}{v_x}\theta_{v_x}-\frac{1}{w_y} \theta_{w_y}, \\ & \ \tau^1_{\Gamma_{G_1},\semibasic}=
\frac{1}{v_x} d \theta_{v_x}-\frac{v_{xx}}{v_x^2} d \theta_v  , \
\tau^2_{\Gamma_{G_1},\semibasic} =  \frac{1}{w_y} d \theta_{w_y}-\frac{w_{yy}}{w_y^2} d \theta_w  \ \bigg>_{\text {alg}}.
\end{aligned}
\EqTag{sb1}
\end{equation}
Written out in full, the semi-basic forms in \EqRef{sb1} are
\begin{equation}
\begin{aligned}
\theta_{\Gamma_{G_1},\semibasic}& =  \frac{1}{v_x}dv_x-\frac{1}{w_y} dw_y- \frac{v_{xx}}{v_x} dx + \frac{w_{yy}}{w_y} dy\ , \\
\tau^1_{\Gamma_{G_1},\semibasic}&=\left(- \frac{1}{v_x} d v_{xx} + \frac{v_{xx}}{v_x^2} dv_x \right) \wedge dx\ ,  \ 
\tau^2_{\Gamma_{G_1},\semibasic}  =\left(-  \frac{1}{w_y} d {w_{yy}}+\frac{w_{yy}}{w_y^2} dw_y \right)\wedge dy\ .
\end{aligned}
\EqTag{sb21}
\end{equation}
The reduced system $\CalI_1=\CalI/\Gamma_{G_1}$ is then easily computed using the fact that the $\Gamma_{G_1}$ semi-basic forms in \EqRef{sb21} are in fact $\Gamma_{G_1}$ basic. In terms of the $\Gamma_{G_1}$ invariants used in equation \EqRef{QG1} we find that the forms in  \EqRef{sb21} are given by
\begin{equation}
\theta_{\Gamma_{G_1},\semibasic}= \bfq_{G_1}^*(du_1 -u_{1x} dx -u_{1y} dy), \ 
\tau^1_{\Gamma_{G_1},\semibasic}= \bfq_{G_1}^*(-du_{1x} \wedge dx), \ \tau^2_{\Gamma_{G_1},\semibasic}= \bfq_{G_1}^*(du_{1y}\wedge dy).
\EqTag{sb21b}
\end{equation}
This proves that
\begin{equation}
\CalI/\Gamma_{G_1} = \langle \ du_1 -u_{1x} dx -u_{1y} dy, \ du_{1x} \wedge dx,\  du_{1y} \wedge dy\  \rangle_{\text {alg}}.
\EqTag{Imod1}
\end{equation}
Thus $\CalI/\Gamma_{G_1}$ is the Monge-Amp\'ere representation of the wave equation on the 5-dimensional manifold $N_1$. 

The syzygy between the differential invariants of $\Gamma_{G_1}$ on $M$ is
$$
 D_x D_y( g^3_1) = D_x D_y ( \log {v_x}- \log{w_y}) = 0.
$$
This observation, together with the equation $u_1= g^3_1=\log v_x - \log v_y $ (from \EqRef{QG1}), provides an alternative way to identify the quotient $\CalI/\Gamma_{G_1}$ as the wave equation for $u_1$.

Finally, we check directly that $\CalB$ is an integrable extension of $\CalI_1=\CalI/\Gamma_{G_1}$. Using equations \EqRef{Imod1} and \EqRef{p1map} we have
\begin{equation}
\begin{aligned}
\bfp_1^*( du_1 -u_{1x} dx -u_{1y} dy) & = dV-V_x dx - dW  -W_y dy = \beta_1 - \beta_2,  \\ 
\bfp_1^* \ ( du_{1x} \wedge dx )& =d( V_x+ \rme^V) \wedge dx= -d \beta_1+\rme^V \beta_1 \wedge dx+\rme^W\beta_2\wedge dy, \\
\bfp_1^* (-du_{1y} \wedge dy)
&= d(W_y+\rme^W) \wedge dy = -d \beta_2 +\rme^V \beta_1\wedge dx + \rme^W \beta_2\wedge dy.
\end{aligned}
\EqTag{SEpb}
\end{equation}
From these equations and the fact that $\CalB$ is a Pfaffian system, we have 
$\CalB= \langle \, \beta_1 ,\, \beta_2,\,  d\beta_1,\, d\beta_2 \, \rangle_{\text {alg}}=\langle \, \beta_2 ,\,  \bfp^* \CalI_ 1 \, \rangle_{\text {alg}}$. Consequently equation \EqRef{CheckIE} is satisfied and $\CalB$ is an integrable extension of $\CalI_1$.   This completes the construction and verification of the entire left-hand side of diagram \EqRef{Intro21} for this example.

To finish this example we construct the right-hand side of diagram \EqRef{Intro21}  by adding to $\Gamma_H$   the vector field
\begin{equation}
	Z_2 =    
	w^2\partial_w+ D_y(w^2) \partial_{w_y} + D_y^2(w^2)\partial_{w_{yy}}-v^2\partial_v -D_x(v^2) \partial_{v_x} - D_x^2(v^2)\partial_{v_{xx}}.
	\EqTag{defZ2}
\end{equation}
Let $\Gamma_{G_2}=\spn \{ X_1, X_2, Z_2\}$. The action of $\Gamma_{G_2}$ 
	is the diagonal action  on   $J^2(\real, \real) \times J^2(\real, \real)$  of the standard projective action of ${\mathfrak s\mathfrak l}(2,\real)$ 	on each of the dependent variables.

We first compute the quotient of $M$ by the action of $\Gamma_{G_2}$. The lowest 
order invariants of $\Gamma_{G_2}$ on $M$ can be computed by finding 
the $Z_2$ invariants which are functions of the $\Gamma_H$ invariants in \EqRef{Hinv}
(see equation \EqRef{FRDI}). We find
\begin{equation}
g^1_2=x,\quad g^2_2= y, \quad g^3_2= \log \frac{2v_xw_y}{(v+w)^2}
\EqTag{ptou2}
\end{equation}
are $\Gamma_{G_2}$ invariant. The higher order $\Gamma_{G_2}$ invariants are given by the $D_x$ and $D_y$ derivatives of $g^3_2$. 
Using the invariants from \EqRef{ptou2} in equation \EqRef{bfqs}, we find that
 the quotient map  $\bfq_{\Gamma_2}:M \to N_2=M/\Gamma_{G_2}$ is given, in coordinates $(x,y,u_2,u_{2x},u_{2y})$ on $N_2$, by 
 \begin{equation}
\begin{aligned}
\bfq_{\Gamma_{G_2}}= [ & x=x,\, y=y,\,  u_2=g^3_2=\log \frac{2v_xw_y}{(v+w)^2} ,\\
& u_{2x}=D_x (g^3_2)= \frac{v_{xx}}{v_x}-\frac{2v_x}{v+w}, \  u_{2y}=D_y(g^3_2)= \frac{w_{yy}}{w_y}-\frac{2w_y}{v+w} ].
\end{aligned}
\EqTag{QG2}
\end{equation}
The map $\bfp_2: N \to N_2$ is found, using equations \EqRef{QG2}, \EqRef{sigmaex1}, and equation \EqRef{csa}, to be
\begin{equation}
\bfp_2= \bfq_{\Gamma_{G_2}}\circ \sigma  = [\,  x=x,\, y=y,\, u_2=V+W+\log 2 , \,
 u_{2x} = V_x- e^V ,\, u_{2y}= W_y-  e^W\, ].
\EqTag{P2M}
\end{equation}

In order to compute the reduced system $\CalI_2=\CalI/\Gamma_{G_2}$ we note that $\Gamma_{G_2}$ 
is also not transverse to $\CalI'$ and so, according to \EqRef{sbg}, we need to compute the $\Gamma_{G_2}$ semi-basic 1-forms and 2-forms in $\CalI$. These semi-basic forms are
\begin{equation}
\begin{aligned}
\CalI^1_{\Gamma_{G_2},\semibasic} & =\spn\{ \,  \theta_{\Gamma_{G_2},\semibasic}=  \frac{1}{v_x}\theta_{v_x}+\frac{1}{w_y} \theta_{w_y}+\frac{2}{v+w}\left( \theta_v + \theta_w\right)\, \}
 \qquad {\rm and} \\ 
\CalI^2_{\Gamma_{G_2},\semibasic} & = \spn\{\, \tau^1_{\Gamma_{G_2},\semibasic}=\left( \frac{1}{v_x}dv_{xx}+\frac{v_x}{w_y(v+w)} \theta_{w_y}-(\frac{v_{xx}}{v_x^2}+\frac{1}{v+w} )\theta_{v_x}\right) \wedge dx,\\
&\qquad \qquad \tau^2_{\Gamma_{G_2},\semibasic} =\left( \frac{1}{w_y}dw_{yy}+\frac{w_y}{v_x(v+w)} \theta_{v_x}-(\frac{w_{yy}}{w_y^2}+\frac{1}{v+w}) \theta_{w_y}\right) \wedge dy     \, \}.
\end{aligned}
\EqTag{sb2}
\end{equation}
  By using the quotient map \EqRef{QG2} we then find
\begin{equation}
\begin{aligned}
\bfq_{\Gamma_{G,2}}^* ( du_2-u_{2x} dx -u_{2y} dy) & = \theta_{\Gamma_{G_2},\semibasic},\\
\bfq_{\Gamma_{G,2}}^*  \left( ( du_{2x}- \rme^{u_2}dy) \wedge dx\right)  &= \tau^1_{\Gamma_{G_2},\semibasic} -\frac{v_x}{v+w} \theta_{\Gamma_{G_2},\semibasic} \wedge dx ,\\
\bfq_{\Gamma_{G,2}}^*  \left( (du_{2y} -\rme^{u_2} dx) \wedge dy \right) &=  \tau^2_{\Gamma_{G_2},\semibasic} -\frac{w_y}{v+w} \theta_{\Gamma_{G_2},\semibasic} \wedge dy.
\end{aligned}
\EqTag{PL}
\end{equation}
From equation \EqRef{PL} we conclude, using equation \EqRef{Imodp},  that 
\begin{equation}
\CalI_2=\CalI/\Gamma_{G_2} = \langle \, du_2-u_{2x} dx -u_{2y} dy,\, ( du_{2x}- \rme^{u_2}dy) \wedge dx,\, (du_{2y} -\rme^{u_2} dx) \wedge dy \, \rangle_{\text {alg}}.
\EqTag{I2}
\end{equation}
Thus $\CalI/\Gamma_{G_2}$  is the Monge-Amp\`ere system on a 5-manifold for  the Liouville equation $u_{2xy}= \exp(u_2)$.

One can also identify the reduced system $\CalI_2$ by noting the syzygy on the differential invariants of $\Gamma_{G_2}$ on $J^2(\real,\real)\times J^2(\real,\real)$ given by
$$
D_x D_y (g^3_2)= \frac{2v_xw_y}{(v+w)^2}= \rme^{g^3_2}.
$$
With $u_2=g^3_2$ (see equation \EqRef{QG2}) the syzygy becomes the Liouville equation $u_{2xy}=\exp(u_2)$.

Finally, it is easy to check that $\CalB$ is an integral extension of $\CalI_2$. By 
using equation \EqRef{P2M}
and the forms in equations \EqRef{ExA3} and \EqRef{I2}, we find
\begin{equation}
\begin{aligned}
\bfp_2^*(du_2-u_{2x} dx -u_{2y} dy)&=\beta_1 + \beta_2,\\
\bfp_2^*\left( ( du_{2x}- \rme^{u_2}dy) \wedge dx \right)& = -d\beta_1 -\rme^V \beta_1\wedge dx+\rme^W \beta_2\wedge dy, \\
\bfp_2^*\left( (du_{2y} -\rme^{u_2} dx) \wedge dy \right) & =- d \beta_2 + \rme^V \beta_1\wedge dx-\rme^W \beta_2\wedge dy,
\end{aligned}
\end{equation}
and therefore $\CalB = \langle \, \beta_1 ,\, \beta_2,\,  d\beta_1,\, d\beta_2 \, \rangle_{\text {alg}} = \langle\,  \beta_2 , \, \bfp_2^* \CalI_2\, \rangle_{\text{alg}}$. This shows that condition \EqRef{CheckIE} is satisfied and $\CalB$ is an integrable extension of $\CalI_2$. This completes the construction and verification of the right-hand side of diagram \EqRef{Intro21} and  \EqRef{Intro222}.

In summary, we have used the Lie algebra of vector fields $\Gamma_H$, $\Gamma_{G_1}$ and 
$\Gamma_{G_2}$ to explicitly produce the following commutative diagram of differential systems
\begin{equation}
\EqTag{ExA51}
\begin{gathered}
\beginDC{\commdiag}[3]
\Obj(0, 30)[Jet]{$\boxed{J^2(\real,\real) \times J^2(\real,\real)}$}
\Obj(0,12)[B]{$\boxed{\begin{gathered} V_y = - \rme^W,\\  W_x = -\rme^V \end{gathered}\quad (\CalB) }$}
\Obj(-50, 0)[Eq1]{$\boxed{ u_{1xy}= 0 \quad (\CalI_1)}$\ }
\Obj(50, 0)[Eq2]{$\boxed{u_{2xy} = \rme^{u_2} \quad (\CalI_2)}$}
\mor{Jet}{B}{$\bfq_{\Gamma_H}$}[\atright, \solidarrow]
\mor{B}{Eq1}{$\bfp_1$}[\atright, \solidarrow]
\mor{B}{Eq2}{$\bfp_2$}[\atleft, \solidarrow]
\mor(-10,26){Eq1}{$\bfq_{\Gamma_{G_1}}$}[\atright, \solidarrow]
\mor(10,26){Eq2}{$\bfq_{\Gamma_{G_2}}$}[\atleft, \solidarrow]
\enddc
\end{gathered}
\end{equation} 
The maps in this diagram are given in equations \EqRef{bfpH}, \EqRef{QG1}, \EqRef{p1map}, \EqRef{QG2}, and \EqRef{P2M}.  Theorem A insures that all the maps $\bfq_{\Gamma_H}$, $\bfq_{\Gamma_{G_1}}$, $\bfq_{\Gamma_{G_2}}$, $\bfp_1$ and $\bfp_2$ in diagram \EqRef{ExA51} define integrable extensions. In diagram \EqRef{ExA51}, $\CalI_1$ and $\CalI_2$ are
the standard Monge-Amp\`ere system on 5 manifolds corresponding to the given differential equation, and $\CalB$ is 
the standard rank 2 Pfaffian system on a 6 manifold representing the first order pair of PDE. 
	
In equation \EqRef{QDES}  we demonstrated how the differential 
equations for $V$ and $W$ in \EqRef{PDEH1} representing the B\"acklund transformation can be expressed as  syzygies amongst the
differential invariants of $\Gamma_H$ on $M$.  Since $\Gamma_H \subset \Gamma_{G_i}$, the
$\Gamma_{G_1}$ invariant $u_1=g^3_1$ in equation \EqRef{Fp1} and the $\Gamma_{G_2}$ invariant $u_2=g^3_2$ in equation \EqRef{ptou2} are also $\Gamma_H$ invariants.
By computing the syzygies amongst the $\Gamma_H$ differential invariants on $M$ generated by total differentiation of $g^3_1$
and $g^3_2$ we find, from equations \EqRef{QG1} and \EqRef{QG2}, that
\begin{equation}
	u_{1x} - u_{2x} = \sqrt{2} \exp \dfrac{u_1 + u_2}{2}, \qquad u_{1y} +  u_{2y} =  -\sqrt{2}  \exp\dfrac{-u_1 + u_2}{2} .
\EqTag{SBT}
\end{equation} 
This is the usual form of the B\"acklund transformation relating the wave equation to the Liouville equation.  Equation \EqRef{SBT} can also be derived by re-writing the differential equations (or the differential system $\CalB)$  in \EqRef{PDEH1} using the change of variables obtained by solving equations \EqRef{p1map} and \EqRef{P2M}  for $V$ and $W$ in terms of $u_1$ and $u_2$. 
\end{Example}


\begin{Example}  
\StTag{ExA2}
We can construct  other B\"acklund transformations using  Theorem  A  by adding  the vector fields 
\begin{equation}
    Z_3 = y\partial_{w}+ \partial_{w_y} \qquad {\rm or } \qquad   Z_4  = y\partial_w +\partial_{w_y}- x\partial_v-\partial_{v_x}.
\EqTag{defZ3Z4}
\end{equation}
to the Lie algebra $\Gamma_H$ in Example \StRef{ExA1}. 	Let  $\Gamma_{G_i} =  \Gamma_H+\spn \{\,  Z_i \,\}$, $i=3,4$. The Lie algebras $\Gamma_{G_i}$ are all symmetries of the contact system $\CalI=\CalK_1+\CalK_2$ on  $J^2(\real, \real) \times J^2(\real, \real)$.	

The quotients $N_i=M/\Gamma_{G_i}$ are again 5 dimensional. The reductions $\CalI_i=\CalI/\Gamma_{G_i}$ and the projection maps $\bfq_{\Gamma_{G_i}}:M \to N_i$ are computed as in Example \StRef{ExA1} but now using the invariants
\begin{equation}
u_3=g^3_3=\frac{yw_y-v-w}{v_x} +x, \quad u_4= g^3_4= \frac{x v_x+y w_y-v-w}{v_x+w_y}.
\EqTag{u3u4}
\end{equation}
The resulting orbit projection maps $\bfp_i:N \to N_i$ are given below, see equation \EqRef{projeq}.

The systems $\CalI_i$ are all generated by a single 1-form and a pair of decomposible 2-forms and are therefore type  $s=1$  hyperbolic systems \cite{bryant-griffiths-hsu:1995a}.  Equation \EqRef{ExA51} together with \EqRef{u3u4} produce 
the following 4-fold collection of integrable extensions $\bfp_i \: \CalB \to \CalI_i$,
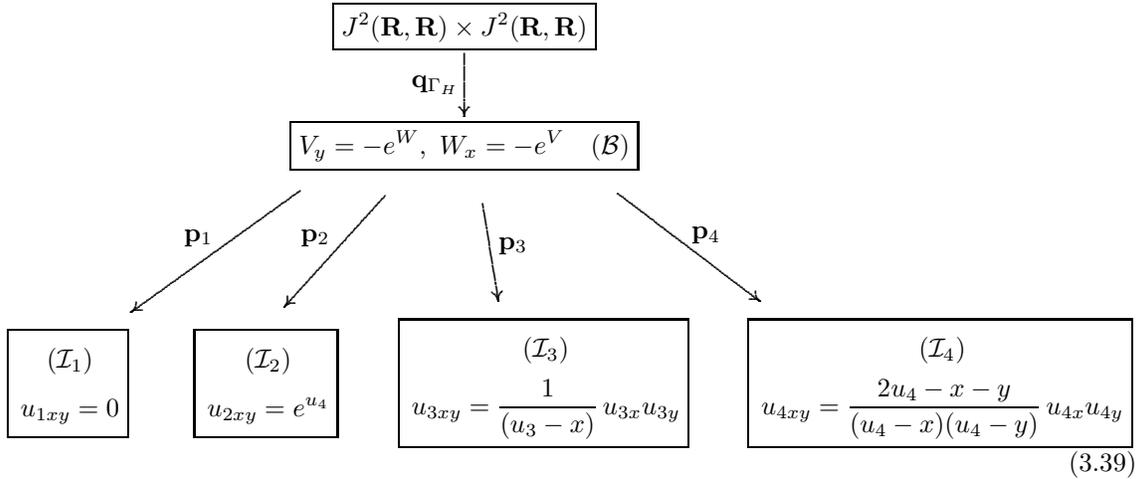
\begin{equation}
\EqTag{ExA5}
\begin{gathered}
\beginDC{\commdiag}[3]
\Obj(-20, 45)[Jet]{$\boxed{J^2(\real,\real) \times J^2(\real,\real)}$}
\Obj(-20, 30)[B]{$\boxed{ V_y = - e^W,\  W_x = -e^V\quad (\CalB)}$}
\Obj(-70, 0)[Eq1]{$\boxed{ \begin{gathered} (\CalI_1)\\ u_{1xy}= 0\end{gathered}}$}
\Obj(-45, 0)[Eq2]{$\boxed{\begin{gathered} (\CalI_2) \\ u_{2xy} = e^{u_4} \end{gathered} }$}
\Obj(-10, 0)[Eq3]{$\boxed{\begin{gathered} (\CalI_3) \\ u_{3xy}= \frac{1}{(u_3-x)} \, u_{3x} u_{3y} \end{gathered}}$}
\Obj(40, 0)[Eq4]{$\boxed{\begin{gathered} (\CalI_4) \\ u_{4xy}= \frac{2u_4 -x-y}{(u_4-x)(u_4-y)} \, u_{4x} u_{4y} \end{gathered}}$}
\mor{Jet}{B}{$\bfq_{\Gamma_H}$}[\atright, \solidarrow]
\mor(-38,26)(-65,7){$\bfp_1$}[\atright, \solidarrow]
\mor(-28,26)(-45,7){$\bfp_2$}[\atright, \solidarrow]
\mor(-18,26)(-15,7){$\bfp_3$}[\atleft, \solidarrow]
\mor(-3,26)(20,8){$\bfp_4$}[\atleft, \solidarrow]
\enddc
\end{gathered}
\end{equation} 
The orbit projection maps $\bfp_i$ are determined from
\begin{equation}
	u_1 = V - W,
\quad
	u_2 = V + W + \log 2 , 
\quad
	u_3 =  y e^{W-V}-e^{-V}+x, 	 
\quad
	u_4 =  \frac{ye^W + xe^V -1}{e^W +e^V},
\EqTag{projeq}
\end{equation}
and their  total derivatives.   All the differential systems $\CalI_i$, $i=1,\ldots, 4$ are the standard Monge-Amp\`ere systems on a 5-manifold for the given second order PDE in the plane. As with $u_1$ and $u_2$, the differential equations in diagram \EqRef{ExA5} for $u_3$ and $u_4$ can be determined as the syzygies for the differential invariants on $M$ in \EqRef{u3u4}.  At this point we would like to emphasize the fact 
that in this example.
{\it each of the projections  $\bfp_i$ defines $\CalB$ as an integrable extension of $\CalI_i$ so that any two of the equations in \EqRef{ExA5} are B\"acklund equivalent via the common Pfaffian system $\CalB$.}

If  we express the variables $V$ and $W$ (and their derivatives) in terms of the variables $u_3$ and $u_4$ from equation  \EqRef{projeq} we find the B\"acklund transformation relating the third and fourth equations in \EqRef{ExA5} can be written as
\begin{equation}
	u_4(u_4 -x)u_{3x} -u_3(u_3-x)u_{4x} = 0, 
	\quad 
	u_4(u_4 -y)u_{3y} +  y u_3  u_{4y} = 0. 
\end{equation}


\end{Example}

\begin{Example} 
\StTag{ExB}  In this example we use Theorem A to construct B\"acklund transformations between some hyperbolic PDE in the plane, which are {\it not} of Monge-Amp\`ere type, and the wave equation.  This example is based upon results found in  the PhD thesis of  Francesco Strazzullo \cite{strazzullo:2009a}.

We start with the canonical Pfaffian systems $\CalK_1$ and $\CalK_2$ for the Monge equations
\begin{equation}
	\frac{d u}{d s} = F(\frac{d^2 v}{d \,s^2} ) \quad\text{and}\quad  \frac{d y}{d t}  = G(\frac{d^2 z}{d\, t^2} ),  \quad  \text{where} \ F''  \neq 0\ \text{and}\ G'' \neq 0.
\EqTag{MA2}
\end{equation}
On the 5-manifolds  $M_1=\real^5[s, u, v, v_s,v _{ss}]$ and $M_2=\real^5[ t, y, z, z_t, z_{tt}]$ these Pfaffian systems are
\begin{align*}
	\CalK_1 &= \langle dv - v_s\,ds,\, dv_s - v_{ss}\, ds, \, du -F(v_{ss})\, ds\,  \rangle_{\text{diff}}
	\quad\text{and}\quad
\\
	\CalK_2 &= \langle dz -z_t\, dt,\,  dz_t -z_{tt}\,dt, \,  dy- G(z_{tt})\,dt  \, \rangle_{\text{diff}}.
\end{align*}
The conditions $F'' \neq 0 $ and $G'' \neq 0$ imply that the derived flags for $\CalK_1$ and $\CalK_2$  have ranks $[3,2,0]$.

	We shall calculate reductions of $\CalI = \CalK_1 + \CalK_2$ on $M=M_1\times M_2$ (see footnote 1 on page 7) using the symmetry algebras
	$\Gamma_{G_1} =\spn \{\,X_1, Z_1,  Z_2\}$,  
	$\Gamma_{G_2} = \spn \{\,Z_1, Z_2,  Z_3\}$ and  $\Gamma_H =  \Gamma_{G_1} \cap \Gamma_{G_2} =\spn \{\,Z_1, \, Z_2 \, \}$, where
\begin{equation}
\begin{gathered}
	X_1 = \partial_v, \quad  X_2 = s\partial_v + \partial_{v_s},  \quad X_3 = \partial_u, \quad
        Y_1 = \partial_z, \quad Y_2 = t\partial_z + \partial_{z_t}, \quad Y_3 = \partial_y ,
\\
	Z_1  = X_1  - Y_1, \quad Z_2 = X_2 + Y_2,  \quad Z_3 =   X_3  - Y_3.
	\EqTag{GammaExB}
\end{gathered}
\end{equation}

	Let $N = \real^8[x_1, \ldots,  x_6, x_7, x_8]$, $N_1 = \real^7[x_1, \dots, x_6,x_7]$ and   $N_2 = \real^7[x_1,\ldots, x_5,x_8,x_9]$.
	Then the projection maps   $\bfq_H : M \to N=M/\Gamma_H$,     $\bfq_{G_1} : M \to N_1=M/\Gamma_{G_1}$  and        
	$\bfq_{G_2} : M \to N_2=M/\Gamma_{G_2}$   for the reduction by these 3 actions are (see equation \EqRef{bfqs})
\begin{equation}
\begin{aligned}
\bfq_{\Gamma_H} &=        [ x_1 = t,\  x_2 = s,\  x_3 = z_{tt},\  x_4 = v_{ss},\  x_5 = z_t -  v_s,\  x_6 = u,\  x_7 = y, \ x_8 = z + v -(s+t)v_s] ,
\\
\bfq_{\Gamma_{G_1}} &= [ x_1 = t,\  x_2 = s,\  x_3 = z_{tt},\  x_4 = v_{ss},\ x_5 = z_t -  v_s,\  x_6 = u,\  x_7 = y],
\\
\bfq_{\Gamma_{G_2}} &= [x_1 = t,\  x_2 = s,\  x_3 =  z_{tt},\  x_4 = v_{ss},\ x_5 = z_t - v_s,\ x_8 = z + v -(s+t)v_s, \  x_9 = u + y].
\end{aligned}
\EqTag{bfqExB}
\end{equation} 
The actions $\Gamma_H$ and $\Gamma_{G_i}$ are transverse to $I'$, so that the
differential systems in the commutative diagram \EqRef{Intro21} are all Pfaffian systems \cite{anderson-fels:2005a}. The reductions are easily computed to be	

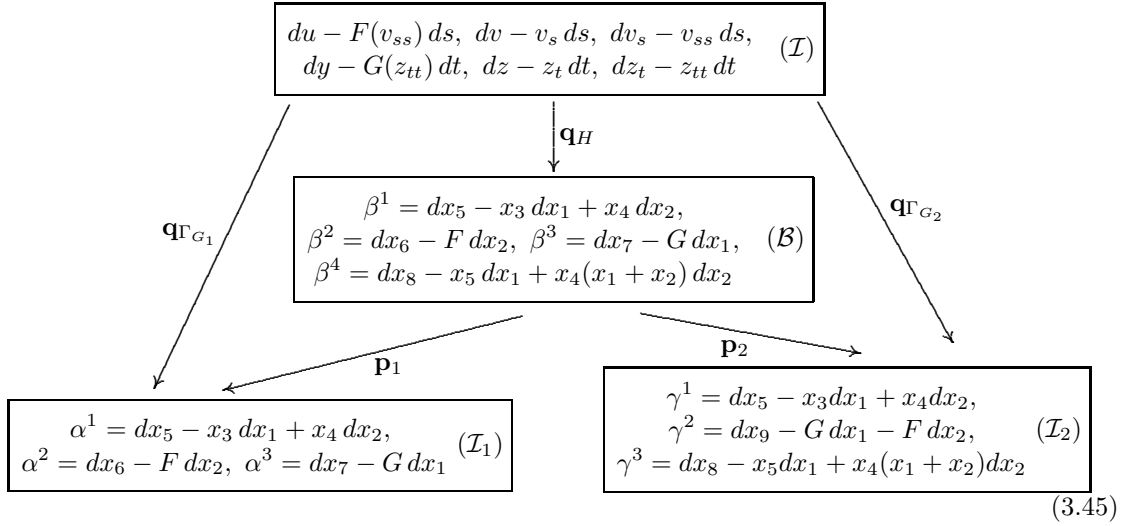
\begin{equation}
\beginDC{\commdiag}[3]
\Obj(0, 50)[I]{$ \boxed{\begin{gathered} du - F(v_{ss})\,ds, \ dv- v_s\,ds, \ dv_s - v_{ss}\,ds , \\[-2\jot] dy -G(z_{tt})\, dt, \ dz - z_t\,dt,\  dz_t -z_{tt}\, dt  \end{gathered}  \quad (\CalI) } $ }
\Obj(0, 26)[H]{$\boxed{ \begin{gathered}  \beta^1 = dx_5 - x_3\,dx_1 + x_4\,dx_2, \\[-2\jot]  \beta^2 = dx_6- F\,dx_2 ,\    \beta^3 =  dx_7 - G\, dx_1 , \\[-2\jot] \beta^4 =  dx_8 - x_5\,dx_1 +x_4(x_1+x_2)\,dx_2     \end{gathered}  \quad (\CalB) } $}
\Obj(-37, 0)[I1]{$\boxed{ \begin{gathered} \alpha^1 = dx_5  -x_3\,dx_1+x_4\,dx_2,\  \\[-2\jot] \alpha^2 = dx_6 -F\,dx_2,\  \alpha^3 = dx_7 -G\,dx_1  \end{gathered}   \  (\CalI_1)  }  $}
\Obj(37, 2)[I2]{$ \boxed{ \begin{gathered} \gamma^1 = dx_5 -x_3dx_1+x_4dx_2,  \\[-2\jot]  \gamma^2 = dx_9-G\,dx_1  -F\,dx_2 , \\[-2\jot]   \gamma^3 = dx_8   -x_5dx_1+ x_4(x_1+x_2)dx_2  \end{gathered}   \  (\CalI_2)   }$}
\mor{I}{H}{$\bfq_H$}[\atleft, \solidarrow]
\mor(-1, 17)(-45, 6){$\bfp_1$}[\atleft, \solidarrow]
\mor(8, 17)(42, 11){$\bfp_2$}[\atright, \solidarrow]
\mor(-32, 46)(-52, 4){$\bfq_{\Gamma_{G_1}}$}[\atright, \solidarrow]
\mor(32, 46)(52, 10){$\bfq_{\Gamma_{G_2}}$}[\atleft, \solidarrow]
\enddc
\EqTag{ExB3}
\end{equation}
where $F= F(x_4)$, $G = G(x_3)$, and $x_9=x_6+x_7$ in $\bfp_2$.

For the Pfaffian system $\CalB =  \CalI/\Gamma_{H}$ in \EqRef{ExB3} the structure equations are
\begin{equation}
\begin{aligned}
	d\,\beta^1 
	&= \hpi^1 \wedge \hpi^2 - \cpi^1 \wedge \cpi^2,
	&\qquad  d\, \beta^2
	&= F'\, \cpi^1 \wedge \cpi^2  &&\!\!\! \mod \{\, \beta^1, \beta^2 \, \},
\\	
	d\, \beta^3 
	&=  G' \,\hpi^1 \wedge \hpi^2, 
	& \qquad d\,\beta^4  
	&= -(x_1+ x_2) \cpi^1 \wedge \cpi^2  &&\!\!\! \mod \{\,  \beta^1, \beta^2 \, \},
\end{aligned}
\EqTag{SEB2}
\end{equation}
where  $\hpi^1=dx_1$, $\hpi^2=dx_3$,   $\cpi^1 = dx_2$, $\cpi^2 = dx_4$.  Thus $\CalB$ is a class $s=4$ hyperbolic Pfaffian system \cite{bryant-griffiths-hsu:1995a}. 
	
It is straightforward to check that the maps $\bfp_a$ are integrable extensions by using condition \EqRef{IEP} and noting that the  derived system for $B$ is
\begin{equation*}
	B' =  \spn \{\, \delta^1=\beta^4 + \frac{x_1 +x_2}{F'}\,\beta^2,\  \delta^2=\beta^3 - \frac{1}{G'} \beta^1 +\frac{1}{F'} \beta^2 \, \}.
\end{equation*}
	The spans of the 1-forms  $\delta^1$ and $\delta^2$ in $B'$ define admissible sub-bundles for $\CalB$  as  integrable extensions of $\CalI/\Gamma_{G_1}$  and 	$\CalI/\Gamma_{G_2}$,   respectively. 

The change of variables
\begin{equation}
\begin{gathered}
	X' =- x_6  + x_2F,   \quad    Y' = x_7 - x_1G,\quad  U'=   x_5  - x_1x_3 + x_2x_4,\   
\\ 	
	U'_{X'}= \frac{1}{F'} ,\quad U'_{Y'} = \frac{1}{G'}, \quad      U'_{X'X'} = -\frac{F''}{x_2F'^3}, \quad  U'_{Y'Y'} = \frac{G''}{x_1G'^3}
\end{gathered}
\EqTag{ExB8}
\end{equation}
transforms the system $\CalI_1 = \CalI/\Gamma_{G_1}$ in diagram \EqRef{ExB3} into the standard rank 3 Pfaffian system for the wave equation $U'_{X'Y'} = 0$. 

	To calculate the structure equations  for $\CalI_2 = \CalI/\Gamma_{G_2}$, given in diagram \EqRef{ExB3}, we define a new co-frame by 
\begin{equation}
\begin{gathered}
	\theta^1  =  (x_1 + x_2) G'\,\gamma^1  - (x_1+x_2)\,\gamma^2 - (F' + G')\, \gamma^3 , \quad
	\theta^2 =  F'\,\gamma^1 + \gamma^2, \quad \theta^3 =  \gamma^2-G'\,\gamma^1,
\\
	\hpi^1  =  dx_1 -  \mu_2(x_1+ x_2)\, dx_3, \quad \hpi^2 =  dx_3,  \quad
	\cpi^1  = dx_2 -  \mu_1(x_1 +x_2) \,dx_4, \quad \cpi^2  =  dx_4,
\end{gathered}
\EqTag{CofEB}
\end{equation}
	where $\mu_1 = \dfrac{F''}{F' +G'}$ and $\mu_2 = \dfrac{G''}{F' +G'}$. The structure equations for this new coframe are 
\begin{equation}
\begin{aligned}
	d\, \theta^1 & \equiv \theta^2 \wedge \hpi^1 + \theta^3 \wedge \cpi^1,  \quad \mod \theta^1
\\
	d\, \theta^2 & = (F' + G') \hpi^1 \wedge\hpi^2 -  \mu_1(\theta^2 -\theta^3) \wedge \cpi^2,  
\\
	d\, \theta^3 &= (F' + G') \cpi^1 \wedge\cpi^2  + \mu_2(\theta^2 - \theta^3) \wedge \hpi^2 . 
\end{aligned}
\EqTag{SeEB}
\end{equation}
	These  structure equations show that $\CalI_2$  defines a  hyperbolic second order PDE, see Theorem 11.1.1 in \cite{stormark:2000a}. From the equations  for 
	$d\, \theta^2$ and  $d\, \theta^3$    one can determine that the  Monge-Amp\`ere invariants for $\CalI_2$ are
	proportional to  $F''$ and $G''$. Thus   $\CalI_2$ is of  generic type  (type (7, 7) in the terminology of \cite{gardner-kamran:1993a}) 
	when $F'' \neq 0$ and $G'' \neq 0$.

	In the special case where \EqRef{MA2} consists of a pair of Hilbert-Cartan equations
\begin{equation}
	\frac{d u}{d s} = \bigl(\frac{d^2 v}{d \,s^2} \bigr)^2\quad\text{and}\quad  \frac{d y}{d t}  = \bigr (\frac{d^2 z}{d\, t^2} \bigl)^2 
	\EqTag{HCeq}
\end{equation}
the change of variables on $N_2$ defined by
\begin{equation}
\begin{gathered}
	X  = -2 \frac{x_3 +x_4}{x_1+x_2},  \quad Y =x_5 - \frac12 (x_3-x_4)(x_1+x_2),\quad   U_{XY} =   \frac12(x_2- x_1), \quad U_{YY}= \frac2{x_1+x_2},
\\
	U  =  2 \frac{x_3 +x_4}{x_1 +x_2} (x_1x_5-x_8) - x_9   +\frac13(2x_1 -x_2)x_4^2
	-\frac23(x_1 +x_2) x_3x_4 -\frac13(x_1 - 2x_2)x_3^2,
\\
	U_X  =  x_8 - x_1x_5 + \frac16 (x_1 +x_2)((2x_1 -x_2)x_3  -  (x_1-2x_2)x_4) ,\quad 
	U_Y  =  2\frac{x_4x_1-x_3x_2}{x_1+x_2}
\end{gathered}
\EqTag{ExB10}
\end{equation}
transforms $\CalI_2$  to the standard differential system for $3U_{XX}{U^3_{YY}}+1=0$.
Accordingly, we can rewrite the B\"acklund transformation \EqRef{ExB3} as
\begin{equation}
\begin{gathered}
\beginDC{\commdiag}[3]
\Obj (1, 27)[B]{$\boxed{\begin{gathered}  dx_5 - x_3\,dx_1 + x_4\,dx_2, \\[-2\jot]  dx_6- x_4^2\,dx_2 ,\    dx_7 - x_3^2\, dx_1,  \\[-2\jot]    dx_8 - x_5\,dx_1 +x_4(x_1+x_2)\,dx_2     \end{gathered}    \quad (\CalB)   } $}
\Obj(-28, 0)[I1]{$\boxed{U'_{X'Y'} = 0\quad (\CalI_1)}$}
\Obj(35, 0)[I2]{$\boxed{3U_{XX}U^3_{YY} +1 =  0 \quad (\CalI_2)}$}
\mor(-2, 17)(-32,3){$\bfp_1$}[\atright, \solidarrow]
\mor(2,17)(38,3){$\bfp_2$}[\atleft, \solidarrow]
\enddc
\end{gathered}
\end{equation}
Here $\CalI_1$ and $\CalI_2$ are the standard rank 3 Pfaffian systems for the given second order PDE in the plane. 

\end{Example}

\begin{Example}  Examples \StRef{ExA1}-\StRef{ExB} provide B\"acklund transformations for scalar PDE in the plane. In this example we shall use Theorem A to find a B\"acklund transformation for the $A_2$ Toda system
\StTag{ExD}
\begin{equation*}
	u_{xy} =  2 e^u - e^v, \quad v_{xy} =  -e^u +2e^v.
\end{equation*}
By a simple linear change of variables,  the $A_2$ Toda system can be rewritten as
\begin{equation}
	U^1_{xy} = e^{2U^1 -U^2}, \quad U^2_{xy} = e^{-U^1 +2U^2}.
\EqTag{95}
\end{equation}
This latter formulation proves to be more amenable to our analysis.

We begin with the following  choice of non-standard coordinates $[y, w, w_1, w_2, w_3,  z, z_1, z_2, z_3, z_4]$ for $J^{3,4}(\real, \real^2)$ for which the contact system is given by
\begin{equation*}
K_1 =\spn \{ dw - w_1dy,\ dw_1 - w_2 dy,\ dw_2 -w_3dy,\ dz - z_1dw,\  dz_1 - z_2dw,\ dz_2 - z_3 dw, \ dz_3 - z_4 dw \}.
\end{equation*}
For our second copy of $J^{3,4}(\real, \real^2)$  we use coordinates $[x$, $u$, $u_1$, $u_2$, $u_3$,  $v$, $v_1$, $v_2$, $v_3$, $v_4]$, 	with 
\begin{equation*}
		K_2 =\spn \{ du - u_1dx,\ du_1 - u_2 dx,\ du_2 -u_3dx,\ dv - v_1du,\  dv_1 - v_2du,\ dv_2 - v_3 du, \ dv_3 - v_4 du \}.
\end{equation*}
In these coordinates the total derivatives $D_x$ and $D_y$ are
\begin{equation}
\begin{aligned}
D_x&= \partial_x+ u_1 \partial_u + u_2 \partial_{u_1} + u_3 \partial _{u_2} + u_1 v_1 \partial_v + u_1 v_2 \partial_{v_1} + u_1 v_3 \partial_{v_2} +u_1 v_4 \partial_{v_3} \quad {\rm and} \\
D_y&= \partial_y+ w_1 \partial_w + w_2 \partial_{w_1} + w_3 \partial _{w_2} + w_1 z_1 \partial_z + w_1 z_2 \partial_{z_1} + w_1 z_3 \partial_{z_2} +w_1 z_4 \partial_{z_3}.
\end{aligned}
\EqTag{DxDyNC}
\end{equation}
	
For the symmetry algebra of $\CalI=\CalK_1+\CalK_2$ we start with the infinitesimal generators $\Gamma_{G_2}$ for the diagonal action of  $G_2=SL(3,\real)$   on $ J^{3,4}(\real, \real^2) \times J^{3,4}(\real, \real^2) $  given by the prolongation of the vector fields
\begin{gather}
\begin{align*}
	X_1 &= \partial_w+\partial_u , &\  X_2 &= \partial_z+ \partial_v,&\  X_3 &=w\,\partial_w+u\,\partial_u ,\qquad\qquad
\\
	X_4 &= z\,\partial_w + v\,\partial_u,
	&\ 
	X_5 &= w\,\partial_z+u\,\partial_v  ,
	&\ 
	X_6 &=z\,\partial_z+ v\,\partial_v ,\qquad
\end{align*}
\\
	Z_1 = w^2\partial_w + wz\,\partial_z+u^2\partial_u  + uv\,\partial_v  ,
	\quad
	Z_2 = wz\,\partial_w + z^2\,\partial_z+uv\,\partial_u + v^2\partial_v .
\EqTag{SL3c}
\end{gather}
On each factor this is the standard projective action of $\sl(3,\real)$ on $\real^2$. 
In order to compute the prolongation the standard Lie prolongation formula for vector fields must be adapted to the above form of the contact systems.

 We now construct  a B\"acklund transformation for \EqRef{95}  using the Lie algebras of vector fields on $J^{3,4}(\real, \real^2) \times J^{3,4}(\real, \real^2)$ given by
\begin{equation}
\begin{aligned}
	\Gamma_{G_1}  &= \spn \{\, X_1, \,X_2,\, X_3,\, X_4,\, X_5,\, X_6,\,   Z_3 =  \partial_u - \partial_w,\  Z_4 = \partial_v - \partial_z \,\},
\\
	\Gamma_{G_2}  & = \spn \{\,X_1,\, X_2,\, X_3,\, X_4,\, X_5,\, X_6,\,  Z_1,\, Z_2 \, \} \ {\rm and}
\\		
	\Gamma_H  &=  \Gamma_{G_1} \cap  \Gamma_{G_2}  = \spn \{\,  X_1,\, X_2,\, X_3,\, X_4,\, X_5,\, X_6 \, \} .
\end{aligned}
\EqTag{GammaSL3}
\end{equation}
We will work 
on the open set  $M \subset J^{3,4}(\real, \real^2) \times J^{3,4}(\real, \real^2)$ where
\begin{equation}
(w-u)z_1>z-v,\, v_1> z_1,\, u_1>0,\, w_1>0, \, v_2 <0,\,  z_2 > 0.
\EqTag{Todadomain}
\end{equation}

It can be shown that the reduction $\CalI_2=(\CalK_1+\CalK_2)/\Gamma_{G_2}$ is a rank 6 Pfaffian system on the 12-manifold $N_2=M/\Gamma_{G_2}$. The system $\CalI_2$ is the standard Pfaffian system obtained from the Toda equations in \EqRef{95}, viewed as a submanifold of $J^2(\real^2,\real^2)$. Indeed, the lowest order joint differential invariants for the $\Gamma_{G_2}$ action are
\begin{equation}
	g^1_2 = \frac13\log \frac{-u_1^3v_2w_1^3  z_2^2}{((w-u)z_1+v-z)^3} \quad \text{and}\quad   
	g_2^2 = \frac13\log\frac{  u_1^3 v_2^2w_1^3 z_2}{ ((w-u)v_1+v-z)^3}.
\EqTag{32}
\end{equation}
The syzygies among the $x$ and $y$ total derivatives (computed using equation \EqRef{DxDyNC}) of these invariants  give precisely  the $A_2$ Toda system \EqRef{95}, after substituting $U^1=g_2^1$ and $U^2=g_2^2$.

The rank 6 Pfaffian system  $\CalI/\Gamma_{G_1}$ is the standard differential system on the 12 manifold $M/\Gamma_{G_1}$ for the  decoupled Liouville-wave system
\begin{equation}
	V^1_{xy}= e^{\textstyle V^1},  \quad V^2_{xy} = 0.
\EqTag{52}
\end{equation}
Equations \EqRef{52} are the syzygies among the $x$ and $y$ total derivatives (equation \EqRef{DxDyNC}) of the invariants
\begin{equation}
	g_1^1 = \log  \frac{-2u_1w_1v_2z_2}{(v_1 - z_1)^2} \quad \text{and} \quad g_1^2 =  \log \frac{-w_1^3z_2}{u_1^3v_2},
\EqTag{VTs}
\end{equation}
after letting $V^1=g^1_1$ and $V^2=g^2_1$.
		
	Finally the symmetry  reduction of $\CalK_1+\CalK_2 $ by the 6-dimensional 
	Lie algebra $\Gamma _H$ is a rank 8 Pfaffian system  $\CalB$ on the 14-dimensional  manifold $M/\Gamma_H$. The Pfaffian system $\CalB$ is found to be a partial prolongation of the standard Pfaffian system for the equations
\begin{equation}
	W^1_{xy}  = -W^1_y \, (W^2_x+ e^{W_1})    \quad\text{and}\quad W^2_{xy}  = -W^2_x\, (e^{W^2}+W^1_y).
\EqTag{53}
\end{equation}
Again this can be seen by using the  lowest order joint differential invariants for $\Gamma_H$ on M, namely
\begin{equation}
	h^1 = \log \frac{u_1(v_1 -z_1)}{(w-u)z_1+v-z} \quad\text{and}\quad  h^2 =  \log  \frac{w_1(v_1-z_1)}{(w -u)v_1 + v -z}.
\EqTag{WTs}
\end{equation}
The syzygies among the  $x$ and $y$ total derivatives of these invariants  produce \EqRef{53}.

	At this point we have constructed the  Pfaffian systems for  all of the equations in the following commutative diagram
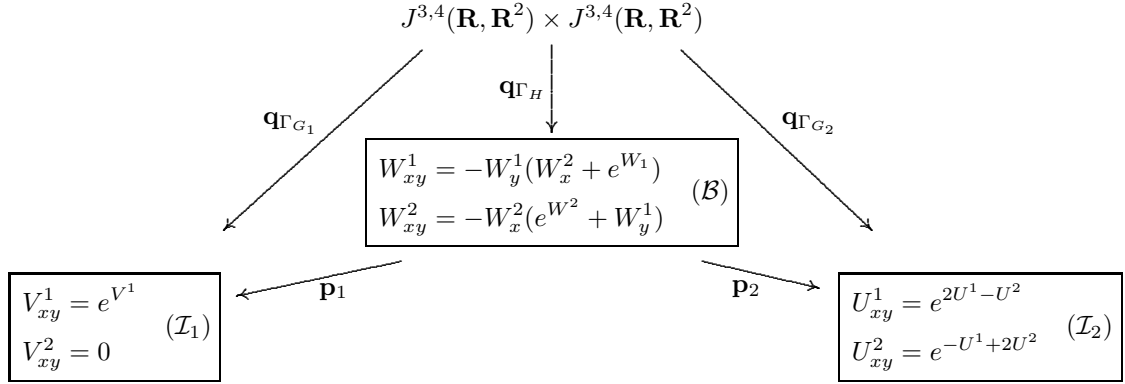
\begin{equation*}
\begin{gathered}
\beginDC{\commdiag}[3]
\Obj(0, 32)[Jet]{$J^{3,4}(\real, \real^2) \times J^{3,4}(\real, \real^2)$}
\Obj(0, 10)[B]{$\boxed{\begin{aligned}W^1_{xy} &= -W^1_y (W^2_x+ e^{W_1})    \\ W^2_{xy}  &= -W^2_x(e^{W^2}+W^1_y) \end{aligned}\quad (\CalB) }$}
\Obj(54,  -7)[A2]{\ $\boxed{\begin{aligned}U^1_{xy} &=  e^{2U^1 -U^2}    \\ U^2_{xy} &=e^{-U^1 +2U^2} \end{aligned}\quad (\CalI_2) }$ }
\Obj(-54, -7)[A1]{$\boxed{\begin{aligned}     V^1_{xy} &= e^{V^1} \\  V^2_{xy} &=0  \end{aligned} \quad (\CalI_1) }$ \ }
\mor{Jet}{B}{{$\bfq_{\Gamma_H}$}}[\atright, \solidarrow]
\mor( -14,30)(-44,3){$\bfq_{\Gamma_{G_1}}$}[\atright, \solidarrow]
\mor(14,30)(43,3){$ \bfq_{\Gamma_{G_2}}$}[\atleft, \solidarrow]
\mor(16,2)(37,-3){$\bfp_2$}[\atright, \solidarrow]
\mor(-16,2)(-43,-4){$\bfp_1$}[\atleft, \solidarrow]
\enddc 
\end{gathered}
\end{equation*}
The quotient maps for the actions $\Gamma_H$, $\Gamma_{G_1}$ and $\Gamma_{G_2}$ are defined using the fomulas  $U^1=g_2^1$, $U^2=g^2_2$, $V^1=g^1_1$, $V^2=g^2_1$ from  \EqRef{32}, \EqRef{VTs}, and \EqRef{WTs}. To give the coordinate form for the maps $\bfp_a$ (as in equation \EqRef{pinc}) we calculate the expressions for $U^1=g_2^1$, $U^2=g_2^2$, $V^1=g^1_1$, $V^2=g^2_1$ in terms of $W^1=h^1$ and $W^2=h^2$ and their derivatives to deduce that
\begin{equation}
\begin{gathered} 
	U^1 = \frac23W^1+ \frac13W^2+ \frac13 \log(-(W^1_y)^2W^2_x), 
	\quad
	U^2 =  \frac13W^1 +\frac23W^2+\frac13 \log(-W^1_y(W^2_x)^2),
\\
	V^1 = \log(2W^1_yW^2_x),   
	\quad
	V^2 =  2W^2-2W^1+\log(\frac{W^1_y}{W^2_x}).
\end{gathered} 
\EqTag{55}
\end{equation}
Note that, on account of \EqRef{Todadomain}, $W^1_y<0$ and $W^2_x <0$. 
	Formulas \EqRef{55}, together with their $x$ and $y$ total derivatives to order 2,  define the projection maps $\bfp_1$ and $\bfp_2$. 
	It is a simple matter to check that $\bfp_1$ and $\bfp_2$ define integrable extensions.	
	
Finally,  the  elimination of the variables $W^1$ and $W^2$ and their derivatives from  \EqRef{55} lead to the following first order system of PDE 
\begin{alignat*}{2}
V^1_x- \frac13V^2_x-2U^2_x &= \sqrt{2}\,e^{(\displaystyle U^2-  {\textstyle \frac16V}^2-U^1+  {\textstyle \frac12}V^1)},
&\quad
V^2_x+3U^1_x &= -3\sqrt{2}\,e^{(\displaystyle -  {\textstyle \frac12}V^1+U^1-   {\textstyle \frac16}V^2)},
\\
V^1_y+ \frac13V^2_y-2U^1_y &= \sqrt{2}\,e^{(\displaystyle U^1-U^2+ {\textstyle \frac12}V^1+  {\textstyle \frac16}V^2)},
&\quad
-V^2_y+3U^2_y &= -3\sqrt{2}\,e^{(\displaystyle U^2 -   {\textstyle \frac12}V^1+   {\textstyle \frac16}V^2)}
\end{alignat*}
 for $U^1$, $U^2$, $V^1$, $V^2$. This gives a B\"acklund transformation for the $A_2$ Toda system \EqRef{95} in the classical sense.
\end{Example}

\begin{Example} \StTag{ExE} In  this example we find a B\"acklund transformation for the over-determined system
\begin{equation}
	u_{xz} = uu_x , \quad u_{yz} = u u_y
	\EqTag{OverDS}
\end{equation}
considered in \cite{anderson-fels-vassiliou:2009a}. We shall obtain the system \EqRef{OverDS} as 
the quotient of the standard contact structure $\CalK_1+\CalK_2$ on the jet space $ J^2(\real,\real)\times J^1(\real^2,\real)$ by the fiber action of $\sl(2,\real)$. With respect to the coordinates $(  z, w,w_z,w_{zz};x,y,v,v_x,v_y)$ the action is $\Gamma_{G_2}=\spn\{X_1,X_2,X_3\}$, where
\begin{equation}
\begin{aligned}
X_1& =\partial_w+\partial_v, \quad X_2 =w\partial_w+w_z\partial_{w_z}+w_{zz}\partial_{w_{zz}}+ v\partial_v+v_x\partial_{v_x}+v_y\partial_{v_y}, \\
X_3&=w^2\partial_w+2ww_z\partial_{w_z}+ 2(w_z^2+ww_{zz})\partial_{w_{zz}}
+v^2 \partial_v+2vv_x\partial_{v_x}+2vv_y\partial_{v_y}.
\end{aligned}
\EqTag{GammaExE}
\end{equation}
The domain $M$ we work on is defined by $v>w, v_x>0,v_y>0,w_z>0, w_{zz} \neq 0$.

From the invariants of $\Gamma_{G_2}$ on $M$, we find that the quotient $\bfq_{G_2}:M \to N_2=M/\Gamma_{G_2}$ (equation \EqRef{bfqs}) is
\begin{equation}
\bfq_{\Gamma_{G_2}}=[\,x=g^1_2=x,\, y=g^2_2=y,\, z=g^3_2= z,\, u=g^4_2=\frac{w_{zz}}{w_z} + 2\frac{w_z}{v-w},\, u_x=D_x(g^4_2), u_{y}=D_y(g^4_2)\, ].
\EqTag{bfGg2}
\end{equation}
The reduced system $\CalI_2=(\CalK_1+\CalK_2)/\Gamma_{G_2}$ is generated by a pair of 2-forms
\begin{equation}
\CalI_2= \langle \ (du - u_x dx -u_y dy)\wedge dz,\, ( du_x-uu_x dz) \wedge dx + (du_y-uu_ydz) \wedge dy\ \rangle_{\rm alg},
\end{equation}
and represents equations \EqRef{OverDS} as an EDS. The fact
that the system $(\CalK_1+\CalK_2)/\Gamma_{G_2}$ leads to equations \EqRef{OverDS} can also be obtained by computing the syzygies among the repeated total derivatives of the differential invariant $g^4_2$ in \EqRef{bfGg2}.

For our next reduction  we let
\begin{equation}
\Gamma_{G_1}= \spn \{\, \partial_w,\, \partial_v,\, X_2\, \}.
\EqTag{G1ExE}
\end{equation}
The map $\bfq_{\Gamma_{G_1}}:M \to N_1=M/ \Gamma_{G_1}$ is given by finding the $\Gamma_{G_1}$ invariants and is (equation \EqRef{bfqs})
\begin{equation}
\bfq_{\Gamma_{G_1}} = [\, x=x,\, y=y,\, z=z,\, P =\log{w_z}-\log {v_x}, \, Q=\log{w_z}- \log {v_y},\, P_z=w_z^{-1} w_{zz}\, ],
\EqTag{ExEbfg1}
\end{equation}
where  $(x,y,z,P,Q,P_z)$ are coordinates on the 6-manifold $N_1$. The reduced EDS $\CalI_1 = (\CalK_1+\CalK_2)/\Gamma_{G_1}$ is
\begin{equation}
\CalI_1=\langle \ dP_z \wedge dz, \ {\rm e}^Q(dP -P_z dz) \wedge dx + {\rm e}^{P} (dQ-P_z   dz) \wedge dy\  \rangle_{\rm alg},
\EqTag{ODRE1}
\end{equation}
and the corresponding differential equations are
\begin{equation}
P_{xz}=P_{yz}=0,\   P_z=Q_z, \ {\rm e}^Q P_y={\rm e}^{P} Q_x.
\EqTag{I1f}
\end{equation}

For the last reduction, let $\Gamma_H=\Gamma_{G_1}\cap\Gamma_{G_2}=\spn \{X_1,X_2\}$. From the lowest order $\Gamma_H$ invariant on $M$, namely,
\begin{equation}
h^4=\log \frac{w_z}{v-w}, 
\EqTag{RST}
\end{equation}
we determine the function $\bfq_{\Gamma_H}:M \to N=M/\Gamma_{H}$ to be (equation \EqRef{bfqs})
\begin{equation}
\bfq_{\Gamma_H}=[\, x=x,\, y=y,\, z=z,\, S=h^4=\log \frac{w_z}{v-w},\, S_x=D_x(h^4),\, S_y=D_y(h^4), \, S_z=D_z(h^4)\, ].
\EqTag{ODqH}
\end{equation}
Note that $S_x<0$ and $S_y<0$ on  $M$. The reduced  system $\CalB= \CalI/\Gamma_H$ is
$$
\CalB=\langle \,  dS -S_xdx-S_y dy -S_z dz,\,  
dS_x\wedge dx + dS_y\wedge dy + dS_z\wedge dz, \, 
(dS_z-{\rm e}^{S}S_x dx-{\rm e}^{S}S_y dy)\wedge dz\,
 \rangle_{{\rm alg}}.
$$
This corresponds to the system of partial differential equations for $S$ as a function of $x,y,z$,
\begin{equation*}
 S_{xz}= {\rm e}^SS_x, \quad S_{yz}={\rm e}^SS_y.
\end{equation*}
The function $\bfp_2:N \to N_2$ is given by (see equations \EqRef{bfGg2} and \EqRef{ODqH})
\begin{equation}
\bfp_2=[\ x=x,\ y=y,\ z=z,\ u= S_z+{\rm e}^S,\ u_x=2{\rm e}^SS_x,\ u_y=2{\rm e}^SS_y\ ],
\EqTag{ODp2}
\end{equation}
while, from  equations \EqRef{ODqH} and \EqRef{ExEbfg1}, we find the map $\bfp_1:N \to N_1$ to be
\begin{equation}
\bfp_1=[\ x=x,\ y=y,\ P= S-\log(- S_x), \ Q= S-\log (-S_y),\ P_z= S_z-{\rm e}^S\ ].
\EqTag{ODSbfp1}
\end{equation}

In summary we have
\begin{equation}
\begin{gathered}
\beginDC{\commdiag}[3]
\Obj(-10, 34)[Jet]{$J^{2}(\real, \real) \times J^{1}(\real^2, \real)$}
\Obj(-10, 15)[B]{$\boxed{\begin{aligned}  S_{xz}&={\rm e}^S S_x   \\   S_{yz} & ={\rm e}^S S_y \end{aligned} \quad (\CalB) }$}
\Obj(30,  -5)[A2]{\ $\boxed{\begin{aligned} u_{xz}&= uu_x,    \\  u_{yz} &=uu_y \end{aligned} \quad (\CalI_2) }$\ }
\Obj(-60, -5)[A1]{$\boxed{\begin{aligned}     P_z &= Q_z,\,  {\rm e}^QP_y = {\rm e}^P Q_x,\\ P_{xz}&=P_{yz}=0 \end{aligned} \quad (\CalI_1) }$ \ }
\mor{Jet}{B}{$\bfq_{\Gamma_H}$}[\atright, \solidarrow]
\mor( -26,31)(-55,3){$ \bfq_{\Gamma_{G_1}}$}[\atright, \solidarrow]
\mor(6,31)(24,3){$\bfq_{ \Gamma_{G_2}}$}[\atleft, \solidarrow]
\mor(-5,8)(17,-4){$\bfp_2$}[\atright, \solidarrow]
\mor(-16,8)(-38,-4){$\bfp_1$}[\atleft, \solidarrow]
\enddc 
\end{gathered}
\EqTag{CDOD}
\end{equation}
\end{Example}
\noindent
The differential equations for $\CalB$  define a B\"acklund transformation between 
the over-determined system of differential equations for $\CalI_2$ and $\CalI_1$.

\begin{Example}
\StTag{ExC}
In this example we construct a B\"acklund transformation between two Goursat equations \cite{goursat:1900a}
\begin{equation}
	u_{xy} = \dfrac{2n\, \sqrt{u_x u_y}}{x+y}, \ n=0,1,\ldots
\EqTag{ExC1}
\end{equation}
with parameter values $n-1$ and $n$. This example is of interest because the dimension of the algebras used in the reduction diagram have dimensions which increase linearly as a 
function of the parameter $n$.

Let $\CalG_n$ be the standard Monge-Amp\`ere differential system for \EqRef{ExC1} on a 5-manifold. To construct a B\"acklund transformation  between
$\CalG_{n-1}$ and $\CalG_n$, we start with  $\CalI =  \CalH^{n+1}_1 + \CalH^{n +1}_2$, where     
	$\CalH^{n+1}_1$ and  $\CalH^{n+1}_2$
	are the standard rank $n+1$ Pfaffian systems on $(n+3)$-manifolds for the $n$-th order Monge equations 
\begin{equation*}
	 \frac{d u}{d s} = \bigl( \frac{d^n v}{d\, s^n}  \bigr)^2 \quad\text{and}  \quad 
	 \frac{d y}{d t} =  \bigl(\frac{d^n z}{d\, t^n}  \bigr)^2.
\end{equation*}  
	The vector fields
\begin{equation}
	 R= \partial_u,  \quad S_i  =s^{i}\partial_v,  \quad Y= \partial_y, \quad  Z_i = t^i\partial_z, \quad i = 0,1, \ldots 2n-1
\end{equation}
	all lift  to symmetries  of $\CalI$, which we again denote by $R,S_i,Y$ and $Z_i$. To apply Theorem A we  let
\begin{gather}
	\Gamma_{G_1} =\spn \{\, R +Y,  S_0, Z_0, S_i + Z_i\, \}_{1 \leq i \leq 2n-2}, \quad  
	\Gamma_{G_2} =\spn \{\, R +Y,   S_i + Z_i\, \}_{0 \leq i \leq 2n-1}  \\ \text{and}
\quad
	\Gamma_H = \Gamma_{G_1} \cap \Gamma_{G_2}  = \spn \{\, R +Y,   S_i + Z_i\, \}_{0 \leq i \leq 2n-2}.
 \end{gather}
	We conjecture that the reduction of $\CalI =  \CalH^{n+1}_1 + \CalH^{n +1}_2$  by $\Gamma_{G_1},\Gamma_{G_2}$ and $\Gamma_H$ leads to the following commutative diagram of  Pfaffian systems
\begin{equation}
\begin{gathered}
\beginDC{\commdiag}[3]
\Obj(0, 23)[H]{$\CalH^{n+1}_1 + \CalH^{n+1}_2$}
\Obj(0, 8)[B]{$\CalB_n$}
\Obj(-35, -10)[I1]{$\boxed{ U_{xy} = \dfrac{2(n-1)\, \sqrt{U_x U_y}}{x+y} \quad (\CalG_{n-1})}$}
\Obj(30,  -10)[I2]{$\boxed{ V_{xy} = \dfrac{2n\, \sqrt{V_xV_y}}{x+y}\quad (\CalG_n)}$}
\mor{H}{B}{$\bfq_{\Gamma_H}$}[\atright, \solidarrow]
\mor(-8,22)(-30,-5){$\bfq_{\Gamma_{G_1}}$}[\atright, \solidarrow]
\mor(8,22)(26,-5){$\bfq_{\Gamma_{G_2}}$}[\atleft, \solidarrow]
\mor{B}(-20,-5){$\bfp_2$}[\atright, \solidarrow]
\mor{B}(20,-5){$\bfp_1$}[\atleft, \solidarrow]
\enddc 
\end{gathered}
\EqTag{CONJ}
\end{equation}
Detailed formulas  for the various projection maps and quotients in \EqRef{CONJ} when $n=1,2,3$ can be found in Section 8 of \cite{anderson-fels:2009a}. For these small values of $n$ the explicit construction of the B\"acklund transformation $\CalB_n$ is not difficult.  We are unable to calculate the $\Gamma_{G_a}$ and $\Gamma_H$ differential invariants in closed form for a general $n$. However, the following theorem can be easily proved through direct computation, independent of \EqRef{CONJ}.

\begin{Theorem} For any $n\in \real$, the  differential equations 
\begin{equation}
	U_{xy} = \dfrac{2(n-1)\, \sqrt{U_x U_y}}{x+y} \quad\text{and}\quad  V_{xy} = \dfrac{2n\, \sqrt{V_xV_y}}{x+y}
\end{equation}
are related by the B\"ackland  transformation
\begin{equation}
(\sqrt{U_x} - \sqrt{V_x})^2  = \frac{(2n -1)(U- V)}{x+y} =    (\sqrt{U_y} + \sqrt{V_y})^2 .
\EqTag{BackGour}
\end{equation}
\end{Theorem} 
Note that for $n=1$ this coincides with the B\"acklund transformation given in \cite{zvyagin:1991a}. For $n=2$ the system $\CalI$ coincides with the system formed in Example \StRef{ExB} using equation \EqRef{HCeq}. Finally the B\"acklund transformations \EqRef{CONJ}, yield a chain of  B\"acklund transformations, having 1-dimensional fibres, relating \EqRef{ExC1} to the wave equation $\CalG_0$
\begin{equation}
\begin{gathered}
\begindc{\commdiag}[3]
\obj(-45, 12)[B1]{$\CalB_1$}
\obj(-15, 12)[B2]{$\CalB_2$}
\obj(15, 12)[B3]{$\CalB_3$}
\obj(-60, 0)[G0]{$\CalG_0$}
\obj(-30,  0)[G1]{$\CalG_1$}
\obj(0,  0)[G2]{$\CalG_2$}
\obj(30,  0)[G3]{$\CalG_3$}
\mor{B1}{G0}{$\bar \bfp_0$}[\atright, \solidarrow]
\mor{B1}{G1}{$\bfp_1$}[\atright, \solidarrow]
\mor{B2}{G1}{$\bar \bfp_1$}[\atright, \solidarrow]
\mor{B2}{G2}{$\bfp_2$}[\atright, \solidarrow]
\mor{B3}{G2}{$\bar \bfp_2$}[\atright, \solidarrow]
\mor{B3}{G3}{$\bfp_2$}[\atright, \solidarrow]
\enddc
\end{gathered}
\ldots
\EqTag{SeqBT}
\end{equation}
The system $\CalE_n = (\CalH^{n+2}_1 + \CalH^{n +2}_2)/\Gamma_{K}$ where $
\Gamma_K = \spn \{\ R+Y ,\ S_i + Z_i\ \}_{0\leq i \leq 2n-1}$,
defines a B\"acklund transformation between $\CalB_n$ and $\CalB_{n+1}$ with one dimensional fibres. Diagram \EqRef{SeqBT} can then be extended to include a chain of B\"acklund transformations between the systems $\CalB_n$.
\end{Example}

\section{Darboux Integrability and Symmetry Reduction}\StTag{DDI}

The examples in Section \StRef{Examples} are all {\it Darboux integrable}.  We begin this section by reviewing this definition in Section \StRef{DefDI}.  In Section \StRef{SRDI} we highlight the role that symmetry reduction of differential systems plays in the construction of Darboux integrable systems.  The Darboux integrability
of the examples in Section \StRef{Examples} is then shown in Section \StRef{ExP2} to be an immediate consequence of symmetry reduction. We also show how symmetry reduction can be used to determine the fundamental geometric invariants of a Darboux integrable system.

\subsection{Darboux Integrability}\StTag{DefDI}

To begin, we recall the following definition \cite{bryant-griffiths-hsu:1995a}. A {\deffont hyperbolic system of class $s$} is an EDS $\CalI$ on a $(s+4)$-dimensional manifold $M$ with the following property. For each point $x\in M$, there is an open neighbourhood $U$ of $x$ and a coframe $\{ \theta^1,\ldots,\theta^s, \hsigma^1,\hsigma^2,\csigma^1,\csigma^2 \}$ on $U$ such that
$$
\CalI|_U = \langle \ \theta^1,\,\ldots,\, \theta^s,\ \hsigma^1 \wedge \hsigma^2,\
 \csigma^1 \wedge \csigma^2\ \rangle_\text{\rm alg}.
$$
A generalization of this definition is the following \cite{anderson-fels-vassiliou:2009a}.

\begin{Definition}
\StTag{DI1}	
	An exterior differential system  $\CalI$ on $M$ is  {\deffont decomposable of type $[p,\rho]$}, where
	$p,\rho \geq 2$, if for each point $x \in M$ there is an open neighbourhood $U$ of $x$ and a coframe
\begin{equation*}
	\tilde \theta^1,\ \ldots,\ \tilde \theta^s,\
	\hsigma^1,\ \dots,\ \hsigma^{p}, \
	\csigma^1,\ \dots,\ \csigma^{\rho}
\end{equation*}
 on $U$	such that $\CalI|_U$ is algebraically generated by  1-forms and 2-forms
\begin{equation}
	\CalI|_U = \langle \, \tilde \theta^1, \ \dots, \ \tilde \theta^s,\ 
	A^1_{ab} \, \hsigma^a \wedge  \hsigma^b,\ldots,A^t_{ab} \, \hsigma^a \wedge  \hsigma^b,\
	B^1_{\alpha\beta} \ \csigma^\alpha \wedge  \csigma^\beta,\ldots,B^\tau_{\alpha\beta} \ \csigma^\alpha \wedge  \csigma^\beta
	\, \rangle_\text{\rm alg},
\EqTag{Intro5}	
\end{equation}
	where $t, \tau \geq  1$.
\end{Definition}	

A class $s$ hyperbolic system is a decomposable differential system of type [2, 2].  The standard rank $3k$ Pfaffian system on a $(5k+2)$-manifold for a $f$-Gordon system $u^\alpha_{xy}=f^\alpha(x,y,u^\beta, u^\beta_x,u^\beta_y)$, $\alpha,\beta=1,\ldots,k$ is a decomposable system of type [$k+1$, $k+1$].

The  characteristic systems defined in  \cite{bryant-griffiths-hsu:1995a} for a class $s$ hyperbolic differential system possess the following analogue for decomposable systems.
	
\begin{Definition}\StTag{Singsys}  Let $\CalI$ be a decomposable differential system with given decomposition \EqRef{Intro5}. The bundles $\hV, \cV \subset T^*M$ defined by
\begin{equation}
\hV = \text{span} \{\,\tilde \theta^e,\, \hsigma^a  \}
	\quad\text{and}\quad
	\cV = \text{span} \{\, \tilde \theta^e,\, \csigma^\alpha  \}
\EqTag{SPS}
\end{equation}
are called the {\deffont associated singular Pfaffian systems} with respect to the decomposition \EqRef{Intro5}. \end{Definition}
	
In the special case where $\CalI$ is a class $s$ hyperbolic system,  the decomposition in Definition \StRef{DI1} is unique up to an interchange of the singular Pfaffian systems. However, for a general decomposable system the decomposition in Definition \StRef{DI1} may not be unique. For further information on the relationship between a decomposable differential system and its singular systems see Theorem 2.6 of  \cite{anderson-fels-vassiliou:2009a}.


A {\deffont local first integral}  of $\hV$ is a smooth function 
	$f\: U \to \Real$, defined on an open set $U$ of $M$,  such that $d f \in \CalS(\hV)$, the space of sections of $\hV$. For each point $x \in M$ we define
\begin{equation}
	\hV_x^\infty  = \spn \{\, df _x \, | \,   \text{$f$ is a local first integral, defined about $x$}\,\} .
\end{equation}
	We shall always assume that  $\hV^\infty   = \bigcup _{x\in M} \hV_x^\infty$ is a constant rank bundle on $M$ which we call the {\deffont bundle of first integrals} of $\hV$. We also make analogous assumptions for $\cV^\infty$.  
		
\begin{Definition} \StTag{DInv} A  {\deffont local intermediate integral} for a decomposable system $\CalI$ is
a local first integral of either singular Pfaffian system $\hV$ or $ \cV$, that is, a $C^\infty(U)$ function $f$ such that $d f \in \CalS( \hV|_U)$ or  $df \in  \CalS(\cV|_U)$.
\end{Definition}

The number of (functionally) independent intermediate integrals is therefore given by the sum of the ranks  of  the completely integrable systems  $\hV^\infty$ and $\cV^\infty$.  The definition of a Darboux integrable differential system is then given in terms of the bundles of first integrals of its singular Pfaffian systems $\hV$ and $\cV$.

\begin{Definition}
\StTag{Intro2} A decomposable differential system $\CalI$ is {\deffont Darboux integrable}  
	if its singular Pfaffian systems  \EqRef{SPS} satisfy

\noindent
{\bf[i]}
\vskip -30pt
\begin{equation}\hV  + \cV^{\infty} = \cTM
	\quad\text{and}\quad
        \cV  + \hV^{\infty} = \cTM ,\quad \text{and}
\EqTag{Intro7}
\end{equation}
\par
\smallskip
\noindent
{\bf[ii]}
\vskip -30pt
\begin{equation}
 \hV^{\infty} \cap \cV^{\infty} = \{\,0\, \}.
\EqTag{Intro8}
\end{equation}
\end{Definition}

This definition implies that if a decomposable system $\CalI$ on an $n$-manifold is Darboux integrable, then it admits at least $n-\rank(\hV)$ intermediate integrals which are first integrals of $\cV$ and at least $n-\rank (\cV)$ intermediate integrals which are first integrals of $\hV$.
In the case where $\CalI$ itself admits no first integrals, condition \EqRef{Intro7} is sufficient to
establish Darboux integrability, see Theorem 4.6 in \cite{anderson-fels:2014a}.


\subsection{Symmetry Reduction and Darboux Integrabilty}\StTag{SRDI}

Theorem 6.1 and Theorem 7.7 in \cite{anderson-fels:2014a} provide a general way to construct Darboux integrable systems  and to compute their intermediate integrals through symmetry reduction.  We summarize these theorems in Theorem \StRef{Dreduce} below and  then,
in Section \StRef{ExP2}, apply them to the examples from Section \StRef{Examples}.

To state these theorems, let $\CalK_1$ and $\CalK_2$ be EDS  on manifolds $M_1$ and $M_2$.
	The {\deffont direct sum } $\CalK_1 + \CalK_2$ is the EDS on $M_1 \times M_2$  which is
	algebraically generated by the pullbacks of $\CalK_1$ and $\CalK_2$  to $M_1 \times M_2$ 
	by the canonical projection maps $\pi_a\colon M_1\times M_2 \to M_a$, $a=1,2$.

	Let $G$ be a Lie group acting on $M_a$,  let  $L$ be a subgroup of  the product  group $G \times G$, 
	let $\rho_a : G\times G \to G$, $a=1,2$ be the projection onto the $a^{\text{th}}$ factor and let 
	$L_a = \rho_a(L) \subset  G$.  The 
	action of $L$ on $M_1\times M_2$  is  then given in terms of these projection maps and the actions of $G$ on 
	$M_1$  and $M_2$ by
\begin{equation}
	  \ell \cdot ( x_1, x_2) = (\rho_1(\ell) \cdot x_1, \,\rho_2(\ell) \cdot x_2) \quad \text{for $\ell \in L$}.
\EqTag{Laction}
\end{equation}
	The projection maps $\pi_a: M_1 \times M_2 \to M_a$ are equivariant  with respect the actions of $L$ and $L_a$, that  is,  
	for any $\ell \in L$ and  $x\in M_1 \times M_2$ 
\begin{equation}
	\pi_a(\ell \cdot x) = \rho_a(\ell) \cdot \pi_a(x).
\EqTag{Lequiv}
\end{equation}
The actions of $L$ and $L_a$ can be then used to construct the commutative diagram 
\begin{equation}
\begin{gathered}
\beginDC{\commdiag}[3]
\obj(0,22)[M1M2]{$M_1\times M_2$}
\obj(0, 8)[B]{$(M_1\times M_2)/L$}
\obj(-32, 0)[I1]{$M_1/L_1$}
\obj(32, 0)[I2]{$M_2/L_2$\,}
\mor{B}{I1}{$\tbfp_{L_1}$}[\atleft, \solidarrow]
\mor{B}{I2}{$\tbfp_{L_2}$}[\atright, \solidarrow]
\mor{M1M2}{B}{$\bfq_L$}[\atright, \solidarrow]
\mor{M1M2}{I1}{$\tbfpi_{L_1}$}[\atright, \solidarrow]
\mor{M1M2}{I2}{$\tbfpi_{L_2}$}[\atleft, \solidarrow]
\enddc
\end{gathered}
\EqTag{CDCP}
\end{equation}
where $\tbfpi_{L_a}= \bfq_{L_a} \circ \pi_a$, and $\tbfp_{L_a}( L (x_1,x_2))= L_a x_a$. 

We say that a subgroup $L\subset G\times G$ is diagonal if $L \subset G_{\diag}$. In
this case $L_a \cong L$. If $L=G_\diag=\{(g,g)\, | \, g\in G\}$, then $L_a = G$ and the actions of $L_a$ on $M_a$ coincide with the original actions of $G$ on  $M_a$.

The following theorem, whose statement and proof rely on diagram \EqRef{CDCP}, summarizes the essential facts  regarding the construction of Darboux integrable systems by 
	symmetry reduction.

  \begin{Theorem}
\StTag{Dreduce}  
	Let $\CalK_a, a=1,2$ be EDS on $M_a$, $a = 1,2$ which are algebraically generated by
	1-forms and 2-forms. Assume that $( K_a^1)^\infty  = 0$ (where $\CalK^1_a=\CalS(K^1_a)$, see Section \StRef{EREDS}) and let
\begin{equation}
	\hW= K_1^1 + T^*M_2 \quad {\rm and} \quad  \cW= T^*M_1 +K_2^1.
\EqTag{Wsing}
\end{equation}
Furthermore, consider a  Lie group  $G$  which  acts freely on $M_1$ and $M_2$,  is a {\it common} symmetry group of  both $\CalK_1$  
	and $\CalK_2$, and  acts transversely to $\CalK_1$  and $\CalK_2$. Let $L\subset G\times G$ be a given subgroup and assume that the actions of $L$ on $M_1\times M_2$ and  $L_a=\rho_a(L)$ on $M_a$ 
	are regular and set
\begin{equation}
	M = (M_1 \times M_2)/L, \quad \hV = (K^1_1+T^*M_2)/L, \quad \cV=(T^*M_1+K^1_2)/L.
\EqTag{WMODG}
\end{equation}
	Finally, assume that $\hV^\infty$ and $\cV^\infty$ are constant rank bundles.

\smallskip
\noindent
{\bf [i]} The reduced differential system  $ \CalI = (\CalK_1 + \CalK_2)/L$ on  $M/L$
	is a decomposable Darboux integrable system with singular Pfaffian systems $\{\hV, \cV\}$ as defined in  \EqRef{WMODG}.

\smallskip
\noindent
{\bf [ii]} The bundle of first integrals $\hV^{\infty} $ and $\cV^{\infty} $ for  $\hV$ and $\cV$ satisfy
\begin{equation}
\hV^{\infty} = \tbfp_{L_2}^*(T^*(M_2/L_2)),\quad\text{and}\quad \cV^{\infty} = \tbfp_{L_1}^*(T^*(M_1/L_1)),
\EqTag{QINVS}
\end{equation}
where $ \tbfp_{L_a}: M/L \to M_a/L_a$  are the orbit projection maps given in equation \EqRef{CDCP}.
\end{Theorem}


When $L$ is a diagonal subgroup and $\CalI=(\CalK_1+\CalK_2)/L$, we call $L$ the {\deffont{Vessiot group of $\CalI$}} and we call $(\CalK_1+\CalK_2)/L=\CalI$ {\deffont the canonical quotient representation of $\CalI$}. We also define the Vessiot algebra $\vess(\CalI)$ to be the  Lie algebra of the Vessiot group $L$. 
It then follows from Theorem \StRef{Dreduce} that
\begin{equation}
\dim \vess(\CalI) = \dim L = \dim M - \rank (\hV^\infty) - \rank (\cV^\infty).
\EqTag{VAdim}
\end{equation}
In the case where $L=G_{\diag}$, then $L\cong G$ and $G$ is the Vessiot group. 

If $L\subset G\times G$ is {\it not} diagonal then $L$ is not the Vessiot group of $(\CalK_1+\CalK_2)/L$. In this case the following theorem (see Section 6.3 in \cite{anderson-fels:2014a}) describes the reduction which produces the Vessiot algebra and canonical quotient representation of $\CalI=(\CalK_1+\CalK_2)/L$.
\noindent
\begin{Theorem} \StTag{Rdiag} Reduction to a Diagonal Action.  Let $\CalK_a$, $G$ and $L$ be as in Theorem \StRef{Dreduce}
and define the subgroups $A_1\subset L_1$ and $A_2\subset L_2$ by
\begin{equation}
	A_1 = \rho_1(\ker \rho_2) = \{ g_1\in G \,|\, (g_1,e) \in L\} \quad\text{and} \quad 
	A_2 = \rho_2(\ker \rho_1)  =  \{ g_2\in G \,|\, (e, g_2) \in L\} .
	\EqTag{AGS}
\end{equation}
The differential systems   $\widetilde \CalK_a= \CalK_a/A_a$ on $\widetilde M_a=M_a/A_a$, 
and the group $\widetilde L = L/(A_1 \times A_2)$ satisfy the hypothesis of Theorem \StRef{Dreduce}, and
\begin{equation}
\CalI= (   \widetilde \CalK_1 +\widetilde \CalK_2 ) / \widetilde L_{\diag}.
\EqTag{CQR}
\end{equation}
\end{Theorem}
The action of $\widetilde L$ on $\widetilde M_a$ in Theorem \StRef{Rdiag} is defined using the fact that $\widetilde L \cong L_a/A_a$. See Section 6.3 of \cite{anderson-fels:2014a} for details. 

We conclude this subsection with a few remarks.
 
\noindent 
{\bf [i]} Theorem \StRef{Dreduce} was first proved for the case $L=G_\diag$ in  \cite{anderson-fels-vassiliou:2009a}, see Corollary 3.4. The extension to non-diagonal actions is given by Theorem 6.1 in \cite{anderson-fels:2014a}. This extension is the essential result we use to construct B\"acklund transformations between Darboux integrable systems.

\noindent
{\bf [ii]} A remarkable fact relating Darboux integrable systems and symmetry reduction of EDS is that  Theorem \StRef{Dreduce} admits a local converse (Theorem 1.4, \cite{anderson-fels-vassiliou:2009a}). Specifically, every Darboux integrable system can be realized locally as the quotient of a pair of exterior differential systems with a common symmetry group whose dimension is given by \EqRef{VAdim}. 

\noindent
{\bf [iii]} The system $\CalK_1+\CalK_2$ in Theorem \StRef{Dreduce} is Darboux integrable with singular Pfaffian systems in equation  \EqRef{Wsing}.  This observation allows us to say, roughly speaking, that Theorem \StRef{Dreduce} shows Darboux integrability is preserved under reduction; see also Theorem 3.2 in \cite{anderson-fels-vassiliou:2009a}.

\noindent
{\bf [iv]} An application of Theorem \StRef{Dreduce} to the Cauchy problem for Darboux integrable systems can be found in \cite{anderson-fels:2013a}.

\subsection{Finding Intermediate Integrals using Symmetry Algebras}\StTag{Firstint}

The examples in Section \StRef {Examples} were constructed using symmetry algebras of vector fields. Since we will continue in this manner in Section \StRef{ExP2}, we now provide a local coordinate form for the computation of the intermediate integrals produced by the local version of part {\bf [ii]} in Theorem \StRef{Dreduce} using Lie algebras of vector fields.  This 
requires that we first produce the local coordinate form of diagram \EqRef{CDCP} in this case.

Let $\Gamma_G^a$ be Lie algebras of vector fields on the manifolds $M_a$, $a=1,2$ which are  isomorphic as abstract Lie algebras, and let $\Gamma_L \subset \Gamma_G^1+\Gamma_G^2$ be a subalgebra of vector fields on $M_1\times M_2$.  Theorem \StRef{Dreduce} requires we compute the projected algebras
\begin{equation}
\Gamma_{L_a} = \rho_a(\Gamma_{L}),
\EqTag{AAS}
\end{equation}
where $\rho_a:\Gamma^1_{G} + \Gamma^2_{G} \to \Gamma^a_{G}, a=1,2$ is the Lie algebra homomorphism which projects onto the $a^{\text{th}}$ factor. If $X \in \Gamma^1_G+\Gamma_G^2$, then $\rho_a(X) = \pi_{a*}(X)$.

The invariants of the Lie algebras of vector fields $\Gamma_{L_a}$ on $M_a$ can now be used to compute the intermediate integrals of $(\CalK_1+\CalK_2)/\Gamma_L$ (or the first integrals of $\hV$ and $\cV$) using equation \EqRef{QINVS}. Let $\{\cJ^a\}$ and $\{\hJ^i\}$ be sets of independent $\Gamma_{L_1}$ and $\Gamma_{L_2}$  local invariants on $M_1$ and $M_2$ respectively. The projection maps $\bfq_{\Gamma_{L_a}}:M_a\to M_a/\Gamma_{L_a}$ can be written as
\begin{equation}
\bfq_{\Gamma_{L_1}} = [  \ \cf^a = \cJ^a(y^t) \ ]\, , \ {\rm and} \quad \bfq_{\Gamma_{L_2}} = [  \ \cf^i = \hJ^i(x^s) \ ], 
\EqTag{qLc}
\end{equation}
where $y^t$ and $x^s$ are local coordinates on $M_1$ and $M_2$, $\cf^a$ are local coordinates on $M_1/\Gamma_{L_1}$,  and $\hf^i$ on $M_2/\Gamma_{L_2}$. Similarly, let $\{h^A\}$ be a set of independent  $\Gamma_L$ invariants on $M_1\times M_2$ so that 
$\bfq_{\Gamma_L}:M_1\times M_2 \to (M_1\times M_2)/\Gamma_L$ can be written as
$$
\bfq_{\Gamma_L}= [\, u^A= h^A(y^t, x^s) \ ] ,
$$
where $u^A$ are local coordinates on $(M_1\times M_2)/\Gamma_L$. 

The functions $\pi_1^{*} (\cJ^a)$ and $\pi_2^*( \hJ^i)$ on $M$ are $\Gamma_L$  invariant and hence, just as in the argument that produces \EqRef{pinc}, there exists functions $\cF^a$ and $\hF^i$ on $M$ such that
$$
\pi_1^{*} (\cJ^a) = \cF^a( h^A) \quad {\rm and} \quad \pi_2^{*} (\hJ^a) = \hF^a( h^A).
$$
The maps $\tbfp_{\Gamma_{L_a}}:M/\Gamma_L \to M_a/\Gamma_{L_a}$ are then given in coordinates by
\begin{equation}
\bfp_{\Gamma_{L_1}}= [   \ \cf^a = \cF^a( u^A) \ ] \quad {\rm and} \quad \bfp_{\Gamma_{L_2}}= [   \ \hf^a = \hF^a( u^A) \ ].
\EqTag{cfpL1}
\end{equation}
Finally,  by equation \EqRef{QINVS}, the functions $\cI^a$ and $\hI^i$ on $(M_1\times M_2)/\Gamma_L$ defined by
\begin{equation}
\cI^a = \cf^a\circ \tbfp_{\Gamma_{L_1}}(u^A)= F^a(u^A)  \quad {\rm and} \quad \hI^i =  \hf^i\circ \tbfp_{\Gamma_{L_2}} (u^A)  = \hF^i( u^A)
\EqTag{infFI}
\end{equation}
are a complete set of intermediate integrals of $(\CalK_1+\CalK_2)/\Gamma_L$. That is we have,
$$
\hV^\infty = \spn \{ \ \tbfp_{\Gamma_{L_2}}^* d \hf^ i\ \}\quad  {\rm and} \quad \cV^\infty = \spn \{\ \tbfp_{\Gamma_{L_1}}^* d \cf^ a \ \}.
$$   

The intermediate integrals for $(\CalK_1+\CalK_2)/L$ can also be computed using a local cross-section $\sigma:U  \to M_1\times M_2$, $U \subset (M_1\times M_2)/\Gamma_L$  (as in equation \EqRef{csa}) by 
\begin{equation}
\cI^a=  \cJ^a \circ \pi^1 \circ \sigma \, , \ {\rm and} \quad \hI^i = \hJ^i \circ \pi^2 \circ \sigma.
\EqTag{FIpb}
\end{equation}
{\it Equation \EqRef{FIpb} clearly shows how the intermediate integrals of $(\CalK_1+\CalK_2)/\Gamma_L$ can be expressed in terms of the invariants of $\Gamma_{L_a}$ on $M_a$}.

\section{Examples: Intermediate Integrals by Quotients}\StTag{ExP2}

The differential systems $\CalI_1$, $\CalI_2$ and $\CalB$ in our examples from Section \StRef{Examples} are Darboux integrable. Therefore, we use Theorem \StRef{Dreduce} (or Section \StRef{Firstint}) to  compute the intermediate integrals for these differential systems. While many of
the intermediate integrals are well-known \cite{goursat:1897a}, \cite{vessiot:1942a}, our goal
here is to illustrate the important relationship between the intermediate integrals of $\CalI$
and the (differential) invariants of its Vessiot group.  In Section \StRef{CVA} we use Theorem \StRef{Rdiag}  to compute the Vessiot algebra for the examples constructed from non-diagonal group actions. 

\medskip
The Lie algebra $\Gamma_L$ in Section \StRef{Firstint}  will  be replaced by $\Gamma_H, \Gamma_{G_1}$ and $\Gamma_{G_2}$ for the examples,  in which case diagram \EqRef{CDCP} becomes,
\begin{equation}
\begin{gathered}
\beginDC{\commdiag}[3]
\obj(-27,22)[M1M2]{$M_1\times M_2$}
\obj(-27, 8)[B]{$N$}
\obj(-54, 0)[I1]{$M_1/\Gamma_{H_1}$}
\obj(0, 0)[I2]{$M_2/\Gamma_{H_2} $\,}
\obj(47,22)[M1M2t]{$M_1\times M_2$}
\obj(47, 8)[Bt]{$N_a$}
\obj(20, 0)[I1t]{$M_1/\Gamma_{G_a,1}$}
\obj(74, 0)[I2t]{$M_2/\Gamma_{G_a,2}$\,}
\mor{B}{I1}{$\tbfp_{\Gamma_{H_1}}$}[\atleft, \solidarrow]
\mor{B}{I2}{$\tbfp_{\Gamma_{H_2}}$}[\atright, \solidarrow]
\mor{M1M2}{B}{$\bfq_H$}[\atright, \solidarrow]
\mor(-32,21){I1}{$\tbfpi_{\Gamma_{H_1}}$}[\atright, \solidarrow]
\mor(-22,21){I2}{$\tbfpi_{\Gamma_{H_2}}$}[\atleft, \solidarrow]
\mor{Bt}{I1t}{$\tbfp_{\Gamma_{G_a,1}}$}[\atleft, \solidarrow]
\mor{Bt}{I2t}{$\tbfp_{\Gamma_{G_a,2}}$}[\atright, \solidarrow]
\mor{M1M2t}{Bt}{$\bfq_{G_a}$}[\atright, \solidarrow]
\mor(42,21){I1t}{$\tbfpi_{\Gamma_{G_a,1}}$}[\atright, \solidarrow]
\mor(52,21){I2t}{$\tbfpi_{\Gamma_{G_a,2}}$}[\atleft, \solidarrow]
\enddc
\end{gathered}
\EqTag{CDCP2}
\end{equation}
where $\Gamma_{H_i}=\rho_i(\Gamma_H)$ and $\Gamma_{G_a,i} = \rho_i(\Gamma_{G_a})$. Using equation \EqRef{infFI} or \EqRef{FIpb} with the coordinate form of the maps from diagram \EqRef{CDCP2} will produce the intermediate integrals for our examples.


\subsection{Examples}

\begin{Example}\StTag{ExA1p}   We calculate the intermediate integrals for the
three Darboux integrable systems $\CalB$, $\CalI_1$, and $\CalI_2$ in Example \StRef{ExA1}.  We start with the system 
$\CalB$ which was constructed as a quotient by the Lie algebra of vector fields $\Gamma_H$ defined in \EqRef{ExA2}. The quotient map $\bfq_{\Gamma_H}:M \to N$ is given in coordinates by equation \EqRef{bfpH}.  The structure equations for $B=\spn\{\, \beta^1, \beta^2\, \}$ in equation \EqRef{ExA3} are
\begin{equation}
d \beta^1 =-\hpi \wedge \homega+ \rme^W \beta_2 \wedge \comega, \quad  d \beta^2=- \cpi\wedge \comega + \rme^V \beta^1\wedge \homega\, ,
\EqTag{StructB1}
\end{equation}
where
\begin{equation*}
\homega =dx,\quad \hpi=dV_x-\rme^{V+W}dy, \quad \comega =dy, \quad
\cpi= dW_y -\rme^{V+W} dx.
\end{equation*} 
These structure equations show that the singular Pfaffian systems (Definition \StRef{Singsys}) for $\CalB$ are
\begin{equation}
\hV = \spn\{ \ \beta ^1 ,\ \beta^2,\ \hpi,\ \homega\  \}  \quad {\text{and}} \quad
\cV = \spn\{\ \beta^1 ,\ \beta^2,\ \cpi,\ \comega \ \}.
\EqTag{BH1}
\end{equation}

The first integrals for the singular systems $\hV$ and $\cV$ are straight-forward to compute directly, but here we wish to use part {\bf [ii]}  of Theorem \StRef{Dreduce} in the form of Section \StRef{Firstint} to find them. We
first determine the invariants $\hJ^i$ and $\cJ^a$ of the projected algebras $\Gamma_{H_1}$ and $\Gamma_{H_2}$ of $\Gamma_H$ in \EqRef{ExA2}. The algebra $\Gamma_{H_2}$ on $M_2=J^2(\real,\real)$ with coordinates $(x,v,v_x,v_{xx})$ is
$$
\Gamma_{H_2} = \rho_2(\Gamma_H)=\spn \{\, 	 \partial_v,\,     v\,\partial_v+v_x \partial_{v_x}+v_{xx} \partial_{v_{xx}} \, \}.
$$
The invariants $\hJ^1$ and $\hJ^2$ are easily found so that the projection $ \bfq_{\Gamma_{H_2}}:M_2 \to M_2/\Gamma_{H_2}$  in equation \EqRef{CDCP2} is (see also equation \EqRef{qLc}) 
\begin{equation}
\bfq_{\Gamma_{H_2}} =[ \ \hf^1 = \hJ^1=x, \ \hf^2= \hJ^2=v_{xx}v_x^{-1} ].
\EqTag{Ffs}
\end{equation}

By expressing the $\Gamma_{H_2}$ invariants $\hJ^1$ and $\hJ^2$ (pulled back to $M=M_1\times M_2$) in \EqRef{Ffs} in terms of the $\Gamma_H$ invariants  in equation \EqRef{bfpH} we arrive at
$$
\pi_2^*(\hJ^1)=x= h^1,\quad \pi_2^* (\hJ^2)=v_{xx}v_x^{-1}= h^5 +{\rm e}^{h^4}.
$$
This will allow us to determine the coordinate form for the map $\tbfp_{\Gamma_{H_2}}$ in equation \EqRef{CDCP2} as described in equation \EqRef{cfpL1}.  By using the coordinates $\hf^i$ in \EqRef{Ffs} on $M_2/\Gamma_{H_2}$ and those used in \EqRef{bfpH} on $N=(M_1\times M_2)/\Gamma_H$, we find $\tbfp_{\Gamma_{H_2}}:N\to M_2/\Gamma_{H_2}$ is
$$
\tbfp_{\Gamma_{H_2}}=[\ \hf^1=x,\ \hf^2 = V_x + e^V\ ].
$$
This gives the first integrals for $\hV$ (see equation \EqRef{infFI})
\begin{equation}
\hI^1=\tbfp_{\Gamma_{H_2}}^*( \hf^1)=x \quad {\rm and}\quad \hI^2 = \tbfp_{\Gamma_{H_2}}^* (\hf^2)=V_x + e^V .
\EqTag{B1infh}
\end{equation}
From \EqRef{BH1}, it is easily checked that $d\hI^a$ lies in $\hV$ . To check that $\hI^2=V_x+ e^V$ is an intermediate integral in the classical sense for the PDE system \EqRef{PDEH1}, we compute
$$
\frac{d\ }{dy} (V_x+ e^V)  = V_{xy} + e^V V_y = 0.
$$
Similarly, the first integrals for $\cV$  in \EqRef{BH1} are computed from the differential invariants 
$\cJ^1=y$ and $\cJ^2=v_{yy}v_y^{-1}$ of  $\Gamma_{H_1}$ on $J^2(\real,\real)$. Alternatively, we can use equation \EqRef{FIpb}  to produce the intermediate integrals from the cross-section of $\bfq_H$ in equation \EqRef{sigmaex1}, that is,
\begin{equation}
\cI^1=\sigma^*( \pi_1^* (\cJ^1))=y\ , \ {\rm and} \quad \cI^2=\sigma^*(\pi_1^*(\cJ^2)) =   W_y + e^W .
\EqTag{B1infc}
\end{equation}

We compute the intermediate integrals of $\CalI_1$ on $N_1$ in the same manner.  The projections  (equation \EqRef{AAS}) for the action of $\Gamma_{G_1}=\spn\{X_1,X_2,Z_1\}$ (equations \EqRef{ExA2} and \EqRef{defZ1}) are
\begin{equation}
\begin{aligned}
\Gamma_{G_1,1} & =\rho_1(\Gamma_{G_1})= \spn\{\  \partial_w,\ w \partial_w+w_y \partial_{w_y}+ w_{yy} \partial_{w_{yy}}\ \} \quad {\text{and}}\\
\Gamma_{G_1,2} & =\rho_2(\Gamma_{G_1})= \spn\{ \ \partial_v,\  v \partial_v+v_x \partial_{v_x}+ v_{xx} \partial_{v_{xx}} \ \}.
\end{aligned}
\EqTag{G1pro}
\end{equation}
The differential invariants of $\Gamma_{G_1,2}$ in \EqRef{G1pro} on $J^2(\real,\real)$ are easily computed so that equation \EqRef{qLc} yields $\bfq_{\Gamma_{G_1,2}}= [ \hf^1 = \hJ^1=x, \hf^2= \hJ^2=v_{xx}v_x^{-1} ]$. By expressing these invariants in terms of the $\Gamma_{G_1}$ invariants from \EqRef{Fp1} we have $\pi_2^*(\hJ^2)=D_x(g^3_1)$. With respect to the coordinates on $N_1$ used in \EqRef{QG1}, we find
\begin{equation*}
\tbfp_{\Gamma_{G_1,2}}=[\, \hf^1= x,\, \hf^2 = u_{1x}\, ].
\end{equation*}
This gives (equation \EqRef{infFI}) $\hI^1=x,\ \hI^2 =u_{1x}$ as a well-known pair of intermediate integrals for the wave equation. The intermediate integrals $\cI^1=y$ and $\cI^2=u_{1y}$ are similarly obtained from $\Gamma_{G_1,1}$.

We now use the Lie algebra of vector fields $\Gamma_{G_2}=\spn\{X_1,X_2,Z_2\}$  given by equations \EqRef{ExA2} and \EqRef{defZ2} to find the intermediate integrals of $\CalI_2$. In this case the algebra $\Gamma_{G_2}$ does {\it not} satisfy the transversality hypothesis in Theorem \StRef{Dreduce} and it is necessary to prolong $\Gamma_{G_2}$ to symmetry vector-fields of the prolongation of $\CalI$. The prolongation of $\Gamma_{G_2}$ (which we will again call $\Gamma_{G_2}$) to $J^3(\real,\real) \times J^3(\real,\real)$ is given by
\begin{equation}
\begin{aligned}
\Gamma_{G_2}  =\spn\{     \partial_w - \partial_v,   \     w\,\partial_w + w_y \partial_{w_y} + w_{yy} \partial_{w_{yy}}+ w_{yyy} \partial_{w_{yyy}}
+v\,\partial_v+v_x \partial_{v_x}+v_{xx} \partial_{v_{xx}}+v_{xxx} \partial_{v_{xxx}},\\
	w^2\partial_w \!+\! D_y(w^2) \partial_{w_y} \!+\! D_y^2(w^2)\partial_{w_{yy}}\!+\! D_y^3(w^2)\partial_{w_{yyy}}
\!-\!v^2\partial_v\! -\!D_x(v^2) \partial_{v_x}\! -\! D_x^2(v^2)\partial_{v_{xx}}\!-\!D_x^3(v^2)\partial_{v_{xxx}}\}.
\end{aligned}
\EqTag{prG2}
\end{equation}
This infinitesimal action satisfies the transversality hypothesis of Theorem \StRef{Dreduce} with respect to the standard contact structure on $J^3(\real,\real) \times J^3(\real,\real)$. Thus  $\CalI_2^{[1]}=(\CalK_1+\CalK_2)/\Gamma_{G_2}$ on $(J^3(\real,\real) \times J^3(\real,\real))/\Gamma_{G_2}$ is Darboux integrable.

The projections  $\Gamma_{G_2,a}=\rho_a(\Gamma_{G_2})$ of $\Gamma_{G_2}$ from equation \EqRef{prG2} to $J^3(\real,\real)$ are the prolongations of the standard action of ${\mathfrak s\mathfrak l}(2,\real)$ on the dependent variable; for example,
\begin{equation*}
\Gamma_{G_2,2}= \spn\{ \partial_v, v \partial_v+v_x \partial_v+ v_{xx} \partial_{v_{xx}}+ v_{xxx} \partial_{v_{xxx}}, v^2 \partial_v+ D_x(v^2) \partial_v+ D_x^2(v^2) \partial_{v_{xx}}+ D_x^3(v^2) \partial_{v_{xxx}} \}.
\end{equation*}
We now compute the intermediate integrals for $\CalI_2$. For the algebra $\Gamma_{G_2,2}$ acting on $J^3(\real,\real)$, two functionally independent  differential invariants  are well-known to be 
\begin{equation}
\hJ^1=x, \quad \hJ^2 = \frac{2v_xv_{xxx} -  3v_{xx}^2}{2v_x^2}.
\EqTag{Schw}
\end{equation}
In order to explicitly write the map $\tbfp_{\Gamma_{G_2,2}}: N_2 \to M_2/\Gamma_{G_2,2}$
we need to express the invariants in \EqRef{Schw} in terms of the invariants of $\Gamma_{G_2}$ on $J^3(\real,\real)\times J^3(\real,\real)$. The $\Gamma_{G_2}$ invariants we need are computed from the invariants in equation \EqRef{ptou2} to be
\begin{equation}
g^4_2 =  D_x\left(\log \frac{2v_xw_y}{(v+w)^2}\right)= \frac{v_{xx}}{v_x} - 2 \frac{v_x}{v+w},\quad g^5_2 =D_x(g^3_2)= \frac{v_{xxx}}{v_x} - \frac{v_{xx}^2}{v_x^2} -2 \frac{v_{xx}}{v+w}+ \frac{v_x^2}{(v+w)^2}.
\EqTag{G2invs}
\end{equation}
Therefore, from \EqRef{Schw} and \EqRef{G2invs}, we obtain $\pi_2^*(\hJ^2)={D_x}^2(g^3_2)-\frac{1}{2} (g^4_2)^2$. By equation \EqRef{cfpL1}, the map $\tbfp_{\Gamma_{G_2,2}}:N_2 \to M_2/\Gamma_{G_2,2}$ is then 
\begin{equation}
\tbfp_{\Gamma_{G_2,2}}= [\, \hf^1=x,\, \hf^2 = u_{2xx} - \frac{1}{2} u_{2x}^2\, ].
\EqTag{Ff3}
\end{equation}

Therefore by equation \EqRef{infFI}, $\hI^1=\tbfp_{\Gamma_{G_2,2}}^*(\hf^1)=x$ and $\hI^2=\tbfp_{\Gamma_{G_2,2}}^*(\hf^2)=u_{2xx} - \dfrac{1}{2}u_{2x}^2$ are two independent intermediate integrals for the Liouville equation $u_{2 xy} = \exp(u_2)$.  Again, the function $\hI^2$ can be
checked to be an intermediate integral in the traditional way by noting that
$\dfrac{d\ }{dy} (u_{2xx} - \dfrac{1}{2}u_{2x}^2)=0$, whenever $u_2$ is a solution to the Liouville equation.  Similarly, the independent intermediate integrals $\cI^1=y$ and $\cI^2=u_{2yy} - \dfrac{1}{2}u_{2y}^2$ are produced from the invariants of $\Gamma_{G_2,1}$.


These calculations underscore our earlier remarks that the intermediate integrals for a Darboux integrable PDE, where $\CalK_1$ and $\CalK_2$ are contact systems on jet space,  can {\bf always} be expressed in terms of classical differential invariants. In the case of the Liouville equation, the differential invariant  $\hJ^2$ in \EqRef{Schw} is just the Schwartzian derivative of $v$.  
\end{Example}

\begin{Example} \StTag{ExA2p} We compute the intermediate integrals for $\CalI_{3}$ and $\CalI_{4}$ from Example \StRef{ExA2} using equation \EqRef{infFI} in Section \StRef{Firstint}. We start with $\CalI_4$ and the symmetry algebra   $\Gamma_{G_4}=\spn\{X_1,X_2,Z_4\}$ given by equations \EqRef{ExA2} and \EqRef{defZ3Z4}. Again, as in the  case of $\Gamma_{G_2}$ in Example \StRef{ExA1p}, we need the prolongation of $\Gamma_{G_4}$ to $J^3(\real,\real)\times J^3(\real,\real)$ in order to satisfy the transversality condition in Theorem \StRef{Dreduce}.  One of the projections  (equation \EqRef{AAS})  of  $\Gamma_{G_4}$  to $J^3(\real,\real)$ is
$$
\Gamma_{G_4,2} = \spn \{\,  \partial_v, \,   v \partial_v+v_x \partial_{v_x}+ v_{xx} \partial_{v_{xx}}+ v_{xxx} \partial_{v_{xxx}},\, x\partial_v+\partial_{v_x}\, \}.
$$
The differential invariants of $\Gamma_{G_4,2}$ acting on $J^3(\real,\real)$ give  $\hf^1=\hJ^1=x$ and $\hf^2=\hJ^2= v_{xxx}(v_{xx})^{-1} $. Equation \EqRef{u3u4} and its derivatives then leads to
\begin{equation}
\hI^1=\bfp_{\Gamma_{G_4,2}}^* \hf^1=x, \quad \hI^2=\bfp_{\Gamma_{G_4,2}}^* \hf^2 = \frac{u_{4xx}}{u_{4x}}+\frac{1-2 u_{4x}}{u_4-x}.
\EqTag{Dinv}
\end{equation}
This produces, by Section \StRef{Firstint}, two independent first integrals for $\hV$. The $y$ analogues of the intermediate integrals in equation \EqRef{Dinv} are also intermediate integrals.

The  projection $\Gamma_{G_3,1} $ for the Lie algebra of vector fields $\Gamma_{G_3} = \spn\{ X_1,X_2,Z_3\}$ defined  by equations \EqRef{ExA2} and \EqRef{defZ3Z4}  does not satisfy the transversality condition for the rank  2 contact Pfaffian system $\CalK_1$ on $J^2(\real,\real)$. We therefore prolong the action of $\Gamma_{G_3}$  to $J^3(\real,\real)\times J^3(\real,\real)$ and again call  these vector fields $\Gamma_{G_3}$. 
The projections of this prolonged action to the factors of $J^3(\real,\real)\times J^3(\real,\real)$ are then
\begin{equation}
\begin{aligned}
\Gamma_{G_3,1} & = \spn\{\ \partial_w,\, w \partial_w+w_y \partial_{w_y}+ w_{yy} \partial_{w_{yy}}+ w_{yyy} \partial_{w_{yyy}},\, y \partial_w + \partial_{w_y}\   \} \quad {\text{and}} \\
\Gamma_{G_3,2} & = \spn\{\ \partial_v,\,  v \partial_v+v_x \partial_{v_x}+ v_{xx} \partial_{v_{xx}}+  v_{xxx} \partial_{v_{xxx}} \}.
\end{aligned}
\EqTag{prG3}
\end{equation}

The algebra $\Gamma_{G_3,2}$ in equation \EqRef{prG3} has 2-dimensional orbits on $J^3(\real,\real)$
and so, by Section \StRef{Firstint}, the singular system $\cV_3$ has 3  first independent first integrals.
For $\Gamma_{G_3,2}$ acting on $J^3(\real,\real)$ we obtain
$$
\bfq_{\Gamma_{G_3,2}} =[\hf^1=x,\   \hf_2= \frac{v_{xx}}{v_x}, \, \hf^3=\frac{v_{xxx}}{v_{xx}}].
$$
Combining this with equation \EqRef{u3u4}  and its derivatives, we find
\begin{equation*}
\hI^1=\bfp_{\Gamma_{G_3,2}}^* \hf^1=x,\  \hI^2=\bfp_{\Gamma_{G_3,2}}^* \hf^2= -\frac{u_{3x}}{u_3-x}, \  \hI^3=\bfp_{\Gamma_{G_3,2}}^*\hf^3=\frac{u_{3xx}}{u_{3x}} +\frac{1-2 u_{3x}}{u_3-x}.
\end{equation*}
These are the 3 independent first integrals  for $\hV$ (or intermediate integrals for the $u_3$ equation in \EqRef{ExA5}).

Since (the prolongation of) $\Gamma_{G_3,2}$ has 3-dimensional orbits on $J^3(\real,\real)$
we deduce, again by Section \StRef{Firstint}, that the singular system  $\cV_3$ for $\CalI_3^{[1]}$ 
has 2 independent first integrals. 
The $\Gamma_{G_3,2}$ invariants $\cf^1=\cJ^1=y$ and $\cf^2=\cJ^2=w_{yyy}(w_{yy})^{-1}$, when
expressed in
terms  of  $u_3$ and its derivatives from equation \EqRef{projeq}, yield the coordinate form of the map $\tbfp_{\Gamma_{G_3,1}}$ and produce the following intermediate integrals 
\begin{equation*}
\cI^1=\bfp_{\Gamma_{G_3,1}}^*\cf^1=y\ , \quad \cI^2=\bfp_{\Gamma_{G_3,1}}^* \cf^2 =\frac{u_{3yy}}{u_{3y}}-\frac{1}{y}
\end{equation*}
for the $u_3$ equation in \EqRef{ExA5}.



Finally, we point out how Theorem 9.1 and Corollary 9.2 in \cite{anderson-fels:2014a} 
apply  to Examples \StRef{ExA1p} and \StRef{ExA2p}. For the B\"acklund transformation $\CalB^{[1]}$  constructed between the pairs $\CalI_i^{[1]}$ given in equation \EqRef{projeq},
we observe that the systems $\CalI^{[1]}_i$  arising from diagonal quotients have strictly fewer independent intermediate integrals than those constructed from non-diagonal quotients.

Finally, we point out how Theorem 9.1 and Corollary 9.2 in \cite{anderson-fels:2014a} 
apply  to Examples \StRef{ExA1p} and \StRef{ExA2p}.

\end{Example}

\begin{Example}  \StTag{ExBp} In Example \StRef{ExB} we computed the reduction of the Pffaffian system $\CalI$ defined by the pair of Monge equations in equation \EqRef{MA2}. The reductions $\CalI_1,\CalI_2$ and $\CalB$ in equation \EqRef{ExB3} were computed for the  symmetry algebras $\Gamma_{G_1} =\spn \{\,X_1, Y_1,  Z_2\}$,  
	$\Gamma_{G_2} = \spn \{\,Z_1, Z_2,  Z_3\}$ and  $\Gamma_H =\spn  \Gamma_{G_1} \cap \Gamma_{G_2} = \{\,Z_1, \, Z_2 \, \}$, where the vector fields are given in \EqRef{GammaExB}. We now compute the intermediate integrals for the Darboux integrable systems $\CalI_1, \CalI_2$ and $\CalB$.


From equation \EqRef{GammaExB} we see that the projection  (equation \EqRef{AAS}) is
$$
\Gamma_{G_1,1} =\rho_1(\Gamma_{G_1})= \spn \{ \ \partial_v, \  s\partial_v + \partial_{v_s}\  \}.
$$ 
The independent invariants of $\Gamma_{G_1,1}$ on $M_1=[s,u,v,v_s,v_{ss}]$ are $\cJ^1=s, \cJ^2=u$ and $\cJ^3=v_{ss}$. The combination of these with equation \EqRef{bfqExB} 
produces the map $\tbfp_{\Gamma_{G_1,1}}: N_1 \to M_1/\Gamma_{G_1,1}$ (equation \EqRef{qLc}) as
\begin{equation}
\tbfp_{\Gamma_{G_1,1}}=[\ \cf^1 = x_2, \ \cf^2 = x_6, \  \cf^3 = x_4 \ ].
\EqTag{bfp31}
\end{equation}
For the Darboux integrable system $\CalI_1=\CalI/\Gamma_{G_1}$ equation \EqRef{infFI} in   Section \StRef{Firstint},  in conjunction with equation \EqRef{bfp31}, gives the first integrals $\cI^1= x_1, \cI^2= x_6$, and $\cI^3=x^4$ of $\cV$.
Similarly, one deduces that 
\begin{equation}
\hI^1=x_1,\quad \hI^2=x_5  \quad {\rm and}\quad  \hI^3=x_3 
\EqTag{IIex2I2}
\end{equation}
are first integrals of $\hV$.  Thus $\CalI_1$ is a rank 3 hyperbolic Pfaffian system for which
each characteristic system admits 3 independent first integrals. By a classical theorem of  Lie, 
	this implies that $\CalI_1$ is (contact) equivalent to  the wave equation. The change of variables in equation \EqRef{ExB8} was obtained using this fact.

For $\CalI_2= \CalI/\Gamma_{G_2}$ on $N_2=M/\Gamma_{G_2}$ we have the projected Lie algebra of vector fields $\Gamma_{G_2,1}=\spn \{\,  \partial_v, \,  s\partial_v + \partial_{v_s},  \,  \partial_u\,   \}$ so that  $\bfq_{\Gamma_{G_2,2}}=[\cf^1=\cJ^1=s,\cf^2= \cJ^2= v_{ss}]$.
Using $\bfq_{\Gamma_{G_2,2}}$ and \EqRef{bfqExB}, we then determine that the map $\tbfp_{\Gamma_{G_2,1}}:N_2 \to M_1/\Gamma_{G_2,1}$ is given by (equation \EqRef{qLc})
$$
\tbfp_{\Gamma_{G_2,1}}=[\ \cf^1  = x_2, \ \cf^2 = x_4 \ ] .
$$
A similar computation with $\Gamma_{G_2,2}$ (using equation \EqRef{QINVS}) 
 yields the intermediate integrals
\begin{equation}
\cI^1=x_2, \  \cI^2= x_4,\quad {\rm and } \quad  \hI^1= x_1,\  \hI^2= x_3.
\EqTag{DINVSB}
\end{equation}

In the particular case where $F(v_{ss})=v_{ss}^2$ and $G(z_{tt})=z_{tt}^2$ in equation \EqRef{HCeq}, the change of variable \EqRef{ExB10} produces from \EqRef{DINVSB} the intermediate integrals 
\begin{equation*}
\cI^1 = U_{XY}+\frac{1}{U_{YY}}\, ,  \ \cI^2 = U_Y-( U_{XY}+\frac{1}{U_{YY}})X,\  {\rm and} \quad 
\hI^1= U_{XY}-\frac{1}{U_{YY}}\, , \ \hI^2= U_Y-(U_{XY}-\frac{1}{U_{YY}})X
\end{equation*}
found on page 130 in \cite{goursat:1897a}. To check that these are intermediate integrals, we compute
$$
\left( \frac{d \ } {dy}+{U^2_{YY}}\frac{d\ }{dx}\right) \cI^a = 0 \quad {\rm and} \quad\left( \frac{d \ } {dy}-{U^2_{YY}}\frac{d\ }{dx}\right) \hI^a = 0.
$$


Finally, for the Darboux integrable system $\CalB= (\CalK_1+\CalK_2)/\Gamma_H$ given in \EqRef{ExB3}, the intermediate integrals are 
\begin{equation}
\cI^1=x_2,\ \cI^2=x_6,\ \cI^3=x_4 \quad {\rm  and} \quad \hI^1=x_1,\ \hI^2=x_5,\ \hI^3=x_3.
\EqTag{IIB2}
\end{equation}
\end{Example}

\begin{Example}\StTag{ExD2}  In this example we use equation \EqRef{infFI} in Section \StRef{Firstint} to determine the intermediate integrals for the $A_2$ Toda system $\CalI_2$ in Example \StRef{ExD}.  The intermediate integrals for $\CalB$ will be computed in Example \StRef{ExD3} using a different approach. The hypothesis in Theorem \StRef{Dreduce} holds for $\Gamma_{G_2}$ after one prolongation and so we let $M=J^{4,5}(\real,\real^2)\times J^{4,5}(\real,\real^2)$ and let $\CalI$ be the corresponding contact system. We remind the reader we are working in non-standard coordinates on each $J^{4,5}(\real,\real)$. See, for example, equation \EqRef{DxDyNC}. 

For the algebra $\Gamma_{G_2}$ (equation \EqRef{GammaSL3}) the projection $\Gamma_{G_2,2}$  (equation \EqRef{AAS}) 
onto $M_2=J^{4,5}(\real,\real^2)$ is the prolongation of
$$
\Gamma_{G_2,2} =\rho_2(\Gamma_{G_2})= {\rm pr} \spn \{ 
\partial_u ,  \partial_v, u\,\partial_u , v\,\partial_u, u\,\partial_v ,  v\,\partial_v,u^2\partial_u  + uv\,\partial_v  ,uv\,\partial_u + v^2\partial_v  \}.
$$
The 4 independent differential invariants of $\Gamma_{G_2,2}$ on $J^{4,5}(\real,\real)$ lead to (equation \EqRef{qLc})
\begin{equation}
\bfq_{\Gamma_{G_2,2}}=[
\hf^1= x, \, \hf^2=\frac{u_1^2v_4}{v_2}-\frac{4u_1^2v_3^2}{3v_2^2}+2\frac{u_3}{u_1}-3\frac{u_2^2}{u_1^2} ,\,  \hf^3 = D_x(\hf^2) , \ \
\hf^4= \frac{u_1^3v_5}{v_2}-5\frac{u_1^3v_4v_3}{v_2^2}+\frac{40u_1^3v_3^3}{9v_2^3}].
\EqTag{fs41}
\end{equation}

The projection map $\bfq_{\Gamma_{G_2}}:M \to N_2=M/\Gamma_{G_2}$ is given by $D_x$  differentiation of $U^1=g_2^1$ and $U^2=g_2^2$ in equation \EqRef{32}. We then have, using the coordinates in equation \EqRef{fs41}, that the projection map $\tbfp_{\Gamma_{G_2,2}}:N_2 \to M_2/\Gamma_{G_2,2}$ is given by
\begin{equation}
\begin{aligned}
\tbfp_{\Gamma_{G_2,2}}  & =[\,\hf^1 = x, \quad \hf^2 = U^1_{xx} + U^2_{xx} -(U^1_x)^2+ U^1_xU^2_x - (U^2_x)^2, \quad \hf^3= D_x(f_2), \\
\hf^4 & =3(U^2_{xxx}-U^1_{xxx}) +(6U^1_x+3U^2_x)U^1_{xx}-(3U^1_x+6U^2_x)U^2_{xx}+6U^1_x(U^2_x)^2-6U^2_x(U^1_{x})^2\, ].
\end{aligned}
\EqTag{DinvToda}
\end{equation}
The  vector field $D_x$ in equation \EqRef{DxDyNC} must be prolonged to compute $U^i_{xxx}$ in \EqRef{DinvToda}. The $y$ total derivative of the functions in equation \EqRef{DinvToda} vanish subject to the Toda equations \EqRef{95}, which verifies directly that they are intermediate integrals. Similar expressions for the first integrals of $\cV$ can be obtained from equation \EqRef{DinvToda} by changing $x$ to $y$. 

The system $\CalI_1$ in this example is the uncoupled wave and Liouville equation each of whose intermediate integrals were given in Example \StRef{ExA1p}.  
\end{Example}

\begin{Example}\StTag{ExE2} In this example we compute the intermediate integrals for $\CalI_2$ and $\CalI_1$ in Example \StRef{ExE}.  We need to prolong the action of $\Gamma_{G_2}$ for  Theorem \StRef{Dreduce} to apply. Therefore we work with the prolongation of $\Gamma_{G_2}$ in equations \EqRef{GammaExE} and \EqRef{G1ExE} to $M=J^3(\real,\real)\times J^2(\real^2,\real)$. The projected algebras from equation \EqRef{AAS} are
$$
\begin{aligned}
\Gamma_{G_2,1} &= \spn\{ \partial_w, w\partial_w+w_z\partial_{w_z}+w_{zz}\partial_{w_{zz}}+w_{zzz}\partial_{w_{zzz}},\\
&\qquad \qquad w^2 \partial_w+ 2ww_z\partial_{w_z}+2(w_z^2+ww_{zz})\partial_{w_{zz}}+(6w_zw_{zz}+2ww_{zzz})\partial_{w_{wzzz}} \} \quad {\rm and} \\
\Gamma_{G_2,2} &= \spn\{ 
\partial_v, v\partial_v+v_x\partial_{v_x}+v_{y}\partial_{v_{y}}+v_{xx}\partial_{v_{xx}}+v_{xy}\partial_{v_{xy}}+y_{yy}\partial_{v_{yy}}, \\ 
\qquad & v^2 \partial_v+2vv_x\partial_{v_x}+2vv_y\partial_{v_y} + 2(v_x^2+vv_{xx})\partial_{v_{xx}}+ 2(v_xv_y+vv_{xy})\partial_{v_{xy}}+2(v_y^2+vv_{yy})\partial_{v_{yy}}  \}.
\end{aligned}
$$
The differential invariants of $\Gamma_{G_2,2}$ on $J^2(\real^2,\real)$ lead to (see equation \EqRef{qLc}),
\begin{equation}
\bfq_{\Gamma_{G_2,2}}=[\,\hf^1=x,\, \hf^2=y,\, \hf^3 = \frac{v_y}{v_x},\, \hf^4=D_x(\hf^3)=\frac{v_xv_{xy}-v_{xx}v_y}{v_x^2},\, \hf^5=D_y(\hf^3)=\frac{v_xv_{yy}-v_{xy}v_y}{v_x^2}\, ].
\EqTag{fexe}
\end{equation}
The map $\bfq_{\Gamma_{G_2}}:M \to N_2$, obtained from
one prolongation of equation \EqRef{bfGg2}, is 
\begin{equation}
\begin{aligned}
\bfq_{\Gamma_{G_2}}& =[\ 
 x= x,\, y=y,\, z=z,\, u=g^4_2= \frac{w_{zz}}{w_z} + \frac{2w_z}{v-w}, \,
u_x= D_x(g^4_2), u_y= D_y(g^4_2),\\
 u_z&= D_z(g^4_2)=\frac{w_zw_{zzz}-w_{zz}^2}{w_z^2} + 2\frac{(v-w)w_{zz}+w_{z}^2}{(v-w)^2}, \,
u_{xx}={D_x}^2(g^4_2)= \frac{4w_zv_x^2-2(v-w)w_zv_{xx}}{(v-w)^3},  \\
u_{xy}&=D_x(D_y(g^4_2))=\frac{4w_zv_xv_{y}-2(v-w)w_zv_{xy}}{(v-w)^2},\,
u_{yy}={D_y}^2(g^4_2) = \frac{4w_zv_y^2-2(v-w)w_zv_{yy}}{(v-w)^2} \, ].
\end{aligned}
\EqTag{CExEG2}
\end{equation}
The map $\tbfp_{\Gamma_{G_2,2}}:N_2\to M_2/\Gamma_{G_2,2}$ is then given in coordinates used in \EqRef{fexe} and \EqRef{CExEG2} by
\begin{equation}
\tbfp_{\Gamma_{G_2,2}}=[\, \hf^1=x,\, \hf^2=y,\, \hf^3 = \frac{u_y}{u_x},\, \hf^4=D_x(\frac{u_y}{u_x}) , \hf^5=D_y(\frac{u_y}{u_x})\, ].
\EqTag{fsExE}
\end{equation}
Again, by equation \EqRef{infFI}, the functions appearing in \EqRef{fsExE} will be a complete set of first integrals for $\hV$. For example, $u_y(u_x)^{-1}$ is  easily checked to be an intermediate integral of equations \EqRef{OverDS} by computing
$$
\frac{d\ }{dz}( \frac{u_y}{u_x}) = \frac{u_{yz}u_x-u_y u_{xz}}{u_x^2}=0.
$$

To find the first integrals for $\cV$ we need the $\Gamma_{G_2,1}$ differential invariants on $J^3(\real,\real)$.  A maximal independent set are easily found to be
$$
\cJ^1=z \quad {\rm and} \quad  \cJ^2=\frac{2w_zw_{zzz} -  3w_{zz}^2}{2w_z^2},
$$
where the second invariant is the Schwartzian derivative of $w$. In terms of the coordinates on $M/\Gamma_{G_2}$ in equation \EqRef{CExEG2} these invariants lead to
$$
\bfp_{\Gamma_{G_2,2}}=[ \ \cf^1= z,  \ \cf^2 = u_z-\frac{1}{2}u^2 \ ].
$$
Again, we can check that this produces an intermediate integral for equations \EqRef{OverDS} by computing
$$
\frac{d\ }{dx}(u_z-\frac{1}{2}u^2)= u_{zx}-uu_x=0, \ {\rm and} \quad \frac{d\ }{dy}( u_z-\frac{1}{2}u^2)= u_{zy}-uu_y=0.
$$

Finally, we compute the intermediate integrals for $\CalI_1$. With $\Gamma_{G_1}$ given by equation \EqRef{G1ExE}  the projected algebras (equation \EqRef{AAS}) are
$$
\begin{aligned}
\Gamma_{G_1,1} &= \spn\{ \partial_w, w\partial_w+w_z\partial_{w_z}+w_{zz}\partial_{w_{zz}}+w_{zzz} \partial_{w_{zzz}} \} \quad {\rm and}\\
\Gamma_{G_1,2} &= \spn\{\partial_v,  v\partial_v+v_x\partial_{v_x}+v_{y}\partial_{v_{y}} +v_{xx} \partial_{v_{xx}} + v_{xy}\partial_{v_{xy}}+ v_{yy} \partial_{v_{yy}}\}.
\end{aligned}
$$
The $\Gamma_{G_1,2} $  lowest order invariants are
$$
\hJ^1= x,\quad \hJ^2=y \quad {\rm and} \quad \hJ^3= \log v_y-\log v_x
$$
which, when expressed in terms of the coordinates on $N_1=M/\Gamma_{G_1}$ in \EqRef{ODSbfp1}, give
$$
\bfp_{\Gamma_{G_1,2}}=[\ \hf^ 1=x,\ \hf^2=y,\ \hf^3 = P-Q\ ].
$$
The function $P-Q$ can be checked directly to be an intermediate integral of equations \EqRef{I1f} by noting that its  $z$ total derivative vanishes.

Similarly, the lowest order invariants of $\Gamma_{G_1,1}$ are
$$
\cJ^1= z\quad {\rm and} \quad \cJ^2=\frac{w_{zz}}{w_z}
$$
which, when expressed in terms of the coordinates in equation \EqRef{ODSbfp1} for  $M/\Gamma_{G_1}$, give
$$
\bfp_{\Gamma_{G_1,1}}=[\ \cf^1=z, \quad \cf^2 = P_z \ ].
$$
The function $P_z$ can easily be  checked to be an intermediate integral of equations \EqRef{I1f} by
noting that both its $x$ and $y$ total derivative vanish.  
\end{Example}

\subsection{The Vessiot Algebra}\StTag{CVA}

In this subsection we demonstrate Theorem \StRef{Rdiag}  and compute the Vessiot algebra and canonical quotient representation for the examples. For all 
the reductions  $\CalB=(\CalK_1+\CalK_2)/\Gamma_H$ and $\CalI_2=(\CalK_1+\CalK_2)/\Gamma_{G_2}$ the Lie algebras of vector-fields $\Gamma_H$ and $\Gamma_{G_2}$ 
are diagonal.  Therefore the Vessiot algebra for the Darboux integrable systems $\CalB$ and $\CalI_2$ are $\Gamma_H$ and $\Gamma_{G_2}$ respectively.

However in each of our examples the Lie algebra  $\Gamma_{G_1}$ (and $G_3$ from Example \StRef{ExA2p}) is {\bf not} a diagonal action and so  $\Gamma_{G_1}$ is not isomorphic to the Vessiot algebra for the Darboux integrable system $\CalI_1=(\CalK_1 +\CalK_2)/\Gamma_{G_1}$. Consequently we need to utilize Theorem \StRef{Rdiag} to find the local canonical quotient representation and Vessiot algebra for these cases. For the infinitesimal version of equation \EqRef{AAS} we define the ideals $\Gamma_{A_a} \subset \Gamma_{L_a}$ by
\begin{equation}
\Gamma_{A_1}=\rho_1(\ker \rho_{2} \cap \Gamma_{L}),\ {\rm and} \quad \Gamma_{A_2}=
\rho_2(\ker \rho_{1} \cap \Gamma_{L_1}).
\EqTag{GammaA}
\end{equation}
The quotient maps $\bfq_{\Gamma_{A_a}}:M_a \to M_a/\Gamma_{A_a}$ 
can then be used to define the Lie algebra of vector fields 
\begin{equation}
\widetilde \Gamma_{L,\diag} = (\bfq_{\Gamma_{A_1}}\times \bfq_{\Gamma_{A_2}})_*(\Gamma_{L})
\EqTag{redL}
\end{equation}
on $M_1/\Gamma_{A_1} \times M_2/\Gamma_{A_2}$. Then
\begin{equation}
\CalI= \left(   \CalK_1/\Gamma_{A_1}+\CalK_2/\Gamma_{A_2} \right) / \widetilde \Gamma_{L,\diag}
\EqTag{infqr}
\end{equation}
is the local canonical quotient representation of $\CalI/\Gamma_L$ and $\widetilde \Gamma_{L,\diag}$ is (isomorphic to) the Vessiot algebra of $\CalI/\Gamma_L$. 
As an abstract algebra $\widetilde \Gamma_{L,\diag}$ is isomorphic to  $\Gamma_{L}/(\Gamma_{A_1} + \Gamma_{A_2})$.

\begin{Example}\StTag{ExA2f}  We begin by demonstrating equation \EqRef{infqr} using Example \StRef{ExA2p} where $\CalI^{[1]}_3$ is the Darboux integrable  rank 3 Pfaffian system for the PDE $u_{3xy}=(u_3-x)^{-1}(u_{3x} u_{3y}) $ on the 7-manifold $N_3$. We have 
$$
\CalI^{[1]}_3= (\CalK_1 + \CalK_2 )/\Gamma_{G_3} \quad {\rm and}  \quad N_3 = (J^3(\real,\real)\times J^3(\real,\real)) /\Gamma_{G_3},
$$
where the Lie algebra $\Gamma_{G_3}$ on $M=J^3(\real,\real)\times J^3(\real,\real)$ is
\begin{equation}
\Gamma_{G_3} \!=\! \spn\{\,  \partial_w+\partial_v,\, y \partial_w + \partial_{w_y},\, w \partial_w+w_y \partial_{w_y}\!+ w_{yy} \partial_{w_{yy}}\! +  w_{yyy} \partial_{w_{yyy}} +
 v \partial_v+v_x \partial_{v_x}\!+ v_{xx} \partial_{v_{xx}}\!  + v_{xxx} \partial_{v_{xxx}} \}.
\EqTag{prG32}
\end{equation}
Since the projections $\rho_a(\Gamma_{G_3})$, $a=1,2$, are not isomorphic, $\Gamma_{G_3}$ is not a diagonal action. The subalgebras defined in \EqRef{GammaA} for $\Gamma_{G_3}$ are easily found to be
$$
\Gamma_{A_1} = \{\, y \partial_w + \partial_{w_y}\, \}, \quad \text{ and} \quad \Gamma_{A_2} = \{\, 0\, \}.
$$
The quotient map $\bfq_{\Gamma_{A_1}}: J^3(\real,\real) \to J^3(\real,\real)/\Gamma_{A_1}$ is
found, by  computing the $\Gamma_{A_1}$ invariants, to be
$$
\bfq_{\Gamma_{A_1}}=[\, y= y ,\, \tilde w=w-yw_y,\, \tilde w_y=w_{yy},\, \tilde w_{yy}=w_{yyy}\, ],
$$
from which we deduce that
\begin{equation}
\widetilde \CalK_1=\CalK_1/\Gamma_{A_1}= \langle \ d \tilde w - \tilde w_y dy\, , \ d\tilde w_y-\tilde w_{yy} dy\, , \ d\tilde w_{yy}\wedge dy \quad  \rangle_{\text{alg}}. 
\EqTag{QtG3}
\end{equation}
Note that $\widetilde \CalK_1$ is the contact system on $J^2(\real,\real)$. The map  $\bfq_{\Gamma_{A_1}}\times \bfq_{\Gamma_{A_2}}$ in \EqRef{redL} is given by
$$
\begin{aligned}
\bfq_{\Gamma_{A_1}}\times \bfq_{\Gamma_{A_2}}=[ & y= y ,\, \tilde w=w-yw_y,\, \tilde w_y=w_{yy},\, \tilde w_{yy}=w_{yyy},\\ & x=x,\, v=v,\, v_x=v_x,\, v_{xx}=v_{xx},\, v_{xxx}=v_{xxx}\, ].
\end{aligned}
$$
From this we can compute $(\bfq_{\Gamma_{A_1}}\times \bfq_{\Gamma_{A_2}})_*(\Gamma_{G_3})$ using
$$
\begin{aligned}
& (\bfq_{\Gamma_{A_1}}\times \bfq_{\Gamma_{A_2}})_*(\partial_w+\partial_v) =\partial_{\tilde w}+\partial_v \ , \quad (\bfq_{\Gamma_{A_1}}\times \bfq_{\Gamma_{A_2}})_*( y \partial_w + \partial_{w_y}) =0 \\
& (\bfq_{\Gamma_{A_1}}\times \bfq_{\Gamma_{A_2}})_*(w \partial_w+w_y \partial_{w_y}\!+ w_{yy} \partial_{w_{yy}}\! +  w_{yyy} \partial_{w_{yyy}} +
 v \partial_v+v_x \partial_{v_x}\!+ v_{xx} \partial_{v_{xx}}\!  + v_{xxx} \partial_{v_{xxx}}) \\
& \qquad = \partial_{\tilde w} + \tilde w_y \partial_{\tilde w_y} + \tilde w_{yy} \partial_{\tilde w_{yy}}+v\,\partial_v+v_x \partial_{v_x}+v_{xx} \partial_{v_{xx}} +v_{xxx}\partial_{v_{xxx}} .
\end{aligned}
$$
Thus for this example,  the algebra \EqRef{redL} is
\begin{equation}
\widetilde \Gamma_{G_3,\diag}=\spn\{\,  \partial_{\tilde w} + \partial_v, \, {\tilde w}\,\partial_{\tilde w} + \tilde w_y \partial_{\tilde w_y} + \tilde w_{yy} \partial_{\tilde w_{yy}}+v\,\partial_v+v_x \partial_{v_x}+v_{xx} \partial_{v_{xx}} +v_{xxx}\partial_{v_{xxx}} \, \} .
\EqTag{tG3}
\end{equation}
Using the data from equations \EqRef{QtG3}, \EqRef{tG3}, and the fact that $\CalK_2/\Gamma_{A_2}=\CalK_2$,  equation \EqRef{infqr}
produces the canonical quotient representation of $\CalI_3^{[1]}$ and shows that the Vessiot algebra of $\CalI_3^{[1]}$ is the 2-dimensional non-Abelian Lie algebra. 

In Examples \StRef{ExA2p} and \StRef{ExBp}  the quotient $\CalI^{[1]}_1=M/\Gamma_{G_1}$ is the rank 3 Pfaffian system  on a 7-manifold representing the wave equation. The reduction in equation \EqRef{infqr}  in both of these case gives the canonical quotient representation as
$$
\CalI^{[1]}_1 = (J^2(\real,\real)\times J^2(\real,\real), \widetilde \CalK_1 +\widetilde \CalK_2)/ \widetilde \Gamma_{G_1,\diag}
$$
where $\widetilde \CalK_a$ are the contact systems, and $ \widetilde \Gamma_{G_1,\diag} = \spn \{ \, \partial_{\tilde w} + \partial_{\tilde v} \, \}$ where $\tilde w$ and $\tilde v$ are the dependent variable. 
\end{Example}	

\begin{Remark}  The Vessiot algebra of $\CalB$ 
in Examples \StRef{ExA1} and \StRef{ExA2} is the 2-dimensional solvable Lie algebra $\Gamma_H$ used to construct $\CalB$. Theorem 7.5 in \cite{anderson-fels:2014a} shows that this is isomorphic to the Vessiot algebra of the prolongation $\CalB^{[1]}$ of $\CalB$.  Theorem  7.3 in  \cite{anderson-fels:2014a} implies that there is a monomorphism from the Vessiot algebra of $\CalB^{[1]}$  to each of the Vessiot algebras of the Darboux integrable systems  $ \CalI^{[1]}_2$, $\CalI_4^{[1]}$ and $\CalI_3^{[1]}$.  In the case of $\CalI_3^{[1]}$ this monomorphism is a isomorphism by dimensional reasons. This same theorem shows that the induced homomorphism from the Vessiot algebra of $\CalB^{[1]}$ to the Vessiot algebra of $\CalI_1^{[1]}$  is {\it not} injective.
\end{Remark}


\begin{Example}  Recall that the system $\CalI_1$ in Example \StRef{ExD} is the rank 6 Pfaffian system representation of the decoupled wave Liouville system in equation \EqRef{52}. 
We compute the Vessiot algebra for the (non-prolonged) Darboux integrable system $\CalI_1$ using the isomorphism $\widetilde \Gamma_{L,\diag} \cong \Gamma_{L}/(\Gamma_{A_1} + \Gamma_{A_2})$ for $\Gamma_L=\Gamma_{G_1}$. 

We begin by finding $\Gamma_{A_1} + \Gamma_{A_2}\subset \Gamma_{G_1}$  using $\Gamma_{G_1}$ in
equation \EqRef{GammaSL3}.   First we note that
$$
\ker \rho_2 = \spn \{\, X_1-Z_3= 2\partial_w, \, X_2-Z_4 =2 \partial_z\, \}, \quad \ker \rho_1=\spn \{\, X_1+Z_3=2\partial_u,\, X_2+Z_4 =2\partial_v \, \}.
$$
and therefore 
\begin{equation}
\Gamma_{A_1} + \Gamma_{A_2} = \spn \{ \, X_1, \, X_2, \, Z_3,\, Z_4 \, \}.
\EqTag{G1id}
\end{equation}
The Vessiot algebra is the computed using the algebra \EqRef{GammaSL3} and the ideal \EqRef{G1id} to be
$$
\Gamma_{G_1}/(\Gamma_{A_1} + \Gamma_{A_2}) \cong \real \oplus {\mathfrak s \mathfrak l}(2,\real).
$$
As expected, this algebra is just the  direct sum of the Vessiot algebra for the wave equation and the Vessiot algebra of the Liouville equation.
\end{Example}

\begin{Example} In this last example we compute the Vessiot algebra and the canonical quotient representation for the system $\CalI^{[1]}_1$ in Example \StRef{ExE2}. We find, from the algebra $\Gamma_{G_1}$ in \EqRef{G1ExE}, that
$$
\Gamma_{A_1}=\spn \{\, \partial_w \, \} \quad \text{and} \quad \Gamma_{A_2}=\spn \{ \, \partial_v \, \}.
$$
The quotient systems and manifolds in equation \EqRef{infqr}  are
\begin{equation}
\begin{aligned}
\CalK_1/\Gamma_{A_1} &=\langle dw_z-w_{zz}dz,dw_{zz}-w_{zzz}dz  \rangle_{\text{diff}} ,\  && M_1/\Gamma_{A_1}=(z,w_z,w_{zz},w_{zzz}),\\
\CalK_2/\Gamma_{A_2} &=\langle dv_x-v_{xx}dx-v_{xy}dy, dv_y-v_{xy}dx-v_{yy}dy \rangle_{\text{diff}} ,&& M_2/\Gamma_{A_2}=(x,y,v_x,v_y,v_{xx},v_{xy},v_{yy}),
\end{aligned}
\EqTag{CGE2}
\end{equation}
and the projected algebra  on $M_1/\Gamma_{A_1} \times M_2/\Gamma_{A_2}$ is
\begin{equation}
\widetilde \Gamma_{G_1,\diag} = \spn \{ w_z\partial_{w_z}+w_{zz}\partial_{w_{zz}} + w_{zzz}\partial_{w_{zzz}}  +v_x\partial_{v_x}+v_{y}\partial_{v_{y}}+v_{xx}\partial_{v_{xx}}+v_{xy}\partial_{v_{xy}}+y_{yy}\partial_{v_{yy}} \}.
\EqTag{RG1}
\end{equation}
Equation \EqRef{infqr}, applied to \EqRef{RG1} with \EqRef{CGE2},  provides the canonical quotient representation for $\CalI^{[1]}_1$. The Vessiot algebra is $\widetilde \Gamma_{G_1,\diag}$. It is interesting to note that $\CalK_2/\Gamma_{A_2}$ in equation \EqRef{CGE2} is 
the rank 2 Pfaffian system for the single PDE $ \alpha_y = \beta _x$ for two functions $(\alpha,\beta)$ of two variables $(x,y)$.
\end{Example}




\section{Maximally Compatible Integrable Extensions of Darboux Integrable Systems}\StTag{IEDI}

In this section we examine the relationship between the intermediate integrals of a Darboux integrable system $\CalI$ and an integrable extension $\CalE$ of $\CalI$. Ultimately we identify a condition on the extension, which we call {\deffont maximal compatibility}, that can be used to determine whether a given B\"acklund transformation between Darboux integrable systems can be constructed by Theorem A.

In Section \StRef{MA} we apply this theory to the case when  $\CalB$ is a B\"acklund transformation between two Darboux integrable systems $\CalI_a$ as considered in \cite{Clelland-Ivey:2009a}. 
By utilizing the double fibration in the B\"acklund transformation and applying Theorem \StRef{DIext} below to {\it both} $\CalI_1$ and $\CalI_2$ we are able to determine the intermediate integrals of $\CalB$. We then show that the prolongation of $\CalB$  is a maximally compatible extension of the prolongation of $\CalI_2$.


Our main theorem on integrable extensions is the following.

\begin{Theorem}
\StTag{DIext} 
	Let $\bfp : (\CalE, N) \to (\CalI, M)$ be an integrable extension with  $J$ an admissible subbundle of $T^*N$ for $(\CalE, \CalI)$.

\smallskip
\noindent
{\bf [i]}  If  $\CalI$ is decomposable of type $[p, \rho]$ with singular Pfaffian systems $\hV$ and $\cV$, 
	then $\CalE$ is decomposable of type $[p, \rho]$ with singular Pfaffian systems  
\begin{equation}
	\hZ = J \oplus \bfp^* ( \hV) \quad\text{and}\quad \cZ = J \oplus \bfp^* (\cV).
\EqTag{singWJV}
\end{equation}
{\bf [ii]}  If $\CalI$ is Darboux integrable and $(E^1)^\infty =0$,
	then $\CalE$ is Darboux integrable.
\end{Theorem}

This produces the following useful corollary (see Example \StRef{ExD3}).

\begin{Corollary} \StTag{PBFI} Let $\bfp : (\CalE, N) \to (\CalI, M)$ be an integrable extension of a decomposable systems $\CalI$ and $\CalE$ where \EqRef{singWJV} holds. Then the intermediate integrals of $\CalI$ pullback to intermediate integrals of $\CalE$, that is, 
 if $f\in C^\infty(M)$ is a first integral of $\hV$ (or $\cV$) then $f\circ \bfp \in C^\infty(N)$ is a first integral of $\hZ$ (or $\cZ$).
\end{Corollary}

	The equations in  \EqRef{singWJV} imply that  $\hZ$ is an integrable extension of $\hV$ and $\cZ$ is
an integral extension of $\cV$.  Equations 2.8 in {\bf IE [v]} from \cite{anderson-fels:2014a}
refine this observation to give the following  bounds on the number of independent first integrals
	for the singular Pfaffian systems $\hZ$ and $\cZ$ for the extension $\CalE$ in terms of the corresponding integrals of $\hV$ and $\cV$,
\begin{equation}
\begin{aligned}
\rank \hV^\infty =	\rank(\bfp^*(\hV^\infty))  & \leq  \rank( \hZ^\infty)  \leq \rank (\bfp^*(\hV^\infty))    + \rank (\ker \bfp_*),\quad \text{and}\\
\rank \cV^\infty =	\rank(\bfp^*(\cV^\infty))  & \leq   \rank( \cZ^\infty)  \leq \rank (\bfp^*(\cV^\infty))    +\rank (\ker \bfp_*). 
\end{aligned}
\EqTag{infbds}
\end{equation}
The left-hand side of these inequalities follow from Corollary \StRef{PBFI}.  Note that if $J$ is an admissible subbundle for the extension $\CalE$, then $\rank (\ker \bfp_*)= \rank(J)$.

Theorem 6.5 in \cite{anderson-fels:2014a} shows that if $(\CalK_1+\CalK_2)/G_{\diag} = \CalI$ is the
canonical quotient representation of a Darboux integrable system $\CalI$ and $H \subset G$ a subgroup, then the intermediate integrals for the extension $\CalE = (\CalK_1+\CalK_2)/H_{\diag}$ of $\CalI$ satisfy the maximum conditions in equation \EqRef{infbds}. Conversely, the requirement that the space of intermediate integrals for the extension $\CalE$ achieve the maximums in equation \EqRef{infbds} provide a sufficiency test
which guarantee that an extension arises locally by a group quotient (see Theorem 8.2 in \cite{anderson-fels:2014a}). This leads to the following definition.

\begin{Definition}\StTag{MaxCT}
	Let   $\bfp\:  (\CalE, N) \! \to\! (\CalI, M)$ be an integrable extension of Darboux integrable  systems  with singular Pfaffian systems satisfying \EqRef{singWJV}. The extension is called {\deffont maximally compatible} if 
\begin{equation*}
\begin{aligned}
& {\rm {\bf [i]}} & \ker( \bfp_*) \cap \ann(\hZ^\infty) =0 ,\quad & \quad  \ker (\bfp_*) \cap \ann(\cZ^\infty) = 0 \quad {\text \rm and} \\
& {\rm {\bf [ii]}} &  \rank (\hZ^\infty) = \rank (\ker \bfp_*) + \rank (\hV^\infty),  \quad &\quad \rank (\cZ^\infty) = \rank (\ker(\bfp_*)) + \rank (\cV^\infty). &
\end{aligned}
\end{equation*}
\end{Definition}
Condition {\bf [ii]} implies that the values $\rank( \hZ^\infty)$ and $\rank( \cZ^\infty)$ in equation \EqRef{infbds} are maximal. A more general definition  of maximal compatibility when $\CalI$ is not Darboux integrable is given by Definition 5.2 in \cite{anderson-fels:2014a}. The equivalence of the two definitions for Darboux integrable systems is proved in Theorem 5.3 in \cite{anderson-fels:2014a}.

An important consequence of maximal compatibility is given by the following corollary (see Corollary 7.4 in \cite{anderson-fels:2014a}).

\begin{Corollary}\StTag{MaxCT2} Let   $\bfp\:  (\CalE, N) \to (\CalI, M)$ be an integrable extension of  Darboux integrable  systems which is maximally compatible. Then there exists a monomorphism  
\begin{equation}
\psi:\vess(\CalE)\to\vess(\CalI)
\EqTag{Mpsi}
\end{equation}
from the Vessiot algebra of $\CalE$ to that of $\CalI$. 
\end{Corollary}

\begin{Example}\StTag{Ex61} We demonstrate equation \EqRef{singWJV} in part {\bf [i]} of Theorem \StRef{DIext} for Example \StRef{ExB} and examine the condition of maximal compatibility. Starting from diagram \EqRef{ExB3} we see that
\begin{equation}
B=I_2/\Gamma_H =\spn\{ \beta^a \} =  \bfp_2^* I_2+ \spn \{\beta^2= dx_6-Fdx_2 \}.
\EqTag{PBEB}
\end{equation}
The singular Pfaffian systems for $\CalI_2$ are computed from equations \EqRef{CofEB} and \EqRef{SeEB} to be
\begin{equation}
\hV_2 =I_2+\spn\{\, dx_1-\frac{(x_1+x_2)G''}{F'+G'} dx_3,\, d x_3\, \}, \,  \ 
\cV_2 = I_2+\spn \{\, dx_2-\frac{(x_1+x_2)F''}{F'+G'} dx_4,\, d x_4 \, \}.
\EqTag{SI2EB}
\end{equation}
The singular Pfaffian systems $\hZ$ and $\cZ$ for $\CalB$ are determined from equation \EqRef{SEB2} to be
\begin{equation}
\hZ= B+\spn\{\, dx_1,\, dx_2\, \} \quad {\rm and} \quad \hZ= B+\spn\{\, dx_2,\, dx_4\, \}.
\EqTag{SBEB}
\end{equation}
By virtue of equations \EqRef{PBEB} and \EqRef{SI2EB} the singular systems \EqRef{SBEB} are easily seen to satisfy
$$
\hZ = \bfp^*_2 \hV_ + \spn\{\, dx_6-Fdx_2\, \} \quad {\rm and } \quad \cZ = \bfp^*_2 \cV_2 + \spn\{\, dx_6-Fdx_2\, \}.
$$
This demonstrates equation \EqRef{singWJV} in part {\bf [i]} of Theorem \StRef{DIext} with $J=\spn\{ dx_6-Fdx_2 \}$.

The extension $\CalB \to \CalI_2$ satisfies the conditions of Definition \StRef{MaxCT}. 
Condition {\bf [ii]} follows immediately from equations \EqRef{DINVSB} and \EqRef{IIB2}. To check
condition {\bf [i]}, we use
the map $\bfp_2$ in diagram \EqRef{ExB3} to find
$$
\ker(\bfp_{2*})= \spn\{ \, \partial_{x_6} - \partial _{x_7} \, \}.
$$
This, along with the formula for $\hZ^\infty$ and $\cZ^\infty$ in equation \EqRef{IIB2}, shows that condition {\bf [i]} holds.

On account of equations \EqRef{IIex2I2} and \EqRef{IIB2},
the extension $\bfp_1:(\CalB,N) \to (\CalI_1,N_1)$ does {\it not} satisfy condition {\bf [ii]} in Definition \StRef{MaxCT}  and is therefore not maximally compatible.
\end{Example}


\begin{Example} \StTag{ExD3}  We demonstrate Corollary \StRef{PBFI} using the B\"acklund transformation in Example \StRef{ExD}. The map $\bfp_1 :(\CalB,N) \to (\CalI_1,N_1)$ is determined by repeated total differentiation of the  formulas 
\begin{equation}
V^1 = \log(2W^1_yW^2_x),   
	\quad
	V^2 =  2W^2-2W^1+\log(\frac{W^1_y}{W^2_x})
\EqTag{p1Exs6}
\end{equation}
from equation \EqRef{55}.
The pullback using $\bfp_1$ (equation \EqRef{p1Exs6} and its derivatives) of the 
intermediate integrals of $\CalI_1$ (see Example \StRef{ExA1p})  produces the following independent intermediate integrals for $\CalB$
\begin{equation}
\begin{aligned}
\bfp_1^*(x) & = x \ , \quad
\bfp_1^* (V^2_x)   =W^2_x-2W^1_x-{\rm e}^{W^1}  - \frac{W^2_{xx}}{W^2_x} , \\
\bfp_1^* \left(V^1_{xx} -\frac{1}{2}(V^1_x)^2\right) & =
\frac{W^2_{xxx}}{W^2_x}+{\rm e}^{W^1}(\frac{W^2_{xx}}{W^2_x}-W^1_x-W^2_x)-\frac{3}{2}\left(\frac{W^2_{xx}}{W^2_x}\right)^2-\frac{1}{2}{\rm e}^{2W^1}-\frac{1}{2}(W^2_x)^2.
\end{aligned}
\EqTag{pbfi6}
\end{equation}
These, along with $D_x(\bfp_1^*(V^2_x))$, produce a complete set of independent first integrals for the singular system $\hZ$ for the rank 8 Pfaffian system $\CalB$ determined by equations \EqRef{53}. The functions in \EqRef{pbfi6} can be checked directly to be intermediate integrals  of the PDE system \EqRef{53} by showing that their $y$ total derivatives vanishes. 

For the extension $\bfp_1:(\CalB, N) \to (\CalI_1, N_1)$  we have $\rank(\hV^\infty_1)=\rank(\hZ^\infty)=4$. The extension is not maximally compatible since {\bf [ii]} in Definition \StRef{MaxCT}  is not satisfied.
\end{Example}

\begin{Example}\StTag{Ex63}  All the integrable extensions $\bfp_2:(\CalB,N) \to (\CalI_2,N_2)$ constructed in the examples of Section \StRef{Examples} have the property the $\Gamma_H \subset \Gamma_{G_2}$ and $\Gamma_{G_2}$ 
 is a diagonal action. Therefore, on account of Theorem 6.5 in \cite{anderson-fels:2014a},  these
 extensions are all maximally compatible.  The homomorphism in Corollary \StRef{MaxCT2}  is just the inclusion map induced by $\Gamma_H \subset \Gamma_{G_2}$.
\end{Example}

\begin{Example} In Example \StRef{ExC} we constructed a sequence of B\"acklund transformations which lead to diagram \EqRef{SeqBT}.  Again, because the actions used to construct  each of the integrable extensions $ \bfp_i :\CalB_i \to \CalG_i$ are diagonal, the extensions are maximally compatible. For the integrable extensions $\bar \bfp_i : \CalB_i \to \CalG_{i-1}$ in diagram \EqRef{SeqBT}, the left inequality in equation \EqRef{infbds} is an equality and these extensions are {\deffont minimally compatible}.
\end{Example}

\section{B\"acklund Transformations between Monge-Amp\`ere systems in the plane and the wave equation}\StTag{MA}


In Section 6 of \cite{anderson-fels:2012a} we showed how Theorem A can be used to construct
B\"acklund transformations with 1-dimensional fibres  between a Monge-Amp\`ere Darboux integrable system whose Vessiot algebra is not ${\mathfrak s \mathfrak o}(3,\real)$, and the Monge-Amp\`ere system for the wave equation. In Section \StRef{MAS} we now show
the converse of this is true by showing that the maximal compatibility conditions in Definition \StRef{MaxCT} are satisfied for such B\"acklund transformations. Therefore we find that all B\"acklund transformations between Monge-Amp\`ere Darboux integrable  systems and the wave equation having 1-dimensional fibres arise through the group quotient process in Theorem A.  See Corollary 9.3 in \cite{anderson-fels:2014a}. 
  
  As our final application we show that if a Darboux integrable Monge-Amp\`ere system has Vessiot algebra  ${\mathfrak s \mathfrak o}(3,\real)$, then no B\"acklund transformation with the wave equation having 1-dimensional fibres exists. The reason for the non-existence of such a transformation hinges on Corollary \StRef{MaxCT2} and the simple fact that  ${\mathfrak s \mathfrak o}(3,\real)$ has no 2-dimensional subalgebras. A specific example of a partial differential equation with  ${\mathfrak s \mathfrak o}(3,\real)$ Vessiot algebra is given.

\subsection{B\"acklund Transformations for Monge-Amp\`ere Systems}\StTag{MAS}

Let $\CalI_2$ be a hyperbolic Monge-Amp\`ere system on a $5$-manifold $N_2$ (see \cite{bryant-griffiths-hsu:1995a} and \cite{Clelland-Ivey:2009a}). Then $\CalI_2$ is an $s=1$ hyperbolic EDS  (see Section \StRef{DefDI}).  The prolongation $\CalI_2^{[1]}$ of $\CalI_2$ to the space of regular 2-dimensional integral elements of $\CalI_2$ is an $s=3$ hyperbolic system \cite{bryant-griffiths-hsu:1995a}. We denote the singular Pfaffian systems of $\CalI_2$ by $\hV_2$ and $\cV_2$,  and let $\hV_{2,[1]}$ and $\cV_{2,[1]}$ denote the singular Pfaffian systems for $\CalI^{[1]}_2$. 

In this section we will assume that $\CalI_2^{[1]}$ is Darboux integrable but not Monge integrable \cite{Clelland-Ivey:2009a}. This is equivalent to  the following rank hypotheses 
\begin{equation}
 \rank( \hV_2^\infty ) = 1  , \quad \rank ( \cV_2^\infty )= 1, \quad {\text{and}}\quad 
 	\rank (\hV^{\infty}_{2,[1]}) = 2 , \quad  \rank (\cV^{ \infty}_{2,[1]} ) = 2
\EqTag{RV1}
\end{equation}
on the singular systems.
Equation \EqRef{VAdim}  shows  that  the dimension of the Vessiot algebra $\vess(\CalI_2^{[1]})$ for the Darboux integrable system $\CalI^{[1]}_2$ is 3.

Next let $\CalI_1$ be the standard Monge-Amp\`ere hyperbolic system on the 5-manifold $N_1$ for the wave equation $u_{xy}=0$.  Then $\CalI_1$ is Darboux integrable (without prolongation) 
	and the singular Pfaffian systems $\hV_1$ and $\cV_1$ satisfy
\begin{equation}
	  \rank(\hV^\infty_1) = 2  \quad\text{and}\quad  \rank(\cV^\infty_1) = 2 .
\EqTag{VdimWave}
\end{equation}
	Equations \EqRef{VAdim} and \EqRef{VdimWave} show that the Vessiot algebra  for the Darboux integrable system $\CalI_1$ (and hence the wave equation) has dimension 1.

Suppose that $\CalB$ is a B\"acklund transformation relating $\CalI_1$ and $\CalI_2$
where the fibres of $\bfp_1:N \to N_1$ and $\bfp_2:N \to N_2$ each have dimension 1. {\it Our goal
in this section is to prove that $\CalB$ is Darboux integrable and that the Vessiot algebra $\vess(\CalB)$ can be identified with a subalgebra of the Vessiot algebra $\vess(\CalI^{[1]}_2)$ of the Darboux integrable system $\CalI^{[1]}_2$.}  The key to proving this is to first show that the prolongation $\CalB^{[1]}$ is a maximally compatible extension of $\CalI_2^{[1]}$. We then apply Corollary \StRef{MaxCT2}.

We begin by proving that the B\"acklund transformation $\CalB$ has the following properties.

\begin{Lemma} \StTag{UMA}  	Let $\bfp_a:(\CalB, N)\to (\CalI_a,N_a)$ be a B\"acklund transformation, with 1-dimensional fibers, between  a hyperbolic 
	Monge-Amp\`ere system $(\CalI_2, N_2)$ satisfying  \EqRef{RV1}
	 and the Monge-Amp\`ere  system $(\CalI_1, N_1)$ for the wave equation. 
	
\smallskip
\noindent
{\bf [i]} The differential system $\CalB$ is an $s=2$ hyperbolic system with  singular Pfaffian systems $\hW$ and $\cW$ which are unique up to interchange.

\smallskip
\noindent
{\bf [ii]} The spaces of first integrals of$\hW$ and $\cW$ satisfy $\hW^\infty = \bfp_1^*(\hV_1^\infty)$ and $\cW^\infty = \bfp_1^*(\cV_1^\infty)$.

\smallskip
\noindent
{\bf [iii]} The differential system $\CalB$ is Darboux integrable.

\smallskip
\noindent
{\bf [iv]} The Vessiot algebra $\vess(\CalB)$ is  2-dimensional.

\end{Lemma}

\begin{proof}[Proof of  Lemma  \StRef{UMA}]

\noindent
{\bf Part [i]}: This follows immediately from part {\bf [i]} in Theorem \StRef{DIext} above. Indeed, an integral extension with 1-dimensional fibre  of a decomposable system of type $[2,2]$ on a 5-manifold (or an $s=1$ hyperbolic system) is a type [2,2] system on a 6-manifold and hence an $s=2$ hyperbolic system.   This conclusion applies whether we view $\CalB$ as an integrable extension of $\CalI_1$ or $\CalI_2$.  We remark that for any hyperbolic system of class $s$ the singular Pfaffian systems $\{ \hW, \cW\} $ are unique up to an interchange.

\smallskip
\noindent
{\bf Part [ii]}:  According to Theorem \StRef{DIext} the systems  $\{J_a\oplus \bfp_a^*\hV_a, J_a\oplus \bfp_a^* \cV_a\}$,  where the $J_a$ are admissible for $a=1,2$, are both singular systems of $\CalB$. By uniqueness we may assume that equation \EqRef{singWJV} in Theorem \StRef{DIext}  holds simultaneously for $a=1,2$. In particular we have $\bfp_a^*\hV_a\subset \hW$, $\bfp_a^* \cV_a\subset \cW$ and $\hV_a = \hW/\bfp_a$ and $\cV_a = \cW/\bfp_a$ for $a=1,2$.

We apply equation \EqRef{infbds} to the integrable extensions $\bfp_a:\hW\to \hV_a$ and $\bfp_a:\cW\to \cV_a$. With $a=1$ we have $ 2 \leq \rank \hW^\infty \leq 3$ while for $a=2$ equation \EqRef{infbds}  implies $1 \leq \rank \hW^\infty \leq 2$. Therefore $\rank \hW^\infty =2$ and $\hW^\infty=\bfp_1^*(\hV_1^\infty)$. A similar argument implies $\rank \cW^\infty= 2$ and $\cW^\infty=\bfp_1^*(\cV_1^\infty)$.   

\smallskip
\noindent
{\bf Part [iii]}:  We first note that $(B^1)^\infty \subset \hW^{\infty}\cap \cW^\infty$. By part {\bf [ii]} we deduce that
$$
\hW^{\infty}\cap \cW^\infty=\bfp_1^*( \hV_1^\infty ) \cap \bfp_1^*(\cV_1^\infty) = \bfp_1^*(\hV_1^\infty\cap \cV_1^\infty)=0,
$$
since $ \hV_1^\infty\cap \cV_1^\infty =0 $. Therefore $(B^1)^{\infty}=0$, and we  conclude from part {\bf [ii]} of Theorem \StRef{DIext}  that  $\CalB$ is Darboux integrable. 

\smallskip
\noindent
{\bf Part [iv]}:
Finally, the rank condition $\rank (\hW^\infty)=\rank (\cW^\infty)=2$ implies, by equation \EqRef{VAdim}, that $\dim(\vess(\CalB)) = 2$. 	
\end{proof}

\begin{Example} Lemma \StRef{UMA} is easily confirmed directly for the B\"acklund transformations $(\CalB,N)$  between the system $(\CalI_1,N_1)$ representing the wave equation and $(\CalI_i,N_i)$, $i=2,3,4$ given in equation \EqRef{ExA3}. The
structure equations for the rank 2 Pfaffian system $\CalB$ on the 6 manifold $N$ are given in equation \EqRef{StructB1} 
and so $\CalB$ satisfies {\bf [i]} in Lemma \StRef{UMA}. See \cite{bryant-griffiths-hsu:1995a} or Section \StRef{DefDI}. 

The space of intermediate integrals for $\CalB$ are given in equations \EqRef{B1infh} and \EqRef{B1infc}, while those for $\CalI_1$ are given by $\hV_1^\infty=\spn\{ dx , du_{1x} \}, \ \cV_1^\infty= \spn\{ dy , du_{1y}\}$. The map $\bfp_1:N \to N_1$, given in equation \EqRef{p1map}, leads immediately to 
$$
\bfp_1^*(\spn\{ dx , du_{1x} \}) = \spn\{ dx, dV_x +e^V dV\}\quad {\rm  and}\quad \bfp_1^*(\spn\{ dy, du_{1y}\}) = \spn\{dy, dW_y+e^W dW\}
$$ 
which, by equations \EqRef{B1infh} and \EqRef{B1infc}, demonstrates part {\bf [ii]} in Lemma \StRef{UMA}.

The direct verification of Darboux integrability (Definition \StRef{Intro2}) of $\CalB$ follows from equations \EqRef{BH1}, \EqRef{B1infh} and \EqRef{B1infc}. This demonstrates part {\bf [iii]} in Lemma \StRef{UMA}.

Lastly, part  {\bf [iv]} in Lemma \StRef{UMA} follows from  the initial discussion in Section \StRef{CVA}.  The Vessiot algebra of $\CalB$ is noted to be isomorphic to the two dimensional algebra $\Gamma_H$ defined in equation \EqRef{ExA2}. \end{Example}

We now consider the prolongation of $\CalI_1$, $\CalI_2$ and $\CalB$. Let $N^{[1]}$ be the 8-dimensional manifold consisting of 2-dimensional regular integral elements of $\CalB$ and let $\CalB^{[1]}$ be the class $s=4$ hyperbolic system for the prolongation of $\CalB$ \cite{bryant-griffiths-hsu:1995a}.  Likewise, let $N_a^{[1]}$ be the 7-dimensional manifolds of regular 2-dimensional integral elements of $\CalI_a$ and let $\CalI^{[1]}_a$ be the $s=3$ hyperbolic systems which are the prolongations of $\CalI_a$. Let $\pi_a: N_a^{[1]} \to N_a$ and $\pi:N^{[1]}\to N$ be the projection maps from the prolongation spaces.   Finally, since $\CalB$ is an integrable extension of $\CalI_a$ the maps $\bfp_a$ canonically define projection maps on the space of regular integral elements which we denote by  $ \bfp_a^{[1]}: N^{[1]}\to N_a^{[1]}$. A direct application of property {\bf IE\, [iv]} in Section 2.1 of \cite{anderson-fels:2014a} produces the following theorem which summarizes the effect of prolongation.

\begin{Theorem} The top portion of the commutative diagram
\begin{equation}
\begin{gathered}
\beginDC{\commdiag}[3]
\Obj(0, 32)[B]{$(\CalB^{[1]},N^{[1]})$}
\Obj(0, 12)[bB]{$(\CalB,N)$}
\Obj(-24, 20)[I1]{$(\CalI_1^{[1]},N_1^{[1]})$}
\Obj(24, 20)[I2]{$(\CalI_2^{[1]},N_2^{[1]})$}
\Obj(-24, 0)[bI1]{$(\CalI_1,N_1)$}
\Obj(24, 0)[bI2]{$ (\CalI_2,N_2)$}
\mor{B}{bB}{\lower 20pt \hbox{$\pi$} }[\atright, \solidarrow]
\mor{I1}{bI1}{$\pi_1$}[\atright, \solidarrow]
\mor{I2}{bI2}{$\pi_2$}[\atleft, \solidarrow]
\mor{B}{I1}{$\bfp_1^{[1]}$}[\atright, \solidarrow]
\mor{B}{I2}{$\bfp_2^{[1]}$}[\atleft, \solidarrow]
\mor{bB}{bI1}{$\bfp_1$}[\atright, \solidarrow]
\mor{bB}{bI2}{$\bfp_2$}[\atleft, \solidarrow]
\enddc
\end{gathered}
\EqTag{57}
\end{equation}
\noindent 
defines a B\"acklund transformation $\CalB^{[1]}$ with 1-dimensional fibres between the Darboux integrable systems $\CalI_1^{[1]}$ and $\CalI_2^{[1]}$. 
\end{Theorem}

By the hypothesis in \EqRef{RV1} we have that $\CalI^{[1]}_2$ is Darboux integrable where $\rank \hV^\infty_{2,[1]}=2$ and $\rank \cV^\infty_{2,[1]}=2$.  The number of independent intermediate integrals for $\CalI_1^{[1]}$, the prolongation of the Monge-Amp\`ere form of the wave equation, is 6 since
\begin{equation}
\rank \hV^\infty_{1,[1]} = 3 \qquad {\rm and} \qquad \rank \cV^\infty_{1,[1]} = 3
\EqTag{prow1}
\end{equation}
where $\hV_{1,[1]}$ and $\cV_{1,[1]}$ are the singular Pfaffian systems for $\CalI_1^{[1]}$.
Following the same arguments used to prove Lemma \StRef{UMA},
or by utilizing Theorem 7.5 in \cite{anderson-fels:2014a} we may use diagram \EqRef{57} to conclude the following.

\begin{Lemma} \StTag{DIBB} The differential system $\CalB^{[1]}$ has the following properties.

\smallskip
\noindent
{\bf [i]}  The system $\CalB^{[1]}$ is a Darboux integrable $s=4$ hyperbolic Pfaffian system.

\smallskip
\noindent
{\bf [ii]} The singular systems $\hW_{[1]}$ and $\cW_{[1]}$ of $\CalB^{[1]}$ satisfy $\rank \hW^\infty_{[1]} = 3$, $\rank \cW^\infty_{[1]} = 3$, and
$$
\hW^\infty_{[1]}=\bfp_1^{[1]*}( \hV_{1,[1]}^\infty)\  \quad {\rm and } \quad \cW^\infty_{[1]}=\bfp_1^{[1]*}( \cV_{1,[1]}^\infty).
$$

\smallskip
\noindent
{\bf [iii]} The Vessiot algebra $\vess(\CalB^{[1]})$ is 2-dimensional and is isomorphic to $\vess(\CalB)$. 
\end{Lemma} 

\begin{proof} Parts {\bf [i]} and {\bf [ii]} follow using arguments similar to those used in demonstrating ${\bf [i]}$ and ${\bf [ii]}$ in Lemma \StRef{UMA}. For example, part ${\bf [ii]}$ in Lemma \StRef{DIBB}
is proved using \EqRef{prow1} and the second two equalities in \EqRef{RV1} in same way that 
the proof of ${\bf [ii]}$ in Lemma \StRef{UMA} uses equation \EqRef{VdimWave} and the first two equalities in \EqRef{RV1}. Part {\bf [iii]} follows from Theorem 7.5 in \cite{anderson-fels:2014a}. 
\end{proof}


Our goal is now to show that the integrable extension $\CalB^{[1]}\to \CalI_2^{[1]}$ is maximally compatible on the set $\widetilde N \subset N^{[1]}$ defined by
\begin{equation}
\widetilde N= \{\ p \in N^{[1]} \ | \ \rank (\hW^\infty_{[1]})_{\bfp^{[1]}_{2, \semibasic}} = 2 \ {\rm and} \  \rank (\cW^\infty_{[1]})_{\bfp^{[1]}_{2, \semibasic}} =2\  \},
\EqTag{SBF}
\end{equation}
where the subscript in $\bfp^{[1]}_{2, \semibasic}$ denotes the $\bfp^{[1]}_2$ semi-basic forms. 

\begin{Lemma} \StTag{N0} The set $\widetilde N \subset N^{[1]}$ is an open dense subset.
\end{Lemma}
\begin{proof}  First, since $\rank \hW^\infty_{[1]} = 3$ and $\rank \ker \, (\bfp_{2, *}^{[1]})=1$,  we have
\begin{equation}
2\leq  \rank \, (\hW^\infty_{[1]})_{\bfp^{[1]}_{2, \semibasic}} \leq 3,
\EqTag{SBID}
\end{equation}
and the set   $\widehat A= \{\ p \in N^{[1]} \ | \ \rank (\hW^\infty_{[1]})_{\bfp^{[1]}_{2, \semibasic}} = 2 \, \} $ is open. The identity analogous to \EqRef{SBID} holds for $(\cW^\infty_{[1]})_{\bfp^{[1]}_{2, \semibasic}}$ from which we conclude that $\widetilde N$ is an open set. We now show that the complement $\widehat A^c$ has no interior (and likewise for the similarly defined set $\check A^c$) and hence $ \widetilde N$ is dense.  

Suppose that $U\subset \widehat A^c $ is a non-empty open set. On the submersion $\bfp_2^{[1]} :U \to \bfp_2^{[1]}( U)$ we have
$$
\hV_{2,[1]} = \hW_{[1]}/ \bfp_2^{[1]}.
$$
On $\widehat A^c$ we have, by hypothesis, the constant rank condition $ \rank (\hW^\infty_{[1]})_{\bfp^{[1]}_{2, \semibasic}} =3$ and
so, by Theorem 3.3 and Lemma 3.4 in \cite{anderson-fels:2012a},
\begin{equation}
\hV_{2,[1]} ^\infty = \hW^\infty_{[1]}/ \bfp_2^{[1]} .
\EqTag{PB2}
\end{equation}
This is a contradiction because the left-hand side of equation \EqRef{PB2} has rank 2 (equation \EqRef{RV1}), while the right-hand side has rank 3. Therefore, no such open set $U$ exists. 
\end{proof}


\begin{Theorem}  \StTag{MC} The integrable extension $\bfp_2^{[1]}:(\CalB^{[1]},\widetilde N) \to ( \CalI_2^{[1]},\widetilde N_2 )$, where $\widetilde N_2=\bfp_2^{[1]}(\widetilde N) \subset N_2^{[1]} $,  satisfies conditions {\bf [i]} and {\bf [ii]} of Definition \StRef{MaxCT}  and is therefore maximally compatible.
\end{Theorem}

\begin{proof}[Proof of  Theorem  \StRef{MC}]   Part {\bf [ii]} in Lemma \StRef{DIBB}
and equation \EqRef{RV1} give the equalities
$$
\rank (\hW^\infty_{[1]})=3= \rank (\hV_{2,[1]}^\infty)+1 \quad {\rm and} \quad  \rank (\cW^\infty_{[1]})=3=\rank (\cV_{2,[1]}^\infty) +1. 
$$
This shows condition {\bf [ii]} in Definition \StRef{MaxCT} holds. 

We now check condition {\bf [i]} or more precisely its dual which, in the present context, is
\begin{equation}
T^*{\widetilde N}_{ \bfp_2^{[1]}, \semibasic}+ \hW^\infty_{[1]} = T^* {\widetilde N}.
\EqTag{PMC}
\end{equation}
We find, by using equation \EqRef{SBF}, that
\begin{equation*}
\begin{aligned}
\rank \left( T^*{\widetilde N}_{ \bfp_2^{[1]}, \semibasic}+ \hW^\infty_{[1]}\right) & = 
\rank \left( T^*{\widetilde N}_{ \bfp_2^{[1]}, \semibasic}\right)+\rank \left( \hW^\infty_{[1]}\right) - \rank\left( T^*{\widetilde N}_{ \bfp_2^{[1]}, \semibasic}\cap \hW^\infty_{[1]}\right)\\ & = 7+3-2=8.
\end{aligned}
\end{equation*}
This implies, by dimensional reasons, that equation \EqRef{PMC} holds.
\end{proof}

We emphasize that Theorem  \StRef{MC} is particular to the extension $\bfp_2^{[1]}:\CalB^{[1]} \to \CalI_2^{[1]}$   and that $\CalB^{[1]}$ is 
 {\it  not} a maximally compatible extension of  the  wave equation $\CalI_1^{[1]}$.  
This follows immediately from  Lemma \StRef{DIBB} {\bf [ii]} which shows that equation {\bf [ii]}  in Definition \StRef{MaxCT} does not hold (with $\rank (J) =1$).

Our next result is a simple consequence of Theorem \StRef{MC} and Corollary \StRef{MaxCT2}, applied to the integrable extension $\bfp_2^{[1]}:(\CalB^{[1]},\widetilde N) \to ( \CalI_2^{[1]},\widetilde N_2 )$.

\begin{Corollary}\StTag{LMQ} There is a Lie algebra monomorphism  from the algebra $\vess(\CalB)$ to the algebra $\vess(\CalI_2^{[1]})$.
\end{Corollary}

The B\"acklund transformation $\CalB^{[1]}$  can therefore be constructed locally as a group quotient using the monorphism from Corollary \StRef{LMQ} in accordance with Theorem 8.2 (or Corollary 9.3 in \cite{anderson-fels:2014a}). The integrable extension $\CalB \to \CalI_2$ can also be constructed using group quotients by the deprolongation of the group quotient construction which produces $\CalB^{[1]} \to \CalI_2^{[1]}$ (see \cite{anderson-fels:2012a}).

Our final result  formalizes our remark in the Introduction of \cite{anderson-fels:2014a} on the non-existence of B\"acklund transformations.

\begin{Theorem}\StTag{NoTom}  Let $\CalI_2$ be a hyperbolic Monge-Amp\`ere system satisfying  \EqRef{RV1}.
	If the Vessiot algebra $\vess(\CalI_2^{[1]})$ of $\CalI_2^{[1]}$ is ${\mathfrak s \mathfrak o}(3,\real)$,
	then $\CalI_2$ does not admit a B\"acklund transformation (having  1-dimensional fibers)
	to the wave equation.  The Vessiot algebra for the equation
\begin{equation}
u_{xy}= \frac{ \sqrt{1-u_x^2} \sqrt{1-u_y^2} }{\sin u }
\EqTag{SO3e}
\end{equation}
	is ${\mathfrak s \mathfrak o}(3,\real)$  and therefore  \EqRef{SO3e}
	does not admit such a  B\"acklund transformation to the wave equation.
\end{Theorem}

\begin{proof}[Proof of Theorem  \StRef{NoTom}.] 
By Corollary \StRef{LMQ} and part {\bf [iii]} in Lemma \StRef{DIBB}, 
if such a B\"acklund transformation existed, then  then there would be a monomorphism from the 
2-dimensional Lie algebra $\vess(\CalB^{[1]})$ to $\vess(\CalI^{[1]}_2)={\mathfrak s \mathfrak o}(3,\real)$. However,
the Lie algebra ${\mathfrak s \mathfrak o}(3,\real)$ has no two dimensional subalgebras and therefore no such B\"acklund transformation exists.

	To finish the proof of the theorem we show equation \EqRef{SO3e} has ${\mathfrak s \mathfrak o}(3,\real)$ 
	as its Vessiot algebra. Let $(x,y,u,u_x,u_y, u_{xx},u_{yy})$
	be the standard jet coordinates on the open set $M^{[1]} \subset \real^7$, where $|u_x| <1 $ and $|u_y| <1$.
	The intermediate integrals for equation \EqRef{SO3e}
	are $x,y$ and
\begin{equation}
	\hxi = \frac{u_{xx}}{\sqrt{1-u_x^2}}-\sqrt{1-u_x^2} \cot u ,
	\quad \cxi =\frac{u_{yy}}{\sqrt{1-u_y^2}}-\sqrt{1-u_y^2} \cot u.
\end{equation}
	In terms of  coordinates $(x,y,u,\alpha=\arcsin u_x,\beta=\arcsin u_y,\hxi,\cxi)$, where 
	$- \pi/2 < \alpha ,  \beta < \pi/2$  the differential system $\CalI^{[1]}$ is the Pfaffian system given by the 1-forms
\begin{equation}
\begin{aligned}
\theta^1 & =\hxi dx-(\cxi \cos u-\sin u \cos \beta ) dy-d\alpha+\cos u\, d\beta, \\
\theta^2 &= (\cxi \sin u \sin \alpha +\sin \alpha \cos u \cos \beta-\cos \alpha \sin \beta) dy+\cos \alpha\, du-\sin u \sin \alpha\, d\beta,\\
\theta^3 &= -dx-(\sin \beta \sin \alpha +\cxi \cos \alpha \sin u +\cos \alpha \cos u \cos \beta)dy+\sin \alpha\, du+\cos \alpha \sin u\, d\beta.
\end{aligned}
\end{equation}
The structure equations for the co-frame $(\theta^1,\theta^2,\theta^3, dx, d\hxi, dy, d\cxi )$  are
$$
\begin{aligned}
	d\theta^1 & =  -\theta^2 \wedge \theta^3 -\theta^2 \wedge dx - dx \wedge  d\hxi + \cos u \, dy \wedge d \cxi, \\
	d \theta^2 &=  \theta^1 \wedge \theta^3+\theta^1 \wedge dx+\hxi\, \theta^3\wedge dx - \sin \alpha  \sin u \, dy \wedge d\cxi, \\
	d \theta^3 &= -\theta^1\wedge\theta^2-\hxi\, \theta^2\wedge dx+ \cos\alpha \sin u \, dy \wedge d \cxi ,
\end{aligned}
$$
and one sees, by inspection, that this is a 4-adapted coframe (see \cite{anderson-fels:2014a} or \cite{anderson-fels-vassiliou:2009a}). The $\theta^i \wedge \theta^j $ terms in the structure equations therefore imply that the Vessiot algebra for \EqRef{SO3e} is ${\mathfrak s \mathfrak o}(3,\real)$.
\end{proof}

\appendix


\end{document}